\documentclass[a4paper,11pt]{report}
\usepackage[utf8]{inputenc}
\usepackage[square,numbers]{natbib}
\usepackage{tikz}
\usepackage{amssymb,amsfonts,amscd,mathrsfs,eurosym}
\usepackage{amsmath}
\usepackage{aecompl}
\usepackage{textcomp} 
\usepackage{stmaryrd}
\usepackage{enumerate}
\usepackage{eucal}
\usepackage{enumitem}
\usepackage{bm}
\usepackage{graphicx}
\usepackage{epstopdf} 
\usepackage{color} 
\usepackage[T1]{fontenc}
\usepackage{exscale,relsize}
\usepackage{titlesec}
\usepackage{xypic}
\usepackage[colorlinks,linkcolor=black,urlcolor=black]{hyperref}
\hypersetup{
     colorlinks   = true,
     citecolor    = black}
\usepackage{tikz}
\usepackage{pgfplots}
\newcommand{\R}{\mathbb R}

\newcommand{\C}{\mathbb C}

\newcommand{\F}{\mathbf {F}}
\newcommand{\G}{\mathbf G}

\newcommand{\V}{\mathbf V}

\newcommand{\E}{\mathbf E}

\newtheorem{numbered}{}[section]
\newtheorem{theo}[numbered]{Theorem}
\newtheorem{prop}[numbered]{Proposition}

\newtheorem{defi}[numbered]{Definition}

\newtheorem{lem}[numbered]{Lemma}

\newtheorem{rem}[numbered]{Remark}

\newtheorem{cor}[numbered]{Corollary}

\newtheorem{ex}[numbered]{Example}

\numberwithin{equation}{section}

\newcommand{\theoref}[1]{Theorem~\ref{#1}}
\newcommand{\chapref}[1]{Chapter~\ref{#1}}

\newcommand{\defiref}[1]{Definition~\ref{#1}}

\newtheorem{exple}{\textbf{Example}}[section]

\newenvironment{proof}
	{\textit{\textbf{Proof.}}}
	{\hfill $\blacksquare$\vskip 8pt}

\makeatletter
\renewcommand{\@chapapp}{}
\makeatother
\titleformat{\chapter}[hang]{\bf\huge}{\thechapter}{2pc}{}
\newlength{\plarg}
\newlength{\glarg}
\DeclareMathOperator{\rank}{\mathrm{rank}}

\DeclareMathOperator{\rot}{\mathrm{rot}}
\DeclareMathOperator{\dive}{\mathrm{div}}
\DeclareMathOperator{\grad}{\mathrm{grad}}
\DeclareMathOperator{\spa}{\mathrm{span}}

\begin{document}

\pagenumbering{roman}
\pagestyle{plain}


\begin{titlepage}
\noindent

\begin{center}
\begin{minipage}{\plarg}
 \centering
{\bfseries \Large{Lectures on Finslerian geometry }}\\
\vspace{1.5cm}
{\bfseries  Oumar Wone}\\ \vspace{1cm}
\end{minipage}
\end{center}
\end{titlepage}

\renewcommand{\baselinestretch}{1.15}\small\normalsize


\newpage
\pagenumbering{arabic}
\chapter*{Introduction}
One can trace back the birth of Finslerian geometry to the thesis of Riemann in $1854$. There he introduced the general concept of pseudo-Riemannian manifolds and thereby generalizing previous work by Gauss in the case of surfaces. Pseudo-Riemannian manifolds are manifolds $\mathbf{M}$ with a smoothly varying symmetric bilinear "scalar" product (or metric) on the tangent space $T_x\mathbf{M}$ at each point. Each such metric is given in local coordinates by an expression of the form $\sum a_{ij}dx_idx_j$, homogeneous of degree $2$ in the $dx_i$s. Pseudo-riemannian geometry is the study of pseudo-Riemannian manifolds up to appropriate isomorphisms. After Riemann in $1918$ Finsler, a student of Caratheodory, introduced the notion of Finslerian metric on the tangent space at each point of the manifold $\mathbf{M}$. These are given by expression of the form $F(x,dx)$ in local coordinates, homogeneous of degree $1$ in $dx$ and smoothly varying. He used his concept principally to study the calculus of variations, ie the problem of the extremum of an expression of the form $\int F(x,y,y^\prime)$ for $F$ and $y$ sufficiently differentiable maps belonging to some given function spaces. After the work of Caratheodory and his school, the study of Finslerian manifolds was given a new impetus in $1933$ by E. Cartan who, using the "absolute differential calculus" of Ricci, Bianchi, Levi-Civita linked the two concepts of Finslerian and pseudo-Riemannian manifolds, and hence the corresponding geometries. Finally one can say the modern presentation of Finslerian geometry is due to S. S. Chern and his school.

\tableofcontents
\newpage
\chapter{Calculus of variations}
\label{ch1}
\noindent
\section{Review of differential calculus on Banach spaces}
Our main reference in this chapter is \citep{hcartan}. Let $(\E,\|.\|)$ be a real Banach space. Let $I=\left[a,b\right]\subset\R$. Consider
$$\V:=\{\varphi:I\to \E, \varphi\in \mathscr{C}^1(I,\E)\}$$
equipped with the norm
$$\varphi\to\|\varphi\|=\|\varphi\|_\infty+\|\varphi^\prime\|_\infty$$
with $\displaystyle\|\varphi\|_\infty:=\sup_{t\in I}\|\varphi(t)\|_\E$ which exists because $t\in I\to \|\varphi(t)\|_\E$ is continuous and $I$ is compact. We adopt a similar definition for $\|\varphi^\prime\|_\infty$.
\begin{rem}
What is really $\varphi^\prime$? Recall that if $U$ is an open set in a Banach space $\F$ and $h:U\to \G$ with $\G$ a Banach space, then $h$ is said to be differentiable at $a\in U$ if there exists an $\R$-linear continuous application $L_a:\F\to \G$ such that
$$h(a+u)=h(a)+L_a(u)+\varepsilon(u)$$
with $\displaystyle\lim_{\|u\|\to 0}\dfrac{\varepsilon(u)}{\|u\|}=0$. One often denotes such an $L_a$ by the symbol $h^\prime(a)$. One has $h^\prime(a)\in\mathcal{L}(\F,\G)$, the space of continuous $\R$-linear maps from $\F$ to $\G$. In our situation we have $\varphi:I\to \E$ so our $\varphi^\prime(a)\in\mathcal{L}(\R,\E)\cong \E$.
\end{rem}
Hence $(\V,\|.\|)$ is a normed $\R$-vector space. We have more
\begin{prop}
$(\V,\|.\|)$ is a normed Banach vector space.
\end{prop}
\begin{proof}
Let $(\varphi_n)$ be a Cauchy sequence in $V$. Then $\exists$ N: $\forall n,m\geqslant N$
$$\|\varphi_n-\varphi_m\|_\infty\leqslant\|\varphi_n-\varphi_m\|.$$
Thus $(\varphi_n)$ is a Cauchy sequence with respect to the sup-norm and similarly for $(\varphi_n^\prime)$. Hence as $(\mathcal{C}(I),\|\|_\infty)$ ($\mathcal{C}(I)$ is the space of continuous maps from $I\to \mathbf{E}$) is a Banach space $(\varphi_n)$ converges uniformly to $\varphi\in\mathcal{C}(I)$ and similarly $(\varphi_n^\prime)$ converges uniformly to $\psi\in\mathcal{C}(I,\mathcal{L}(\R,\E)\cong \E)$. This means that $\varphi$ is differentiable and $\varphi^\prime=\psi$. Therefore $\phi:I\to \E\in \mathscr{C}^1(I,\E)$.  
\end{proof}

Let $(\E,\|\|)$ be a Banach space. Let $U\subset\R\times \E\times \E$ be an open set.
Consider
\begin{equation}
\begin{split}
F&:U\to R\\
(t,x,y)&\mapsto F(t,x,y)
\end{split}
\end{equation}
of class $\mathscr{C}^k$, $k³1$. Take $\varphi\in \mathscr{C}^1(I,\E)$ such that for all $t\in I$, $(t,\varphi(t),\varphi^\prime(t))\in U$. We define the functional
$$f(\varphi)=\int_a^bF(t,\varphi(t),\varphi^\prime(t))dt$$
and also $\Omega=\{\varphi\in \mathbf{V}, (t,\varphi(t),\varphi^\prime(t))\in U, \forall t\in I\}$. As a consequence we have an operator
\begin{equation}
\begin{split}
f&:\Omega\to\R\\
&\phi\mapsto f(\varphi)=\int_a^bF(t,\varphi(t),\varphi^\prime(t))dt.
\end{split}
\end{equation}
\begin{prop}
$\Omega$ is an open set in V.
\end{prop}
\begin{proof}
Take $\varphi_0\in\Omega$. It induces a map
\begin{equation}
\begin{split}
&I\to U\subset \R\times \E\times \E\\
&t\mapsto (t,\varphi_0(t),\varphi_0^\prime(t)).
\end{split}
\end{equation}
Consider $K$ the compact of $U$ which is the image of $I$ by the latter map. From general topology (because $U$ is open) one knows that there exists $\rho>0$ such that if $(t,x,y)\in I\times \E\times \E$ verifies
$$\begin{cases}
   \|x-\varphi_0(t)\|_\E<\rho\\
    \|y-\varphi^\prime_0(t)\|_\E<\rho
\end{cases} $$ 
then $(t,x,y)\in U$. Choose $\varphi\in \mathbf{V}$ such that $\|\varphi-\varphi_0\|<\rho$ (i.e. $\varphi\in B(\varphi_0,\rho)$ the ball of center $\varphi_0$ and of radius $\rho$). Then for each $t\in I$ 
$$\begin{cases}
   \|\varphi(t)-\varphi_0(t)\|_\E<\rho\\
    \|\varphi^\prime(t)-\varphi^\prime_0(t)\|_\E<\rho
\end{cases}$$
hence $(t,\varphi(t),\varphi^\prime(t))\in U,\,\forall t\in I$. And this means $B(\varphi_0,\rho)\subset \Omega$ which is therefore open.
\end{proof}
\begin{prop}
Assume $F$ is $\mathscr{C}^k,\, k³1$ then $\forall u\in V$
$$f^\prime(\varphi)\cdot u=\int_a^b\dfrac{\partial F}{\partial x}(t,\varphi(t),\varphi^\prime(t))\cdot u(t)dt+\int_a^b\dfrac{\partial F}{\partial y}(t,\varphi(t),\varphi^\prime(t))\cdot u^\prime(t)dt.$$
\end{prop}
\begin{rem}
If $\varphi\in\Omega$ and if $f:\Omega\to \R$, $f^\prime(\varphi)\in\mathcal{L}(\V,\R)$. So if $u\in \V$ then $f^\prime(\varphi)(u)\in\R$. Moreover if $\varphi:I\to E$ then $\varphi(t)\in \mathbf{E}$ and $\varphi^\prime(t)\in\mathcal{L}(\R,\E)\cong \mathbf{E}$. It follows that
$$\begin{cases}
   &\dfrac{\partial F}{\partial x}(t,\varphi(t),\varphi^\prime(t))\in\mathcal{L}(\E,\R)\\\\
   &\dfrac{\partial F}{\partial y}(t,\varphi(t),\varphi^\prime(t))\in\mathcal{L}(\E,\R)
   \end{cases}$$
   so the expressions $\dfrac{\partial F}{\partial x}(t,\varphi(t),\varphi^\prime(t))\cdot u(t)$ and $\dfrac{\partial F}{\partial y}(t,\varphi(t),\varphi^\prime(t))\cdot u^\prime(t)$ are well-defined.
\end{rem}
We now begin the proof of the previous proposition; to this end we need the following classical lemma.
\begin{lem}
Let $\E,\, \F$ two Banach spaces and $I=\left[a,b\right]$. Consider a mapping 
$$\varphi:U\times \E\to \F,\,U\subset \E,\,U\,open.$$
For each $x\in U$ we define $\psi(x)=\int_a^b\varphi(x,t)dt$ then $\psi:U\to \F$ is continuous. Moreover if $\dfrac{\partial\varphi(x,t)}{\partial x}$ exists for each $(x,t)\in U\times I$ and as a map in $Map(U\times I,\mathcal{L}(\E,\F))$ is continuous then $\psi$ is of class $\mathscr{C}^1$ and
$$\psi^\prime(x)=\int^b_a\dfrac{\partial\varphi(x,t)}{\partial x}.$$ 
\end{lem}
\begin{proof}
Let $\varepsilon>0$ be given. For each $(x,t)\in U\times I$, $\exists\,\eta(x,t)>0$ such that $\|\varphi(x^\prime,t^\prime)-\varphi(x,t)\|<\dfrac{\varepsilon}{2}$ for
$$\begin{cases}
   &\|x-x^\prime\|<\eta(x,t)\\
   &|t^\prime-t|<\eta(x,t).
   \end{cases}$$
   In particular $\|\varphi(x,t^\prime)-\varphi(x,t)\|<\dfrac{\varepsilon}{2}$ for $|t^\prime-t|<\eta(x,t)$. Hence for $|t^\prime-t|<\eta(x,t)$ and $|x^\prime-x|<\eta(x,t)$ we have $\|\varphi(x^\prime,t^\prime)-\varphi(x,t)\|<\varepsilon$ (add and substract $\varphi(x,t)$ to the expression inside the norm and use the triangle inequality). Now the family of intervals $\{]-\eta(x,t),t+\eta(x,t)[,t\in I\}$ cover $I$ which is a compact. Hence by the Borel-Lebesgue lemma it admits a finite sub-covering by such open intervals $]t_i-\eta(x,t_i),t_i+\eta(x,t_i)[$, say. Then for any $t^\prime\in I$, $\exists t_i\in I:\,|t^\prime-t_i|<\eta(x,t_i)$ and then $\|\varphi(x^\prime,t^\prime)-\varphi(x,t)\|<\varepsilon$. By taking the integral over $I$ one gets the required continuity. 
   
   Now assume that $\dfrac{\partial \varphi}{\partial x}$ exists and is continuous as a map from $U\times I$ to $\mathcal{L}(\E,\F)$. By the previous calculation the map $\lambda:\,\lambda(x)=\int_a^b\dfrac{\partial \varphi}{\partial x}(x,t)dt$ is continuous. We are going to show that $\psi$ is differentiable and that $\psi^\prime(x)=\lambda(x)$. By continuity of the map $\dfrac{\partial \varphi}{\partial x}$ (by hypothesis), to each $\varepsilon>0$ we can associate $\eta_x>0$ such that for each $t\in I$ and $h\in\E$ such that $x+h\in U$ one has
   $$\|\dfrac{\partial \varphi}{\partial x}(x+h,t)-\dfrac{\partial \varphi}{\partial x}(x,t)\|²\varepsilon$$
   for $\|h\|<\eta_x$. By the multidimensional mean-value theorem
   $$\| \varphi(x+h,t)-\varphi(x,t)-\dfrac{\partial \varphi}{\partial x}(x,t)\|²\varepsilon\|h\|,\,\|h\|²\eta_x$$
   hence by integrating one has
  \begin{equation}
  \begin{split}
  \| \psi(x+h,t)-\psi(x,t)-\dfrac{\partial \psi}{\partial x}(x,t)\|&²\int_a^b\| \varphi(x+h,t)-\varphi(x,t)-\lambda(x)\cdot h\|dt\\
  &²\varepsilon\|h\|(b-a).
  \end{split}
  \end{equation}
  Hence $\psi$ is differentiable.
\end{proof}
\begin{cor}
If moreover $\dfrac{\partial^n\varphi}{\partial x^n}$ exists and is continuous from $U\times I\to \mathcal{L}_n(\E,\F)$ then $\psi$ is of class $\mathscr{C}^n$ and $$\psi^{(n)}=\int_a^b\dfrac{\partial^n\varphi}{\partial x^n}(x,t)dt.$$
\end{cor}
\begin{rem}
We have $f:U\subset \E\to \F$, $U\ni a$, an open set. Hence for all $a\in U$ $f^\prime(a)\in\mathcal{L}(\E,\F)$. Therefore we can define a map $f^\prime:U\to \mathcal{L}(\E,\F)$ in the obvious way. Similarly $f^{\prime\prime}(a)=(f^\prime)^\prime(a)\in\mathcal{L}(\E,\mathcal{L}(\E,\F))\cong\mathcal{L}(\E\times \E,\F)$. Thus inductively we set $\mathcal{L}_n(\E,\F):=\mathcal{L}(\E,\mathcal{L}_{n-1}(\E,\F))\cong\mathcal{L}(\E^n,\F)$
\end{rem}
Let us come back to the proof of the proposition
\begin{prop}
With the previous notations if $F$ is of class $\mathscr{C}^k,\,k³1$ then so is $f$ and for each $u$ in $\V$ one has 
$$f^\prime(\varphi)\cdot u=\int_a^b\dfrac{\partial F}{\partial x}(t,\varphi(t),\varphi^\prime(t))\cdot u(t)dt+\int_a^b\dfrac{\partial F}{\partial y}(t,\varphi(t),\varphi^\prime(t))\cdot u^\prime(t)dt.$$
\end{prop}
\begin{proof}
We first show that if $F$ is of class $\mathscr{C}^1$, then $f$ has derivative $f^\prime$ given by
$$f^\prime(\varphi)\cdot u=\int_a^b\dfrac{\partial F}{\partial x}(t,\varphi(t),\varphi^\prime(t))\cdot u(t)dt+\int_a^b\dfrac{\partial F}{\partial y}(t,\varphi(t),\varphi^\prime(t))\cdot u^\prime(t)dt.$$
This will imply that the map $\varphi\mapsto f^\prime(\varphi)$ from $\Omega$ to $\mathcal{L}(\V,\R)$ is continuous hence $f$ is of class $\mathscr{C}^1$. Introduce
$$\lambda:\Omega\times I\to \R,\,\lambda(\varphi,t)=F(t,\phi(t),\varphi^\prime(t)).$$
We then have
$$f(\varphi)=\int_a^b\lambda(\varphi,t)dt.$$ 
By the previous lemma if $\dfrac{\partial\lambda}{\partial\varphi}$ exists and is continuous as a function on $\Omega\times I$ then $f^\prime$ exists and is given by
$$f^\prime(\varphi)\cdot u=\int_a^b\dfrac{\partial\lambda}{\partial\varphi}(\varphi,t)\cdot udt.$$
Therefore we must show the existence of $\dfrac{\partial\lambda}{\partial\varphi}$ and compute it. Using the fact that $F$ is of class $\mathscr{C}^1$ and the fact that $\lambda$ can be written as the composition of the two $\mathscr{C}^1$ maps $(\varphi,t)\mapsto (t,\varphi(t),\varphi^\prime(t))$ and $F$. We get 
$$\dfrac{\partial\lambda}{\partial\varphi}\cdot u=\dfrac{\partial F}{\partial x}(t,\varphi(t),\varphi^\prime(t))\cdot u(t)+\dfrac{\partial F}{\partial y}(t,\varphi(t),\varphi^\prime(t))\cdot u^\prime(t)$$
which gives that $\dfrac{\partial\lambda}{\partial\varphi}$ is continuous and ends the proof.
\end{proof}
\begin{rem}
If $\lambda(\varphi,t)$ is of class $\mathscr{C}^k$ so is $\int_a^b\lambda(\varphi,t)dt$.
\end{rem}
\section{The minimum problem}
Let $\mathbf E$ be a Banach space. $\V$ denotes the previously defined set. Fix two points $\alpha,\beta\in \E$ and consider the set
$$W(\alpha,\beta)=\{\varphi:I\to \E,\mathscr{C}^1:\varphi(a)=\alpha,\varphi(b)=\beta\}.$$
In other words this is the set of $\mathscr{C}^1$ maps from $I$ to $\E$ with fixed end points $\alpha$ and $\beta$. One sees that $W(\alpha,\beta)$ is an affine space with associated vector space $W(0,0)$. Indeed by definition
$$W(\alpha,\beta)=\{\varphi:I\to \E,\mathscr{C}^1:\varphi(a)=0,\varphi(b)=0\}.$$
Introduce the following linear operators on $\V$, $\delta_a:\varphi\mapsto\varphi(a)$ and similarly for $\delta_b$. Then we have $W(0,0)=\ker\delta_a\cap\ker\delta_b$. Moreover for any $\varphi_0,\varphi_1\in W(\alpha,\beta)$, $\varphi_1-\varphi_0\in W(0,0)$ which gives at once the fact that $W(\alpha,\beta)$ is affine with associated vector space $W(0,0)$. Moreover one has $W(\alpha,\beta)=\varphi_0+W(0,0)$ ($\varphi_1$ could have been taken instead of $\varphi_0$). As a consequence the translation $\psi\mapsto\varphi_0+\psi$ induces an homeomorphism between $W(\alpha,\beta)$ and $W(0,0)$. This finally allows us to equip each such a $W(\alpha,\beta)$ with $(\alpha,\beta)\not=(0,0)$, with a distance with respect to which it is complete by defining $$d(\varphi_0+\psi_0,\varphi_0+\psi_1)=\|\psi_0-\psi_1\|$$ with $\|\cdot\|$ the previously defined norm in $\V$.

Recall that
$$W(\alpha,\beta)=\{\varphi:I\to \E,\mathscr{C}^1:\varphi(a)=\alpha,\varphi(b)=\beta\}.$$
 \begin{theo}
 $\varphi\in W(\alpha,\beta)$ is extremal if and only if 
 $$\displaystyle\int_a^b\left[\dfrac{\partial F}{\partial x}(t,\varphi(t),\varphi^\prime(t))\cdot u(t)+\dfrac{\partial F}{\partial y}(t,\varphi(t),\varphi^\prime(t))\cdot u^\prime(t)\right]dt=0$$
 for each $u:I\to \E$, $\mathscr{C}^1$, $u(a)=u(b)=0$.
 \end{theo}
 
 We make a small transformation. We set
 $$A(t):=\dfrac{\partial F}{\partial x}(t,\varphi(t),\varphi^\prime(t))\quad\&\quad B(t):=\dfrac{\partial F}{\partial y}(t,\varphi(t),\varphi^\prime(t))$$
 and suppose that $u(t)$ and $u^\prime(t)\in \mathbf E$ for all $t\in I$.
 Thus we must show that
 $$\displaystyle\int_a^b\left[A(t)\cdot u(t)+B(t)\cdot u^\prime(t)\right]dt=0.$$
\begin{theo}
The previous equation holds if and only if $B^\prime(t)$ exists and $B^\prime(t)=A(t)$.
\end{theo}
Here is another formulation
\begin{theo}
A function $\varphi\in W(\alpha,\beta)$ is extremal if and only if the mapping $\dfrac{\partial F}{\partial y}(t,\varphi(t),\varphi^\prime(t))$ is differentiable with respect to $t\in I$ and 
$$\displaystyle\dfrac{d}{dt}\dfrac{\partial F}{\partial y}(t,\varphi(t),\varphi^\prime(t))=\dfrac{\partial F}{\partial x}(t,\varphi(t),\varphi^\prime(t)).$$
These equations are called the Euler-Lagrange equations for the problem of extremum.
\end{theo}
\begin{proof}
If $A(t)=B^\prime(t)$ then
\begin{equation}
\begin{split}
&\displaystyle\int_a^b\left(B^\prime\cdot u(t)+B(t)\cdot u^\prime(t)\right)dt=\int_a^b\dfrac{d}{dt}(B(t)\cdot u(t))dt\\
&B(b)\cdot u(b)-B(a)\cdot u(a)=0.
\end{split}
\end{equation}
So the condition $A(t)=B^\prime(t)$ is sufficient. Let us prove that it is also necessary and for simplicity reason we assume that $B^\prime(t)$ exists. The function $t\mapsto B(t)\cdot u(t)$ is differentiable and because $u(a)=u(b)=0$ we have
\begin{equation}
\begin{split}
\displaystyle\int_a^b\left(B^\prime\cdot u(t)+B(t)\cdot u^\prime(t)\right)dt&=\left[B(t)\cdot u(t)\right]_a^b+\int_a^b(A(t)-B^\prime(t))\cdot u(t))dt\\
&=0.
\end{split}
\end{equation}
So $\displaystyle\int_a^b(A(t)-B^\prime(t))\cdot u(t))dt=0$ for any $u:I\to \E$ of class $\mathscr{C}^1$ and with $u(a)=u(b)=0$. Our goal is to deduce from this condition that $A(t)=B^\prime(t)$.
\begin{lem}
Let $C:I\to\mathcal{L}(\E,\R)$ be continuous and be such that $$\displaystyle\int_a^bC(t)\cdot u(t)dt=0$$ for each $u:I\to E$ of class $\mathscr{C}^1$ and with $u(a)=u(b)=0$. Then $C\equiv 0$.
\end{lem}
\begin{proof}
The proof is almost similar to the classical case. We assume by contradiction that $C­\not=0$. Then there exists $a<t_0<b$ with $C(t_0)­\not=0$. As $C(t_0)\in\mathcal{L}(\E,\R)$, there exists $u_0\in \E$: $C(t_0)\cdot u_0\not=­0$. We can assume that $C(t_0)\cdot u_0>0$.

By the continuity of $t\mapsto C(t)$, there exists $\varepsilon>0$ such that $\left[t_0-\varepsilon,t_0+\varepsilon\right]\subset\left[a,b\right]$ with $C(t)\cdot u_0>0$ on $\left[t_0-\varepsilon,t_0+\varepsilon\right]$.
Now choose a bump function $\lambda:I\to\R^{+}$, $\mathscr{C}^\infty$ with $supp\lambda\subset\left[t_0-\varepsilon,t_0+\varepsilon\right]$. For instance one make take for $\lambda$
$$\displaystyle\lambda(t)=\begin{cases}
      0&a\leqslant t\leqslant t_0-\varepsilon \text{ }, \\
     0& t_0+\varepsilon\leqslant t\leqslant b\text{}\\
     e^{-\frac{1}{(t_0+\varepsilon-t)(t-t_0+\varepsilon)}}&\, t\in]t_0-\varepsilon,t_0+\varepsilon[.
\end{cases}$$
We consider $u$ with $u(t)=\lambda(t)u_0$. Then for all $t\in I$ we obtain that $$C(t)\cdot u(t)=C(t)\cdot(\lambda(t)u_0)>0.$$ As a consequence we have
$$\displaystyle\int_a^bC(t)\cdot u(t)dt>0.$$
And this is a contradiction.
\end{proof}
Hence the theorem.
\end{proof}
\section{Euler-Lagrange equations in finite dimension}
Set $\E=\R^n$. We have $F(t,x_1,\cdots x_n,y_1,\cdots,y_n)$ of class $\mathscr{C}^k$ on $U\subset\R^{2n+1}$.

A function $\varphi:I\to\R^n$ of class $\mathscr{C}^1$ is given by its $n$ components $\varphi_i,\,1²i²n$, each of class $\mathscr{C}^1$ on $I$.
The mapping $\varphi\in\Omega$ if and only if $$(t,\varphi_1(t),\cdots,\varphi_n(t),\varphi_1^\prime(t),\cdots,\varphi_n^\prime(t))\in U$$
and we define the action
$$\displaystyle f(\varphi)=\int_a^bF(t,\varphi_1(t),\cdots,\varphi_n(t),\varphi_1^\prime(t),\cdots,\varphi_n^\prime(t))dt.$$
In this setting the Euler-Lagrange equations are 
$$\displaystyle\dfrac{d}{dt}\dfrac{\partial F}{\partial y_i}=\dfrac{\partial F}{\partial x_i},\quad 1\leqslant i\leqslant n.$$
More explicitly we have, by the chain-rule, 
\begin{equation}
\begin{split}
\displaystyle &\dfrac{d}{dt}\dfrac{\partial F}{\partial y_i}(t,\varphi_1(t),\cdots,\varphi_n(t),\varphi_1^\prime(t),\cdots,\varphi_n^\prime(t))=\dfrac{\partial^2F}{\partial t\partial y_i}(t,\varphi(t),\varphi^\prime(t))+\\
&\sum_{j=1}^n\dfrac{\partial^2F}{\partial x_j\partial y_i}\varphi_j^\prime(t)+\sum_{j=1}^n\dfrac{\partial^2F}{\partial y_j\partial y_i}\varphi_j^{\prime\prime}(t),\quad1\leqslant i\leqslant n.
\end{split}
\end{equation}
Therefore the Euler-Lagrange equations are equivalent to the system of second order differential equations for the $n$ unknown functions $x_1,\cdots,x_n$ of the variable $t$ given by
\begin{equation}
\begin{split}
&\displaystyle\dfrac{\partial^2F}{\partial t\partial x_i^\prime}(t,x,x^\prime)+\sum_{j=1}^n\dfrac{\partial^2F}{\partial x_j\partial x_i^\prime}(t,x,x^\prime)x_j^\prime+\sum_{j=1}^n\dfrac{\partial^2F}{\partial x_j^{\prime}\partial x_i^\prime}(t,x,x^\prime)x_j^{\prime\prime}\\
&=\dfrac{\partial F}{\partial x_i}(t,x,x^\prime),\,\,\,1\leqslant i\leqslant n.
\end{split}
\end{equation}
Under the non-degeneracy condition $\det\left(\dfrac{\partial^2F}{\partial x_j^{\prime}\partial x_i^\prime}\right)\not=­0$ on $U$ we can express the second derivative of the $x_i$s in terms of the remaining variables leading to equations of the form
$$\displaystyle x_i^{\prime\prime}=G_i(t,x,x^\prime),\,1\leqslant i \leqslant n$$
with the $G_i$s of class $\mathscr{C}^{k-2}$ if $F$ is of class $\mathscr{C}^k$.
\begin{prop}
If the determinant $\det\left(\dfrac{\partial^2F}{\partial x_j^{\prime}\partial x_i^\prime}\right)\not=­0$ and if the given $F$ is of class $\mathscr{C}^2$ then all the $x_i=\varphi_i(t)$ which satisfy the extremum condition are of class $\mathscr{C}^2$. 
\end{prop}
\begin{proof}
This follows from the implicit function theorem.
\end{proof}
Let us now suppose that $F$ is autonomous, i.e. independent of $t$. We are hence trying to look for the curves which are extremum for an action of the type $\displaystyle\int_a^bF(x,y)dt$, with $F$ being of class $\mathscr{C}^1$ on $U\subset \E\times \E$.
\begin{theo}
If $F$ is autonomous then $\displaystyle\dfrac{\partial F}{\partial y}\cdot y-F$ (which is real-valued) is constant on the curves realizing the extremum or again it is a constant of motion.
\end{theo}
\begin{proof}
We have $\displaystyle\dfrac{\partial F}{\partial y}\in\mathcal{L}(\E,\R)$ so $\displaystyle\dfrac{\partial F}{\partial y}\cdot y\in \R$. Suppose $\varphi$ realizes the extremum then it follows from the chain rule and the Euler-Lagrange equations that
\begin{equation}
\begin{split}
\displaystyle A(t)&:=\dfrac{d}{dt}\left[\displaystyle\dfrac{\partial F}{\partial y}(\varphi(t),\varphi^\prime(t))\cdot \varphi^\prime(t)-F(\varphi(t),\varphi^\prime(t))\right]\\
&=\dfrac{d}{dt}\displaystyle\dfrac{\partial F}{\partial y}(\varphi(t),\varphi^\prime(t))\cdot \varphi^\prime(t)-\dfrac{\partial F}{\partial x}(\varphi(t),\varphi^\prime(t))\cdot\varphi^\prime(t)\\
&=0.
\end{split}
\end{equation}
\end{proof}
Another important special case is the case when $F(x,y)$ is quadratic homogeneous with respect to the second variable. This means that we have $F(x,\lambda y)=\lambda^2F(x,y)$ for any $\lambda\in\R$.
\begin{defi}
A curve $\varphi:I\to \E$ of class $\mathscr{C}^1$ such that $\varphi^\prime(t)­\not=0$ on $I$ and $(\varphi(t),\varphi^\prime(t))\in U$ for $t\in I$ is called a parametrized curve. Let $\lambda:I\to I$ be surjective, of class $\mathscr{C}^1$, increasing and such that $\lambda^\prime(t)­\not=0$ on $I$. Then from $\varphi$ we can build another $\E$-valued curve $\varphi\circ\lambda$ by composition. This induces an equivalence relation on the set of parametrized curves and a geometrical curve is an equivalence class of parametrized curves.
\end{defi}
Now we remark that if $F(x,y)$ is quadratic homogeneous then the integral $\displaystyle\int_a^b\sqrt{F(\varphi(t),\varphi^\prime(t))}dt$ is well-defined for a geometric curve $\varphi(t)$. Indeed setting $t=\lambda(u)$ and using the chain rule plus the hypothesis of quadratic homogeneity we have
\begin{equation}
\begin{split}
\displaystyle\int_a^b \sqrt{F(\varphi(t),\varphi^\prime(t))}dt&=\int_a^b\sqrt{F((\varphi\circ\lambda)(u),(\varphi\circ\lambda)^\prime})\dfrac{1}{\lambda^\prime(u)}\times \lambda^\prime(u)du\\
&=\int_a^b \sqrt{F(\varphi\circ\lambda(u),\varphi\circ\lambda^\prime(u))}du.
\end{split}
\end{equation} 
\begin{prop}
Assume that for each $x$, $F(x,y)$ is a non degenerate positive quadratic form. Then each curve which realizes the extremum for $\displaystyle\int Fdt$ also does so for the action $\displaystyle\int \sqrt{F}dt$.
\end{prop}
First notice that we can speak about $G(x,y):=\sqrt{F(x,y)}$ because if $F$ is of class $\mathscr{C}^k$ on $U$ then $G$ is of class $\mathscr{C}^k$ on the open set 
$$\displaystyle\{(x,y): F(x,y)­\not=0\}.$$
\begin{proof}
Recall that the extremums for the action defined by $F$: $\displaystyle\int Fdt$ are solutions of the Euler-Lagrange equations 
$$\displaystyle\dfrac{d}{dt}\left(\dfrac{\partial F}{\partial y}\right)=\dfrac{\partial F}{\partial x}$$
and that the extremums for the action defined by $\sqrt{F}$ namely $\displaystyle\int \sqrt{F}dt$ are solutions of the Euler-Lagrange equations 
$$\displaystyle\dfrac{d}{dt}\left(\dfrac{1}{\sqrt{F}}\dfrac{\partial F}{\partial y}\right)=\dfrac{1}{\sqrt{F}}\dfrac{\partial F}{\partial x}.$$ 
But we know two properties of the problem at hand
\begin{itemize}
  \item we have constants of motions:
  $$\displaystyle\dfrac{\partial F}{\partial y}\cdot y-F$$
  is a constant of motion.
  \item We also have the Euler relation of homogeneity: $\displaystyle\dfrac{\partial F}{\partial y}\cdot y=2F$
\end{itemize}
Thus $2F-F=F=\dfrac{\partial F}{\partial y}\cdot y-F$ is a constant of motion. The result follows because we can multiply by $\sqrt{F}^{-1}$ (which is a constant) the first Euler-Lagrange equation to obtain the second one.
\end{proof}
\begin{prop}
Each extremal of $\displaystyle\int\sqrt{F}dt$ possesses a unique parametrization by $t\in I$ such that $\displaystyle F(\varphi(t),\varphi^\prime(t))$ is constant. The curve $t\mapsto\varphi(t)$ is then an extremal of $\displaystyle\int Fdt$.
\end{prop}
\begin{proof}
Let $\displaystyle x=\varphi(u)$ be an extremal of $\displaystyle\sqrt(F)dt$. We are looking for $u=\lambda(t)$ such that $\displaystyle \sqrt(F(\psi(\lambda(t)),\psi^\prime(\lambda(t)))\cdot\lambda^\prime(t))$ is independent of $t$. But
$$\displaystyle \sqrt(F(\psi(\lambda(t)),\psi^\prime(\lambda(t))))\cdot\lambda^\prime(t)=\lambda^\prime(t)\sqrt(F(\psi(u),\psi^\prime(u))),\quad u=\lambda(t).$$
We define $f(u):=\sqrt(F(\psi(u),\psi^\prime(u)))$. $f(u)$ is continuous and positive. We want to have
$$\displaystyle\dfrac{du}{dt}f(u)=C,\quad u=\lambda(t).$$
This will will give
$$\displaystyle t=\dfrac{1}{C}\int f(u)du+k=\dfrac{1}{C}\widetilde{f}(u)+k.$$
Choosing $C$ and $k$ so that 
$$\displaystyle u=a\Longrightarrow t=a$$
and
$$\displaystyle u=b\Longrightarrow t=b,$$
we see that $u\mapsto t=t(u)$ sends $I=\left[a,b\right]$ to itself and has a positive derivative. With this parameter $t$ the curve 
$$\displaystyle x=\varphi(t)=\psi(\lambda(t))$$
is a solution of the problem of extremum associated to $\displaystyle\int \sqrt{F}$. As $F$ is constant along this curve, the latter gives us a solution of the extremum problem associated to $F$.
\end{proof}
In the finite dimensional case, if $F$ is independent of $t$ (autonomous) then for each extremal, the following expression is constant along it
$$\displaystyle H=-F+\sum_{j=1}^np_jF_{p_j}=constant.$$
\subsection{Vocabulary}
Classically the following appellations are used
\begin{enumerate}
  \item $H$ is called the Hamiltonian. 
  \item $F=L=L(x,\xi)=L(x,\dot{x})$ is called the Lagrangian (remark a change of notation $x^\prime\leftrightarrow \dot{x}$).
  \item $\displaystyle E=E(x,\dot{x})=E(x,\xi)=\sum_1^n\dot{x}^i\dfrac{\partial L}{\partial \dot{x}^i}-L=\sum_1^n\xi^i\dfrac{\partial L}{\partial \xi^i}-L$ is called the Energy.
  \item $\displaystyle f_i:=\dfrac{\partial L}{\partial x^i}$ is called the Force.
  \item $\displaystyle p_i=\dfrac{\partial L}{\partial \dot{x}^i}=\dfrac{\partial L}{\partial \xi^i}$ is called the impulsion.
  \item $\displaystyle\dfrac{d}{dt}\left(\dfrac{\partial L}{\partial \dot{x}^i}\right)=\dfrac{\partial L}{\partial x^i}$ or $\dot{p}_i=f_i$ are the Euler-Lagrange equations.  
\end{enumerate}
  \begin{exple}
Let the Lagrangian $L$ be given by
$$\displaystyle L=\dfrac{m}{2}\sum_1^n(\xi^{i})^2-U(x)=F(x,\xi).$$  
  Then $f_i=-\dfrac{\partial U}{\partial x^i}$; $p_i=m\xi^i=m\dot{x}^i$ and the Euler-Lagrange equations are $$m\ddot{x}^i=-\dfrac{\partial U}{\partial x^i}.$$
  \end{exple}
Let the Lagrangian $L$ be given by
$$\displaystyle L=\dfrac{1}{2}\sum_{i,j}g_{ij}\xi^{i}\xi^{j}=\dfrac{1}{2}g_{ij}\xi^{i}\xi^{j}$$
with Einstein's summation convention understood.
We have 
$$\displaystyle p_i=g_{ij}\xi^j=\sum_{j}g_{ij}\xi^{j}$$
and
$$\displaystyle f_k=\dfrac{1}{2}\dfrac{\partial g_{ij}}{\partial x^k}\dot{x}^{i}\dot{x}^{j}=\dfrac{1}{2}\dfrac{\partial g_{ij}}{\partial x^k}\xi^{i}\xi^{j}.$$  
  Recall that the Euler-Lagrange equations are 
  $$\displaystyle \dfrac{dp_k}{dt}=\dfrac{1}{2}\dfrac{\partial g_{ij}}{\partial x^k}\dot{x}^{i}\dot{x}^{j}.$$
  But 
  $$\displaystyle\dfrac{dp_k}{dt}=\ddot {x}^jg_{kj}+\dot{x}^j\dfrac{dg_{kj}}{dt}=\ddot {x}^jg_{kj}+\dot{x}^i\dot{x}^j\dfrac{\partial g_{kj}}{\partial x^i}$$
  so that
  $$\displaystyle \ddot {x}^jg_{kj}+\dot{x}^i\dot{x}^j\dfrac{\partial g_{kj}}{\partial x^i}=\dfrac{1}{2}\dfrac{\partial g_{ij}}{\partial x^k}\dot{x}^{i}\dot{x}^{j}.$$
  We introduce the inverse matrix $g^{km}$ of $g_{ij}$ so that
  $$g^{km}g_{jk}=\delta_j^k.$$
  Then
  $$\displaystyle \ddot{x}^m+g^{km}\left(\dfrac{\partial g_{jk}}{\partial x^i}-\dfrac{1}{2}\dfrac{\partial g_{ij}}{\partial x^k}\right)\dot{x}^i\dot{x}^j=0.$$
 Using the following remark
  $$\displaystyle g^{km}\dfrac{\partial g_{jk}}{\partial x^i}\dot{x}^i\dot{x}^j=\dfrac{1}{2}\dot{x}^i\dot{x}^jg^{km}\left[\dfrac{\partial g_{jk}}{\partial x^i}+\dfrac{\partial g_{ik}}{\partial x^j}\right]$$
  we obtain
  $$\displaystyle\ddot{x}^m+\Gamma_{ij}^m\dot{x}^i\dot{x}^j=0$$
  with the Christoffel symbols
  $$\displaystyle\Gamma_{ij}^m=\dfrac{1}{2}g^{km}\left(\dfrac{\partial g_{ik}}{\partial x^j}+\dfrac{\partial g_{jk}}{\partial x^i}-\dfrac{\partial g_{ij}}{\partial x^k}\right).$$
  Thus
  \begin{theo}
  If 
  $$\displaystyle L=g_{ij}\dot{x}^i\dot{x}^j=<\dot{x},\dot{x}>=|\dot{x}|^2$$
  and the action is given by 
  $$\displaystyle S=\int_P^Q|\dot{x}|^2dt$$
  then the Euler-Lagrange equations for the action are just the equations of the geodesics of the metric $(g_{ij})$. 
  \end{theo}
  In this perspective let us examine the case of surfaces of revolution $\Sigma$ in $\R^3$. These are surfaces in $\R^3$ given parametrically by an equation of the form
  $$f(z,r=x^2+y^2)=0.$$
  Consider the cylindrical coordinates $x=r\cos\theta$, $y=r\sin\theta$, $z$ in $\R^3$. In $\R^3$ sits the Euclidian metric $dl^2=dx^2+dy^2+dz^2$. Then the induced metric on $\Sigma$ is 
  $$\displaystyle dl^2_{|_\Sigma}=dz^2+r^2d\theta^2+dr^2.$$
  Indeed 
  $$\displaystyle x=r\cos\theta\Longrightarrow dx=\cos\theta dr-r\sin\theta d\theta$$
  and
  $$\displaystyle y=r\sin\theta\Longrightarrow dy=\sin\theta dr+r\cos\theta d\theta.$$
  Plugging the expressions of $dx$ and $dy$ into the metric $dl^2$ we get the desired expression of $dl^2_{|_\Sigma}$.
  Assume that locally we can solve the equation $f(z,r)=0$ to obtain $r=r(z)$. The metric then transforms as
  $$dl^2_{|_\Sigma}=(1+r^{\prime2})dz^2+r^2d\theta^2=:g_{zz}dz^2+r^2d\theta^2.$$
  The Lagrangian for the geodesics is 
  $$L=\dfrac{1}{2}\left(g_{zz}\dot{z}^2+r^2(z)\dot{\theta}^2\right)$$
  Consider the two vectors $e_z$ and $e_\theta$ with
  $$<e_z,e_z>=g_{zz},\quad<e_z,e_\theta>=0,\quad <e_\theta,e_\theta>=r^2(z).$$
  The energy $E=L$ and the impulsion is $p=\dfrac{\partial L}{\partial \dot{\theta}}=r^2\dot{\theta}$ are conserved quantities because $L$ is autonomous.
  We define 
  $$\displaystyle\cos\alpha=\dfrac{<v,e_\theta>}{\sqrt{<v,v><e_\theta,e_\theta>}}$$ 
  with $v=\dot{z}e_z+\dot{\theta}e_\theta$. This gives by an easy computation
  $$\displaystyle\cos\alpha=\dfrac{r^2\dot{\theta}}{\sqrt{2}\sqrt{E}r}=\dfrac{p}{\sqrt{2}\sqrt{E}r}.$$
  Hence
  \begin{theo}
  The quantity $r\cos\alpha$ is constant along geodesics on surfaces of revolution.
  \end{theo}
  \subsection{Hamiltonian description}
  \subsubsection{The Legendre Transform}
  Recall that the energy and the impulsion are given by
  $$\displaystyle E=\sum_1^n\dot{x}^i\dfrac{\partial L}{\partial \dot{x}^i}-L$$
  and
  $$\displaystyle p_i=\dfrac{\partial L}{\partial \dot{x}^i}.$$
  We say that the Lagrangian is strongly non-degenerate if
  $$\displaystyle\det\left(\dfrac{\partial^2L}{\partial\dot{x}^i\partial\dot{x}^j}\right)­\not =0.$$
  If this is the case the equation $p_i=\dfrac{\partial L(x,\dot{x})}{\partial \dot{x}^i}$ can be solved using the implicit function theorem and its solution is of the form
  $$\displaystyle \dot{x}^i=V^i(x,p).$$
  \begin{defi}
  The Hamiltonian $H(x,p)$ is the energy $E$ given as a function of $x$ and $p$. The space with coordinates $(x,p)$ is called the phase space. If the map $(x,\dot{x})\mapsto (x,p)$ is invertible for strongly non-degenerate Lagrangians. Then this map is called the Legendre transform. It transforms Lagrangians $L(x,\dot{x})$ (beings living in the $(x,\dot{x})$ world) into Hamiltonians $H(x,p)$ (beings living in the $(x,p)$ world).
  \end{defi}
  \begin{theo}
  If the Lagrangian $L(x,\dot{x})$ is strongly non-degenerate i.e.
  $$\displaystyle\det\left(\dfrac{\partial^2L}{\partial\dot{x}^i\partial\dot{x}^j}\right)­\not =0$$
  then
  $$\displaystyle p_i=\dfrac{\partial L(x,\dot{x})}{\partial \dot{x}^i}$$
  can be solved bijectively. Let $H(x,p)$ be the Hamiltonian (recall that $\displaystyle \dot{x}=V(x,p)$). Then
  $$\displaystyle\dfrac{\partial L}{\partial x}-\dfrac{d}{dt}\left(\dfrac{\partial L}{\partial \dot{x}}\right)=0\iff\begin{cases}
     \dot{p} =& -\dfrac{\partial H}{\partial x} \\\\
      \dot{x}=&\dfrac{\partial H}{\partial p} 
\end{cases}$$
in the phase space.
  \end{theo}
  \begin{proof}
  To $L(x,\dot{x})$ we associate $H=-L+\sum\dot{x}^ip_i$ and we want to know what the Euler-Lagrange equations become. We have
  $$\displaystyle dH=-dL+\sum p_id\dot{x}^i+\sum \dot{x}^idp_i.$$
  Expanding we get
  $$\displaystyle dH=-\sum\dfrac{\partial L}{\partial x^i}dx^i-\sum\dfrac{\partial L}{\partial \dot{x}^i}d\dot{x}^i+\sum p_id\dot{x}^i+\sum \dot{x}^idp_i.$$
  Using the identity $\dfrac{\partial L}{\partial \dot{x}^i}=p_i$ we have
  $$\dfrac{\partial H}{\partial x^i}=-\dfrac{\partial L}{\partial x^i}\quad \&\quad\dfrac{\partial H}{\partial p_i}=\dot{x}^i.$$
As 
$$\dot{p}_i=\dfrac{\partial L}{\partial x^i}$$
the Euler-Lagrange equations become the Hamilton equations
$$\dot{x}^i=\dfrac{\partial H}{\partial p_i}\quad\&\quad \dot{p}_i=-\dfrac{\partial H}{\partial x^i}.$$
  \end{proof}
  \begin{rem}
  If $L(x,\dot{x})$ does not depend explicitly on $t$ then $H$ is constant along the extremums. Indeed one has
  \begin{equation}\begin{split}\displaystyle\dfrac{dH}{dt}&=\sum_{i=1}^{n}\left(\dfrac{\partial H}{\partial x_i}\dfrac{dx_i}{dt}+\dfrac{\partial H}{\partial p_i}\dfrac{dp_i}{dt}\right)\\
  &=\sum_{i=1}^{n}\left(\dfrac{\partial H}{\partial x_i}\dfrac{\partial H}{\partial p_i}-\dfrac{\partial H}{\partial p_i}\dfrac{\partial H}{\partial x_i}\right)=0.
  \end{split}
\end{equation}  
 Therefore $H$ is a first integral of the Euler equations. More generally if one considers a sufficiently smooth function $\displaystyle\phi=\phi(x_1,\cdots,x_n,p_1,\cdots,p_n)$ then
 $$\displaystyle\dfrac{d\phi}{dt}=\sum_{i=1}^{n}\left(\dfrac{\partial \phi}{\partial x_i}\dfrac{\partial H}{\partial p_i}-\dfrac{\partial \phi}{\partial p_i}\dfrac{\partial H}{\partial x_i}\right)=\{\phi,H\}$$
 where $\{\phi,H\}$ is the Poisson bracket of $\phi$ and $H$. Hence $\phi$ is a first integral of the Euler system 
 $$\displaystyle \dot{x}_i=\dfrac{\partial H}{\partial p_i} \quad \& \quad \dot{p}_i=-\dfrac{\partial H}{\partial x_i}$$
 if and only if $\displaystyle\{\phi,H\}=0$.
 \end{rem}
 \subsubsection{The Legendre Transform in dimension 1}
 Let $f(\xi)$ be a strictly convex function on $\R$. Let us find its minimum. To do so we introduce the tangential coordinate 
 $$p=f^\prime(\xi).$$
 Then
 $$\displaystyle\dfrac{dp}{d\xi}=f^{\prime\prime}(\xi)>0.$$
 Consider the new function of $p$
 $$\displaystyle H\left(p\right)=-f(\xi)+p\xi.$$
 It is well-defined because $\xi$ is a well-defined function of $p$. The transformation
 $$\displaystyle (\xi,f(\xi))\mapsto(p,H\left(p\right))$$
 is the Legendre transform. As
 $$\displaystyle dH=-f^{\prime}(\xi)d\xi+pd\xi+\xi dp,$$
one has  
$$\displaystyle \dfrac{dH}{dp}=\xi \quad \&\quad \dfrac{d^2H}{dp^2}=\dfrac{d\xi}{dp}=\dfrac{1}{\dfrac{dp}{d\xi}}=\dfrac{1}{f^{\prime\prime}}>0.$$
Thus the convexity of $f$ implies the convexity of $H$. Moreover one has
$$\displaystyle -H\left(p\right)+pH^{\prime}\left(p\right)=f(\xi)-p\xi+p\xi=f(\xi).$$
This means that the Legendre transform is involutive. To illustrate the determination of the minimum we restrict to a simple example. Let $\displaystyle a>1$ and $\displaystyle f(\xi)=\dfrac{\xi^{a}}{a}$. A simple computation gives $\displaystyle\xi=p^{\frac{1}{a-1}}$. Plugging this value into $H$ gives
$$\displaystyle H=-\dfrac{\xi^{a}}{a}+p\xi=\dfrac{p^b}{b},\quad\dfrac{1}{a}+\dfrac{1}{b}=1.$$
We consider $-H\left(p\right)+\xi p$ as a function of two variables. Let us show that
$$\displaystyle f(\xi)=\max_p\left(-H\left(p\right)+\xi p\right).$$
We have
$$\displaystyle \dfrac{\partial}{\partial p}(-H\left(p\right)+\xi p)=0\iff\xi=H^\prime\left(p\right)$$
and
$$\displaystyle \dfrac{\partial^2}{\partial p^2}(-H\left(p\right)+\xi p)=-H^{\prime\prime}\left(p\right)=-\dfrac{1}{f^{\prime\prime}(\xi)}<0.$$
As a consequence the extremum is a maximum and we obtain
$$\displaystyle f(\xi)=\max_p(-H\left(p\right)+\xi p)$$
and
$$\min_\xi f(\xi)=\min_\xi\max_p(-H\left(p\right)+\xi p).$$
\section{The Noether theorem}
We have seen that the Euler equations corresponding to the functional
$$\displaystyle\int_a^bF(x_1,\cdots,x_n,\dot{x}_1,\cdots,\dot{x}_n)dt$$
where $F$ is independent of $t$ admit as first integral the following function $H$
$$\displaystyle H=-F+\sum_{i=1}^n\dot{x}_iF_{\dot{x}_i}.$$
The fact that $F$ is independent of $t$ means that $F$ and the integral are invariant under the change of variable $\displaystyle x\mapsto x^\star=x+\epsilon$, with $\epsilon$ a constant.
Consider 
$$\displaystyle J(x_1,\cdots,x_n)=\int_{t_0}^{t_1}F(t,x_1,\cdots,x_n,\dot{x}_1,\cdots,\dot{x}_n)dt=\int_{t_0}^{t_1}F(t,x,\dot{x})dt$$
and the transformation
\begin{equation}
\begin{split}
t^\star&=\phi(t,x,\dot{x})\\
x_i^\star&=\Psi_i(t,x,\dot{x}),\quad 1²i²n
\end{split}
\end{equation}
This transformation sends a curve $\gamma$ defined by $x=x(t),\,t_0\leqslant t\leqslant t_1$ on another curve $\gamma^\star$ defined by $\displaystyle x^\star=x^\star(t^\star),\,t_0^\star \leqslant t^\star \leqslant t_1^\star$.
\begin{defi}
The functional $\displaystyle J(x)=\int_{t_0}^{t_1}F(t,x,\dot{x})dt$ is said to be invariant under the above group of transformations if $\displaystyle J(\gamma)=J(\gamma^\star)$, i.e.
$$\displaystyle\int_{t_0^\star}^{t_1^\star}F(t^\star,x^\star,\dfrac{dx^\star}{dt^\star})dt^\star= \int_{t_0}^{t_1}F(t,x,\dfrac{dx}{dt}).$$
\end{defi}
\begin{ex}
We consider $\displaystyle J(x)=\int_{t_0}^{t_1}\dot{x}^2dt$ and the transformations 
$$\displaystyle\begin{cases}t^\star=t+\epsilon\\x^\star(t^\star)=x(t)=x(t^\star-\epsilon).\end{cases}$$
Then
$$\displaystyle J(\gamma^\star)=\int_{t_0^\star}^{t_1^\star}\dfrac{dx^{\star 2}}{dt^\star}(t^\star)dt^\star=\int_{t_0+\epsilon}^{t_1+\epsilon}\dfrac{dx^{ 2}}{dt^\star}(t+\epsilon)dt^\star=\int_{t_0}^{t_1}\dfrac{dx^{2}}{dt}(t)dt=J(\gamma).$$
\end{ex}
We take a family of transformations
$$\displaystyle \begin{cases}t^\star=\phi(t,x,\dot{x};\epsilon)\\x_i^\star=\Psi_i(t,x,\dot{x};\epsilon)\end{cases}$$
with:
\begin{enumerate}
\item $\phi$ and $\Psi_i$ differentiable functions with respect to $\epsilon$
\item $\phi(t,x,\dot{x};0)=t$;\quad $\Psi_i(t,x,\dot{x};0)=x_i,\,1\leqslant i\leqslant n$.
\end{enumerate}
\begin{theo}[Noether]
If the functional $\displaystyle J(\gamma)=\int_{t_0}^{t_1}F(t,x,\dot{x})dt$ is invariant under the action of the above given family of transformations for every $t_0$ and $t_1$ then
$$\displaystyle\sum_{i=1}^n F_{\dot{x}_i}\psi_i+(F-\sum\dot{x}_iF_{\dot{x}_i})\varphi\equiv constant$$
along any extremal of $J(\gamma)$, where
$$\displaystyle\varphi(t,x,\dot{x})=\dfrac{\partial\phi(t,x,\dot{x};\epsilon)}{\partial\epsilon}_{|_{\epsilon=0}}$$
$$\displaystyle\psi_i(t,x,\dot{x})=\dfrac{\partial\Psi_i(t,x,\dot{x};\epsilon)}{\partial\epsilon}_{|_{\epsilon=0}}.$$
In other words, every one parameter family of transformations leaving the functional $J(\gamma)$ invariant admits a first integral for its Euler system.   
\end{theo}
\begin{proof}
One has by the Taylor-Young formula
$$\displaystyle t^\star=\phi(t,x,\dot{x};\epsilon)=\phi(t,x,\dot{x};0)+\epsilon\dfrac{\partial\phi(t,x,\dot{x};\epsilon)}{\partial\epsilon}_{|_{\epsilon=0}}+\circ(\epsilon)$$
$$\displaystyle x_i^\star=\displaystyle\Psi_i(t,x,\dot{x};\epsilon)=\Psi_i(t,x,\dot{x};0)+\epsilon\dfrac{\partial\Psi_i(t,x,\dot{x};\epsilon)}{\partial\epsilon}_{|_{\epsilon=0}}+\circ(\epsilon)$$
or again
$$\displaystyle t^\star=t+\epsilon\varphi(t,x,\dot{x})+\circ(\epsilon)$$
$$\displaystyle x_i^\star=x_i+\epsilon\psi_i(t,x,\dot{x})+\circ(\epsilon).$$
We suppose that the curve $x_i=x_i(t),\,1\leqslant i\leqslant n$ is an extremal; then one has the following fundamental identity \cite[p.~59]{gelfandfomin}
$$\displaystyle\delta J=\epsilon\left[\sum_{i=1}^{n}F_{\dot{x}_i}\psi_i+(F-\sum\dot{x}_i)F_{\dot{x}_i})\varphi\right]_{x=x_0}^{x=x_1}.$$
Because $J(\gamma)$ is invariant under the group of transformations we have $\delta J=0$ hence the result because $x_0$ and $x_1$ are arbitrary.
\end{proof}
\section{The Linear Sigma model}
Let $\mathbf{M}$ be a Riemannian manifold of dimension $m$, with metric tensor $$\displaystyle(\gamma_{\alpha\beta})_{1\leqslant\alpha,\beta\leqslant m}.$$ We recall the following notations 
$$\displaystyle(\gamma^{\alpha\beta})_{1\leqslant\alpha,\beta\leqslant m};\quad\gamma^{\alpha\beta}\gamma_{\beta\nu}=\delta_\nu^\alpha$$
and the Christoffel symbols
$$\Gamma_{\beta\gamma}^{\alpha}=\dfrac{1}{2}\gamma^{\alpha\delta}\left(\gamma_{\beta\delta,\eta}+\gamma_{\eta\delta,\beta}-\gamma_{\beta\eta,\delta}\right).$$
For a function $\phi:\mathbf{M}\to \R$ of class $\mathscr{C}^1$, we consider the invariantly defined  "norm"
$$\displaystyle \|d\phi\|^2=\langle d\phi,d\phi\rangle_{T^\star \mathbf{M}}=\gamma^{\alpha\beta}\dfrac{\partial\phi(x)}{\partial x^\alpha}\dfrac{\partial\phi(x)}{\partial x^\beta}$$
The Dirichlet integral of $\phi$ with respect to this metric is 
$$\displaystyle \mathcal{S}(\phi)=\dfrac{1}{2}\int_\mathbf{M}\gamma^{\alpha\beta}\dfrac{\partial\phi(x)}{\partial x^\alpha}\dfrac{\partial\phi(x)}{\partial x^\beta}\sqrt{\gamma}dx_1\cdots dx_n=\dfrac{1}{2}\int_\mathbf{M}\|d\phi\|^2dv_\gamma(\mathbf{M}).$$
Minimizers of the functional are harmonic functions: they solve the Laplace-Beltrami equation
$$\displaystyle\Delta_\mathbf{M}\phi=\dfrac{1}{\sqrt{\gamma}}\dfrac{\partial}{\partial x^\alpha}\left(\sqrt{\gamma}\gamma^{\alpha\beta}\dfrac{\partial}{\partial x^\beta}\phi\right)=0.$$
Let us specialize to the two-dimensional case, that is we consider a surface endowed with some Riemannian metric. We have
$$\displaystyle\|d\phi\|^2=\gamma^{zz}\dfrac{\partial\phi}{\partial z}\dfrac{\partial\phi}{\partial z}+2\gamma^{z\overline{z}}\dfrac{\partial\phi}{\partial z}\dfrac{\partial\phi}{\partial \overline{z}}+\gamma^{\overline{z}\,\overline{z}}\dfrac{\partial\phi}{\partial \overline{z}}\dfrac{\partial\phi}{\partial \overline{z}}.$$
The action becomes
\begin{equation}\begin{split}\displaystyle\mathcal{S}(\phi)&=\dfrac{1}{2}\int\left(\gamma^{zz}\dfrac{\partial\phi}{\partial z}\dfrac{\partial\phi}{\partial z}+2\gamma^{z\overline{z}}\dfrac{\partial\phi}{\partial z}\dfrac{\partial\phi}{\partial \overline{z}}+\gamma^{\overline{z}\,\overline{z}}\dfrac{\partial\phi}{\partial \overline{z}}\dfrac{\partial\phi}{\partial \overline{z}}\right)\sqrt{\gamma_{zz}\gamma_{\overline{z}\,\overline{z}}-\gamma_{z\overline{z}}^2}dz\wedge d\overline{z}\\
&=\dfrac{1}{2}\int_\mathbf{M}\gamma^{\alpha\beta}\dfrac{\partial\phi}{\partial x^\alpha}\dfrac{\partial\phi}{\partial x^\beta}\sqrt{\gamma_{11}\gamma_{22}-\gamma_{12}^2}dx\wedge dy.\end{split}\end{equation}
It is fruitful to consider the action not only as a function of the field $\phi$, but also to see it as depending on the metric $\gamma$. We thus write $\mathcal{S}(\phi,\gamma)$. We want to study the effect of variations $\delta\gamma$ of the metric on $\mathcal{S}$. We recall that $$\sqrt{\gamma_{11}\gamma_{22}-\gamma_{12}^2}=(\sqrt{\gamma^{11}\gamma^{22}-\left(\gamma^{12}\right)^2})^{-1}.$$
Computation gives
$$\displaystyle\delta\mathcal{S}(\phi,\gamma)=\int_\mathbf{M}T_{\alpha\beta}\delta\gamma^{\alpha\beta}\sqrt{\gamma_{11}\gamma_{22}-\gamma_{12}^{2}}dx\wedge dy$$
with
$$\displaystyle T_{\alpha\beta}=\dfrac{\partial\phi}{\partial x^\alpha}\dfrac{\partial\phi}{\partial x^\beta}-\dfrac{1}{2}\gamma_{\alpha\beta}\gamma^{\varepsilon\eta}\dfrac{\partial\phi}{\partial x^\varepsilon}\dfrac{\partial\phi}{\partial x^\eta}.$$
$T_{\alpha\beta}$ is called the energy-momentum tensor.
\begin{lem}
$T$ is trace-free, that is $T_\alpha^\alpha=0$. 
\end{lem}
\begin{proof}
This is due to the conformal invariance of $\mathcal{S}$ in dimension $2$: when we change the metric $\gamma$ to $e^\sigma\gamma$ for some $\sigma:\mathbf{M}\to\R$, then $\mathcal{S}$ remains invariant:
$$\displaystyle \mathcal{S}(\phi,e^\sigma\gamma)=\mathcal{S}(\phi,\gamma).$$
Infinitesimally the variation of $\gamma^{-1}$ is $-\delta\sigma\gamma^{-1}$ so that
$$\displaystyle0=\int_\mathbf{M} T_{\alpha\beta}\delta\sigma\gamma^{\alpha\beta}\sqrt{\gamma_{11}\gamma_{22}-\gamma_{12}^2}dx\wedge dy=\int_\mathbf{M} T_{\alpha}^\alpha\delta\sigma\sqrt{\gamma_{11}\gamma_{22}-\gamma_{12}^2}dx\wedge dy$$ 
for all variations $\delta\sigma$. This implies $T_\alpha^\alpha=0$.
\end{proof}
\section{Electromagnetism}
The electromagnetic field is described by two vector fields. The electric field 
$$E=(E_1,E_2,E_3)$$
and the magnetic field
$$H=(H_1,H_2,H_3).$$
They satisfy the Maxwell equations
\begin{equation}
\begin{aligned}
\displaystyle\rot E&=-\dfrac{\partial H}{\partial x_0},\\
\dive H&=0,\end{aligned}
\qquad
\begin{aligned}
\displaystyle\rot H&=\dfrac{\partial E}{\partial x_0}\\
\dive E&=0.
\end{aligned}
\end{equation}
We recall that
$$\displaystyle\dive E=\dfrac{\partial E_1}{\partial x_1}+\dfrac{\partial E_2}{\partial x_2}+\dfrac{\partial E_3}{\partial x_3}$$
and
$$\displaystyle\rot E=\left(\dfrac{\partial E_3}{\partial x_2}-\dfrac{\partial E_2}{\partial x_3},\dfrac{\partial E_1}{\partial x_3}-\dfrac{\partial E_3}{\partial x_1},\dfrac{\partial E_2}{\partial x_1}-\dfrac{\partial E_1}{\partial x_2}\right).$$
We can express $E$ and $H$ in terms of a potential $\{A_j\}=(A_0,A_1,A_2,A_3)$ by introducing
$$\displaystyle E=\grad A_0-\dfrac{\partial A}{\partial x_0},\quad H=\rot A, \quad A:=(A_1,A_2,A_3)$$
and
$$\grad A_0=\left(\dfrac{\partial A_0}{\partial x_1},\dfrac{\partial A_0}{\partial x_2},\dfrac{\partial A_0}{\partial x_3}\right).$$
This potential $\{A_j\}$ is not uniquely determined by $E$ and $H$. Indeed $E$ and $H$ remain unchanged when we make a gauge transformation, i.e. when we replace $\{A_j\}$ by a new potential $\{\mathcal{A}_j\}$, satisfying
$$\displaystyle\mathcal{A}_j(x)=A_j(x)+\dfrac{\partial f}{\partial x_j}(x),\quad j=0,1,2,3$$
with $x=(x_0,x_1,x_2,x_3)$ and $f(x)$ an arbitrary sufficiently differentiable function. To avoid this lack of uniqueness we impose a condition called the Lorentz condition
$$\displaystyle-\dfrac{\partial A_0}{\partial x_0}+\dive A=\sum_{j=0}^3g_{jj}\dfrac{\partial A_j}{\partial x_j}=0$$
where $g_{jj}=(-1,1,1,1)$. We consider the antisymmetric matrix
$$\displaystyle H=\left(\begin{array}{cccc}0 & -E_1 & -E_2 & -E_3 \\E_1 & 0 & H_3 & -H_2 \\E_2 & -H_3 & 0 & H_1 \\E_3 & H_2 & -H_1 & 0\end{array}\right).$$
Then $\displaystyle H_{ij}=\dfrac{\partial A_j}{\partial x_i}-\dfrac{\partial A_i}{\partial x_j}$. It is possible to express the Maxwell equations by means of the matrix $H$:

$$\displaystyle \sum_{i=0}^3g_{ii}\dfrac{\partial H_j}{\partial x_i}=0,\quad j=0,1,2,3$$
$$\dfrac{\partial H_{ij}}{\partial x_k}+\dfrac{\partial H_{ki}}{\partial x_j}+\dfrac{\partial H_{jk}}{\partial x_i}=0,\quad i,j,k=0,1,2;\, 1,2,3;\,2,3,0;\,3,0,1.$$
Summing up we see that the second identity in the previous equation is identically satisfied, while the first one reduces to
\begin{equation}\displaystyle\Box A_j=0,\,j=0,1,2,3,\quad \Box=-\dfrac{\partial^2}{\partial x_0^2}+\dfrac{\partial^2}{\partial x_1^2}+\dfrac{\partial^2}{\partial x_2^2}+\dfrac{\partial^2}{\partial x_3^2}.\end{equation}
$\Box$ is called the d'Alembertian. These d'Alembertian equations can be shown to be the consequence of a variational problem if we consider the Lagrangian 
\begin{equation}\displaystyle\mathscr{L}=\dfrac{1}{8\pi}(E^2-H^2)=\dfrac{1}{8\pi}\left[\left(\grad A_0-\dfrac{\partial A}{\partial x_0}\right)^2-(\rot A)^2\right].\end{equation}
We verify that the Euler-Lagrange of $\mathscr{L}$ reduce to give the d'Alembertian equations. The Euler-Lagrange equations are
$$\dfrac{\partial\mathscr{L}}{\partial A_j}-\sum_{i=0}^3\dfrac{\partial\mathscr{L}}{\partial\left(\dfrac{\partial A_j}{\partial x_i}\right)}=0,\,(j=0,1,2,3).$$
A computation using the Lorentz condition gives the result.
\section{Hadamard's variation formula for the Green's function}
In this section we intend to derive the so-called Hadamard variation formula, which represents the first order term in the expansion of the Green's function in powers of a small parameter describing an infinitesimal displacement of the boundary of the domain. This formula has many applications in the calculus of variations, and it suggests the answers to a variety of extremal problems that are basic to an understanding of the behavior of solutions of linear elliptic partial differential equations.

\subsection{Hadamard's variation formula for the Green's function}
Let $D$ a plane domain with boundary an analytic simple closed curve. We are interested in the derivation of Hadamard's variational formula for the harmonic Green's function of the plane domain $D$. Let $D^\star$ the region whose boundary $\partial D^\star$ is obtained by shifting $\partial D$ an infinitesimal distance 
\begin{equation}
\label{had1}
\delta\nu=\epsilon\rho(s)
\end{equation}
along its inner normal $\nu$, where $\epsilon$ is a small number and $\rho$ is an real analytic function of the arc length $s$ on $\partial D$. It is to be understood that points of $\partial D$ where $\delta \nu$ is negative are moved in the direction of the outer normal instead of the inner normal. Our objective will be to find the first order term $\delta G$ in the expansion of the harmonic Green's function $G^\star$ of $D^\star$ in powers of the parameter $\epsilon$.

Consider the identity
\begin{equation}
\label{had2}
\displaystyle\int\int_{D-D^\star}\nabla G\cdot\nabla G^\star dx_1dx_2=\int_{\partial D}\dfrac{\partial G}{\partial \nu}\dfrac{\partial G^\star}{\partial \nu}\delta\nu ds+O(\epsilon^2),
\end{equation}  
which is a consequence of \eqref{had1} and the fact that $G$ vanishes on $\partial D$, so that
$$\dfrac{\partial G}{\partial s}=0$$
there. It is assumed that the integral on the left in \eqref{had2} is evaluated over $D^\star-D$ instead of $D-D^\star$, but with a reversed sign, at points associated with negative values of the normal displacement $\delta\nu$ of $\partial D$. Moreover we suppose that $\epsilon$ is so small that the analytic extensions of $G$ and $G^\star$ beyond the boundary curves $\partial D$ and $\partial D^\star$, which can be found by the Schwarz reflection principle, are defined in a fixed region including the closures of both $D$ and $D^\star$. 

We use complex numbers $z$, $\zeta$ to indicate points inside the domain $D$, with, for example, z$=x_1+ix_2$, and we adopt the normalization
$$G(z,\zeta)=-\dfrac{1}{2\pi}\log|z-\zeta|+\ldots$$
of the singularity of the Green's function. Thus Green's theorem gives
$$\int\int_{D}\nabla G\cdot\nabla G^\star dx_1dx_2=-\int_{\partial D} G\dfrac{\partial G^\star}{\partial \nu} ds+G(w,\zeta)=G(w,\zeta)$$
and
$$\int\int_{D^\star}\nabla G\cdot\nabla G^\star dx_1dx_2=-\int_{\partial D^\star}\dfrac{\partial G}{\partial \nu}G^\star ds+G^\star(\zeta,w)=G^\star(\zeta,w),$$
where
$$G=G(z,\zeta),\quad G^\star=G^\star(z,w).$$
Comparig these results with \eqref{had2}, we conclude that
$$G^\star(\zeta,w)-G(\zeta,w)=-\int_{\partial D}\dfrac{\partial G}{\partial \nu}\dfrac{\partial G^\star}{\partial \nu}\delta\nu ds+O(\epsilon^2)$$
because of the symmetry of $G$. Since the difference between $G^\star$ and $G$ is of the order of magnitude $\epsilon$, we can replace the normal derivative $\dfrac{\partial G^\star}{\partial \nu}$ on the right by $\dfrac{\partial G}{\partial \nu}$ to establish the announced Hadamard formula
\begin{equation}
\label{had3}
\displaystyle\delta G(\zeta,w)=-\int_{\partial D}\dfrac{\partial G}{\partial \nu}(z,\zeta)\dfrac{\partial G}{\partial \nu}(z,w)\delta\nu ds
\end{equation}
for the first variation $\delta G$ of the Green's function $G$.

A rigorous mathematical analysis of Hadamard's formula requiring fewer assumptions about the boundary $\partial D$ can be made by introducing the notion of interior variations of the domain $D$. Let 
$$S_j=S_j(x_1,x_2),\quad j=1,2,$$
stand for a pair of functions possessing piece-wise continuous partial derivatives of a conveniently high order in a neighborhood of the closure of $D$. For small enough choices of the parameter $\epsilon$ the transformation
\begin{equation}
\label{had4}
x^\star_j=x_j+\epsilon S_j,\quad j=1,2,
\end{equation}  
is one-to-one in $D$ and maps that region onto a nearby domain $D^\star$. We refer to \eqref{had4} as an interior variation because it is defined inside $D$ as well as on the boundary curves $\partial D$. Our intention is to develop the Green's function $G^\star$ of $D^\star$ as a perturbation series in powers of $\epsilon$. For that purpose we shall use the change of variables \eqref{had4} to express it as a function of $x_1$ and $x_2$ over the original domain region $D$, where it will be found to satisfy a linear elliptic partial differential equation with coefficients involving $\epsilon$. The variations of $G^\star$ under changes in these coefficients will suffice to specify in a precise fashion how the Green's function depends in its domain of definition. We employ the complex notation
\begin{equation}
\label{had5}
z^\star=z+\epsilon F(z,\overline{z})
\end{equation} 
for the interior variation \eqref{had4}, where
\begin{equation}
\label{had6}
F(z,\overline{z})=S_1(x_1,x_2)+iS_2(x_1,x_2).
\end{equation}
Similarly, we put
$$\zeta^\star=\zeta+\epsilon F(\zeta,\overline{\zeta}),\quad w^\star=w+\epsilon F(w,\overline{w}).$$
Our aim is to determine a partial differential equation in $D$ for the transformed Green's function 
$$g(z,\zeta;\epsilon):=G^\star(z^\star,\zeta^\star),$$
whose principal arguments are the real and imaginary parts $x_1$ and $x_2$ of the variable $z$.

Since Laplace's equation
$$\Delta G^\star=0$$
is the Euler equation for an extremal problem of the form
$$\displaystyle\int\int\left[\sum_{j=1}^2\left(\dfrac{\partial G^\star}{\partial x_j^\star}\right)^2\right]dx_1^\star dx_2^\star=minimum$$
in the calculus of variation, $g$ must satisfy the Euler equation of the transformed problem
$$\displaystyle\int\int\left[\sum_{j=1}^2\left(\sum_{k=1}^2\dfrac{\partial g}{\partial x_k}\dfrac{\partial x_k}{\partial x_j^\star}\right)^2\right]\dfrac{\partial (x_1^\star,x^\star_2)}{\partial(x_1,x_2)}dx_1dx_2=minimum.$$
Standard procedures show that the latter equation is 
\begin{equation}
\label{had7}
L_\epsilon[g]=0,
\end{equation}
where $L_\epsilon$ represents the self-adjoint differential operator defined by
\begin{equation}
\label{had8}
\displaystyle L_\epsilon[g]=\sum_{k,l=1}^2\dfrac{\partial}{\partial x_k}\left(A_{kl}\dfrac{\partial g}{\partial x_l}\right),
\end{equation}
$$\displaystyle A_{kl}=\dfrac{\partial(x_1^\star,\partial x_2^\star)}{\partial(x_1,x_2)}\sum_{j=1}^2\dfrac{\partial x_k}{\partial x_j^\star}\dfrac{\partial x_l}{\partial x_j^\star}$$
depending on the parameter $\epsilon$. Note, incidentally, that
\begin{equation}
\label{had9}
A_{11}A_{22}-A_{21}A_{12}=1.
\end{equation}
We intend to derive a Fredholm integral equation for $g$ which will enable us to expand it in a Neumann series each of whose terms is multiplied by a successively higher power of $\epsilon$. The integral equation in question will be based on a parametrix for the linear elliptic differential equation \eqref{had7}. For present purposes we shall need the inverse matrix of elements $a_{jk}=a_{jk}(z)$ such that
$$\sum_{k=1}^2a_{jk}A_{kl}=\delta_{jl}.$$
Let us introduce the quadratic form
\begin{equation}
\label{had10}
\displaystyle\Gamma(z,\zeta)=\sum_{j,k=1}^2a_{jk}(\zeta)(x_j-\xi_j)(x_k-\xi_k),
\end{equation}
where $\xi_1$ and $\xi_2$ stand for the real and imaginary parts of $\zeta$. We denote by $\alpha=\alpha(z,\zeta)$ a fixed function of the two points $z$ and $\zeta$ in $D$ which possesses continuous derivatives of an appropriate order, which fulfills the boundary condition $\alpha(z,\zeta)=0$ when either $z$ or $\zeta$ lies on $\partial D$, but which has the value 
$$\alpha(\zeta,\zeta)=\dfrac{1}{4\pi}$$
when $z$ coincides with $\zeta$ inside $D$. The expression
$$P_\epsilon=\alpha(z,\zeta)\log\left(\dfrac{1}{\Gamma(z,\zeta)}\right)$$
then defines the required parametrix $P_\epsilon$.

Direct computation shows that 
\begin{equation}
\label{had11}
L_\epsilon[P_\epsilon]=O\left(\dfrac{1}{r}\right),\,r=|z-\zeta|.
\end{equation}
Thus, in particular, the quantity on the left of \eqref{had11} is integrable over the domain $D$. In a similar way it is found that
\begin{equation}
\label{had12}
L_\epsilon[P_\epsilon]-L_0[P_0]=O\left(\dfrac{\epsilon}{r}\right).
\end{equation}
Finally we shall have occasion in what follows to exploit the estimate
\begin{equation}
\label{had13}
L_\epsilon[G-P_0]-L_0[G-P_0]=O\left(\dfrac{\epsilon}{r}\right),
\end{equation}
which is obviously true for reasonable choices of the factor $\alpha$.

For the self adjoint differential operator \eqref{had8} Green's identity takes the form
\begin{equation}
\label{had14}
\displaystyle\int\int_D\{vL_\epsilon[u]-uL_\epsilon[v]\}dx_1dx_2=\int_{\partial D}B_\epsilon[u,v]ds,
\end{equation}
where
$$B_\epsilon[u,v]=\sum_{k,l=1}^2A_{kl}\{u\dfrac{\partial v}{\partial x_k}\dfrac{\partial x_l}{\partial \nu}-v\dfrac{\partial u}{\partial x_k}\dfrac{\partial x_l}{\partial \nu}\}.$$
We put
$$u=G(z,\zeta)-P_0(z,\zeta)+P_\epsilon(z,\zeta)$$
and
$$v=g(z,w;\epsilon)$$
in
\eqref{had14}, after removing from $D$ small circles about the singular points $\zeta$ and $w$ which are later allowed to disappear. The limits of the contributions from these circles may be evaluated in a standard fashion, with note made that $g$ has the same logarithmic infinity as $P_\epsilon$ and that 
$$a_{11}a_{22}-a_{12}a_{21}=1$$
because of \eqref{had9}. We are thus led to the representation 
\begin{equation}
\label{had15}
\displaystyle g(\zeta,w;\epsilon)=G(w,\zeta)-P_0(w,\zeta)+P_\epsilon(w,\zeta)+\int\int_Dg(z,w;\epsilon)L_\epsilon[G-P_0+P_\epsilon]dx_1dx_2
\end{equation}
of the transformed Green's function $g$, since the equations \eqref{had7} holds throughout $D$ and the boundary conditions
$$g=G=P_\epsilon=P_0=0$$
apply along $\partial D$.

We interpret \eqref{had15} as a Fredholm integral equation for $g$ with the known kernel
$$K(z,\zeta;\epsilon)=L_\epsilon[G(z,\zeta)-P_0(z,\zeta)+P_\epsilon(z,\zeta)].$$
The relations \eqref{had12} and \eqref{had13} provide the simpler restriction
$$K=O\left(\dfrac{\epsilon}{r}\right)$$
on the growth of $K$ which shows that it is integrable and that the Neumann series obtained by iterating \eqref{had15} has to converge for sufficiently small values of the parameter $\epsilon$. Moreover, each successive term of the series contains the factor $\epsilon$ more times than does its predecessor. We conclude, in particular, that
$$g-G=O(\epsilon).$$
Hence we can replace $g$ by $G$ on the right in \eqref{had15} to derive the perturbation formula
\begin{equation}
\label{had16}
\begin{split}
\displaystyle G^\star(\zeta^\star,w^\star)-g(\zeta,w)&=P_\epsilon(w,\zeta)-P_0(w,\zeta)\\
&+\int\int_DG(z,w)K(z,\zeta;\epsilon)dx_1dx_2+O(\epsilon^2)
\end{split}
\end{equation}  
for the Green's function, where on the left the definition of $g$ and the symmetry of $g$ have been taken into account.

Our purpose will be to simplify \eqref{had16} and to develop from it a rigorous verification of the Hadamard variational formula \eqref{had3}. Let
$$\delta A_{jk}=\dfrac{\partial A_{jk}}{\partial \epsilon}\epsilon,$$
where the derivative on the right is to be evaluated at $\epsilon=0$. An application of Green's theorem shows that
\begin{equation}
\begin{split}
\label{had17}
\displaystyle\int\int_D G(z,w)K(z,\zeta;\epsilon)dx_1dx_2&=G(\zeta,w)-\int\int_D\sum_{k,l=1}^2A_{jk}\dfrac{\partial G}{\partial x_j}\dfrac{\partial}{x_k}[G-P_0+P_\epsilon]dx_1dx_2\\
&=G(w,\zeta)-\int\int_D\sum_{j=1}^2\dfrac{\partial G}{\partial x_j}\dfrac{\partial}{\partial x_j}[G-P_0+P_\epsilon]dx_1dx_2\\
&-\int\int\sum_{j,k=1}^2\delta A_{jk}\dfrac{\partial G}{\partial x_j}\dfrac{\partial}{\partial x_k}[G-P_0+P_\epsilon]dx_1dx_2+O(\epsilon^2),
\end{split}
\end{equation} 
since $A_{jk}=1$ when $\epsilon=0$. On the other hand, an additional integration by parts gives using the boundary conditions $g=G=P_\epsilon=P_0=0$ along $\partial D$,
$$\int\int_D\sum_{j=1}^2\dfrac{\partial G}{\partial x_j}\dfrac{\partial}{\partial x_j}[G-P_0+P_\epsilon]dx_1dx_2=G(w,\zeta)-P_0(w,\zeta)+P_\epsilon(w,\zeta),$$and it evident that
$$\int\int_D\sum_{j=1}^2\delta A_{jk}\dfrac{\partial G}{\partial x_j}\dfrac{\partial}{\partial x_k}[-P_0+P_\epsilon]dx_1dx_2=O(\epsilon^2).$$
Therefore \eqref{had16} combines with \eqref{had17} to yield the more elegant variational formula
\begin{equation}
\label{had18}
\displaystyle G^\star(\zeta^\star,w^\star)-G(\zeta,w)=-\int\int_D\sum_{j,k=1}^2\delta A_{jk}\dfrac{\partial G(z,\zeta)}{\partial x_j}\dfrac{\partial G(z,w)}{\partial x_k}dx_1dx_2+O(\epsilon^2),
\end{equation}
from which the parametrix $P_\epsilon$ has been eliminated.

By direct calculation we find that
$$\dfrac{\partial x_k}{\partial x_j^\star}=\delta_{jk}-\epsilon\dfrac{\partial S_k}{\partial x_j}+O(\epsilon^2)$$
and
$$\dfrac{\partial(x_1^\star,x_2^\star)}{\partial(x_1,x_2)}=1+\epsilon\left(\dfrac{\partial S_1}{\partial x_1}+\dfrac{\partial S_2}{\partial x_2}\right)+O(\epsilon^2),$$
whence
\begin{equation}
\label{had19}
\delta A_{jk}=\epsilon\{\delta_{jk}\left(\dfrac{\partial S_1}{\partial x_1}+\dfrac{\partial S_2}{\partial x_2}\right)-\dfrac{\partial S_j}{\partial x_k}-\dfrac{\partial S_k}{\partial x_j}\}.
\end{equation}
substituting \eqref{had19} into \eqref{had18} and rearranging terms appropriately, we obtain
\begin{equation}
\label{had20}
\displaystyle G^\star(\zeta^\star,w^\star)-G(\zeta,w)=\epsilon\int\int_D\sum_{j,k=1}^2T_{jk}\dfrac{\partial S_j}{\partial x_k}dx_1dx_2+O(\epsilon^2),
\end{equation}
\begin{equation}
\label{had21}
\begin{split}
T_{jk}=T_{jk}(z;\zeta,w)&=\dfrac{\partial G(z,\zeta)}{\partial x_j}\dfrac{\partial G(z,w)}{\partial x_k}+\dfrac{\partial G(z,w)}{\partial x_j}\dfrac{\partial G(z,\zeta)}{\partial x_k}\\
&-\delta_{jk}\left(\dfrac{\partial G(z,\zeta)}{\partial x_1}\dfrac{\partial G(z,w)}{\partial x_1}+\dfrac{\partial G(z,\zeta)}{\partial x_2}\dfrac{\partial G(z,w)}{\partial x_2}\right)
\end{split}
\end{equation}
is a generalization of the Maxwell stress tensor. Because $G$ is a harmonic function of the variables $x_1$ and $x_2$, the identities
$$\displaystyle\sum_{k=1}^2\dfrac{\partial T_{jk}}{\partial x_k}=0,\quad j=1,2,$$
are seen to hold throughout $D$. Thus an integration by parts enables us to bring \eqref{had20} in the form
\begin{equation}
\label{had22}
\begin{split}
\displaystyle G^\star(\zeta^\star,w^\star)-G(\zeta,w)&=\epsilon\sum_{j=1}^2\left[\dfrac{\partial G(\zeta,w)}{\partial \xi_j}S_j(\xi_1,\xi_2)+\dfrac{\partial G(\zeta,w)}{\partial u_j}S_j(u_1,u_2)\right]\\
&-\int_{\partial D}\sum_{j,k=1}^2T_{jk}(z;\zeta,w)S_{j}(x_1,x_2)\dfrac{\partial x_k}{\partial\nu}ds+O(\epsilon^2)
\end{split}
\end{equation}
involving only a boundary integral, where $u_1$ and $
u_2$ stand for the real and imaginary parts of $w$.

According to Taylor's theorem we have
$$\displaystyle G^\star(\zeta^\star,w^\star)-G^\star(\zeta,w)=\epsilon\sum_{j=1}^2\left[\dfrac{\partial G^\star(\zeta,w)}{\partial \xi_j}S_j(\xi_1,\xi_2)+\dfrac{\partial G^\star(\zeta,w)}{\partial u_j}S_j(u_1,u_2)\right]+O(\epsilon^2).$$ 
Moreover $G^\star$ can be replaced by $G$ on the right, because of \eqref{had22}. Therefore \eqref{had22} actually reduces to the simpler formula
\begin{equation}
\label{had23}
\displaystyle\delta G=-\epsilon\int_{\partial D}\sum_{j,k=1}^2T_{jk}S_j\dfrac{\partial x_k}{\partial \nu}ds
\end{equation}
for the first order perturbation 
$$\displaystyle\delta G=\epsilon\lim_{\epsilon\to 0}\dfrac{G^\star(\zeta,w)-G(\zeta,w)}{\epsilon}$$
of the Green's function $G$. 
\begin{rem}
It is natural to ascribe to \eqref{had23} the physical interpretation that the variation $\delta G$ in potential energy, viewed as work done, must be proportional to a force, represented by the stress tensor $T_{jk}$, times the displacement defined by the vector $(S_1,S_2)$. \end{rem}
Finally if we insert the expression 
$$\rho=S_1\dfrac{\partial x_1}{\partial\nu}+S_2\dfrac{\partial x_2}{\partial \nu}$$
for the normal shift of the boundary curves $\partial D$ into \eqref{had1} and exploit the boundary condition $G=0$ when calculating $T_{jk}$ along $\partial D$, we find that \eqref{had23} is equivalent to Hadamard's variational formula \eqref{had3}.
\begin{rem}
In fluids at rest pressure is the only contact force. For solids at rest or in motion, this is in general false. Beside pressure like forces acting along the normal to a contact surface, there may also be shear forces acting tangentially to it. In complete analogy with pressure, the relevant quantity turns out to be the shear stress, defined to be the shear force per unit of area. For instance friction is part of shear stresses. In three dimensions we have Cauchy's stress hypothesis which expresses the infinitesimal force $d\mathscr{F}$ exerted on an element of surface $dS$, with the stress tensor $\sigma$ and $dS$:
$$d\mathscr{F}=\sigma\cdot dS$$
with
$$\sigma=\left(\begin{array}{ccc}\sigma{xx} & \sigma{xy} & \sigma{xz} \\\sigma{yx} & \sigma{yy} & \sigma{yz} \\\sigma{zx} & \sigma{zy} & \sigma{zz}\end{array}\right),\quad dS=(dS_x,dS_y,dS_z).$$
Under stress fluids flow, whereas solids deform.

Suppose we act on the external surface of a body, with some distribution of contact forces under control. Imagine now that the material is displaced by an infinitesimal displacement, so as to neglect the effect of internal forces. Let $\delta u=\delta u(x)$, the displacement vector. The work of all external forces is then
$$\displaystyle \delta W=\int_S\delta u\cdot\sigma\cdot dS.$$
If the constraints are plane, meaning that $\sigma{xz}=\sigma{zx}=\sigma{yz}=\sigma{zy}=\sigma{zz}=0$ and $S$ is a plane surface, then we get a formula analogous to \eqref{had23}.  
\end{rem}
\chapter{Symplectic Geometry}
\label{chap3}
In this chapter we refer to \citep[part I]{mcduff}. We have seen that if $H(x),\,x=(q,p)$ is a Hamiltonian function on $\R^{2n}$, then 
$$\displaystyle\dot{q}=H_p(q,p),\quad \dot{p}=-H_q(q,p),\quad p=(p_1,\cdots,p_n),\, q=(q_1,\cdots,q_n)$$
or
\begin{equation}\displaystyle\left(\begin{array}{c}\dot{q} \\\dot{p}\end{array}\right)=\left(\begin{array}{cc}0 & I_n \\-I_n & 0\end{array}\right)\left(\begin{array}{c}H_q \\H_p\end{array}\right).\end{equation}
Let $\mathbf{V}$ be an $m$-dimensional vector space over $\R$ and let $\Omega:\mathbf{V}\times \mathbf{V}\to\R$ be a bilinear map. The map $\Omega$ is skew-symmetric if
$$\displaystyle\Omega(u,v)=-\Omega(v,u),\,\forall u,v\in V\iff\Omega(u,u)=0,\forall u\in V.$$
\begin{theo}[Standard form for skew-symmetric bilinear maps] Let $\Omega$ be a skew-symmetric bilinear map on $\mathbf{V}$. Then there is a basis $\displaystyle B=\{u_1,\cdots,u_k,e_1,\cdots,e_n,f_1,\cdots,f_n\}$ of $\mathbf{V}$ such that:
\begin{enumerate}
\item $\Omega(u_i,v)=0$ for all $i$ and all $v\in \mathbf{V}$.
\item $\Omega(e_i,e_j)=\Omega(f_i,f_j)$ for all $i$ and all $j$.
\item $\Omega(e_i,f_j)=\delta_{ij}$, with $\delta_{ij}$ the Kronecker symbol. 
\end{enumerate}
\end{theo}
\begin{proof}
We apply the skew version of the Gram-Schmidt orthogonalization process. Let
$$\displaystyle \mathbf{U}=\ker(\Omega):=\{u\in \mathbf{V},\Omega(u,v)=0,\forall v\in \mathbf{V}\}.$$
We choose a basis $\{u_1,\cdots,u_k\}$ of $\mathbf{U}$ and a complementary subspace $\mathbf{W}$ to $\mathbf{U}$ in $\mathbf{V}$:
$$\mathbf{V}=\mathbf{U}\oplus \mathbf{W}.$$
Consider any $e_1\in \mathbf{W},\,e_1­0.$ Then there exists an $f_1\in \mathbf{W}$ verifying $\Omega(e_1,f_1)­0$ (by definition). We can assume that $\Omega(e_1,f_1)=1$. Let 
$$\displaystyle \mathbf{W}_1=\spa\{e_1,f_1\}\quad \&\quad \mathbf{W}_1^\Omega=\{w\in \mathbf{W}\,/\, \Omega(w,v)=0,\forall v\in \mathbf{W}_1\}.$$
Claim $1$: $\mathbf{W}_1\cap \mathbf{W}_1^\Omega=\{0\}$. Indeed if $v=ae_1+bf_1\in \mathbf{W}_1\cap \mathbf{W}_1^\Omega$, then 
$\displaystyle\begin{cases}
\Omega(v,e_1)&=a\Omega(e_1,e_1)+b\Omega(f_1,e_1)=-b=0\\
\Omega(v,f_1)&=a\Omega(e_1,f_1)+b\Omega(f_1,f_1)=a=0.\qquad 
\end{cases}$
$\Longrightarrow v=0.$

\noindent Claim $2$: $\mathbf{W}=\mathbf{W}_1\oplus \mathbf{W}_1^\Omega$. In fact if $v\in \mathbf{W}$ and we define
$$c:=\Omega(v,e_1),\,d:=\Omega(v,f_1),$$
then
$$v=(-cf_1+de_1)+(v+cf_1-de_1)$$
with $(-cf_1+de_1)\in \mathbf{W}_1$ and one verifies that $(v+cf_1-de_1)\in \mathbf{W}_1^\Omega.$ Hence the claim.

\noindent We proceed in the same way this time with $\mathbf{W}_1^\Omega$: we choose $e_2\in \mathbf{W}_1^\Omega$ with $e_2­0$ and $f_2\in \mathbf{W}_1^\Omega$ with $\Omega(e_2,f_2)=1$. We then consider $\mathbf{W}_2=\spa\{e_2,f_2\}$ and so on. The process eventually stops as $\dim \mathbf{V}<+\infty$. Thus
$$\displaystyle \mathbf{V}=\mathbf{U}\oplus \mathbf{W}_1\oplus \mathbf{W}_2\oplus\cdots\oplus \mathbf{W}_n$$
where all the summands are orthogonal with respect to $\Omega$ and where $\mathbf{W}_i$ has a basis given by $\{e_i,f_i\}$ and $\Omega(e_i,f_i)=1$.
\end{proof}
\begin{rem}
If $k=\dim \mathbf{U}$ then $\dim \mathbf{V}=k+2n$. Both n and k are invariants of $(\mathbf{V},\Omega)$ i.e. they are independent of the choice of a basis. The integer $2n$ is called the rank of $\Omega$. 
\end{rem}
\begin{rem}
Let $\mathbf{V}$ be a real vector space of dimension $m$ and $\Omega:\mathbf{V}\times \mathbf{V}\to \R$ a bilinear map. One defines
$$\displaystyle\widetilde{\Omega}:\mathbf{V}\longrightarrow \mathbf{V}^\star,\quad \widetilde{\Omega}(v)(u)=\Omega(v,u).$$
The kernel of $\widetilde{\Omega}$ is the previously defined subspace $\displaystyle \mathbf{U}=\{u\in \mathbf{V},\Omega(u,v)=0,\forall v\in \mathbf{V}\}$
\end{rem}
\begin{defi}
A skew symmetric form $\Omega$ is called a symplectic (or non-degenerate) form when $\widetilde{\Omega}$ is is bijective i.e. $\mathbf{U}=\{0\}$. In this case the couple $(\mathbf{V},\Omega)$ is called a symplectic vector space.
\end{defi}
\begin{rem}
When $(\mathbf{V},\Omega)$  is symplectic, $\mathbf{U}=\{0\}$ and $\dim \mathbf{V}=2n$. Also $\mathbf{V}$ admits a basis 
$$\mathscr{B}=\{e_1,\cdots,e_n,f_1,\cdots,f_n\}$$
with
$$\Omega(e_i,f_j)=\delta_{ij},\quad \Omega(e_i,e_j)=\Omega(f_i,f_j)=0.$$
This basis is called a symplectic basis of $(\mathbf{V},\Omega)$. In coordinates one has
\begin{equation}
\begin{split}
\Omega(u,v)&=(u_1,\cdots,u_n,u_{n+1},\cdots,u_{2n})\left(\begin{array}{cc}0 & -I_n \\I_n & 0\end{array}\right)\left(\begin{array}{c}v_1 \\\vdots \\v_n \\ v_{n+1} \\ \vdots\\ v_{2n}\end{array}\right)\\
&=u_1v_{n+1}+u_2v_{n}+\cdots+u_nv_{2n}-u_{n+1}v_{1}-\cdots-u_{2n}v_{n}.
\end{split}
\end{equation}
The prototype of a symplectic manifold is $(R^{2n},\Omega_0)$ with $\Omega_0$ given in coordinates by the previous calculation. A symplectic basis of it is given by the canonical basis of $\R^{2n}$. 
\end{rem}
\begin{defi}
A symplectomorphism $\varphi$ between symplectic vector spaces $(\mathbf{V},\Omega)$ and $(\mathbf{V}^\prime,\Omega^\prime)$ is a linear isomorphism $\varphi:\mathbf{V}\to \mathbf{V}^\prime$ such that $\varphi^\star\Omega^\prime=\Omega$ (by definition $(\varphi^\star\Omega^\prime)(u,v)=\Omega^\prime(\varphi(u),\varphi(v))).$
\end{defi}
\section{Symplectic Manifolds}
Let $\omega$ be a $2$-form on a manifold $\mathbf{M}$, this means that for each $p\in \mathbf{M}$
$$\omega_p: T_p\mathbf{M}\times T_p\mathbf{M}\to \R$$
is a skew-symmetric bilinear form on the tangent space to $\mathbf{M}$ at $p$, and that $\omega_p$ varies smoothly with $p$.
\begin{defi}
The $2$-form $\omega$ is symplectic if $\omega$ is closed and $\omega_p$ is symplectic for each $p\in \mathbf{M}$.
\end{defi}
\begin{rem}
If $\omega$ is symplectic then $\dim \mathbf{M}=\dim T_p\mathbf{M}$ must be even.
\end{rem}
\begin{defi}
A symplectic manifold is a pair $(\mathbf{M},\omega)$ where $\mathbf{M}$ is a manifold and $\omega$ a symplectic form on $\mathbf{M}$.
\end{defi}
\begin{ex}
Let $\mathbf{M}=\R^{2n}$ with coordinates $x_1,\cdots,x_n, y_1,\cdots,y_n$. The form 
$$\displaystyle\omega_0=\sum_1^ndx_i\wedge dy_i$$
is symplectic and
$$\displaystyle \bigg\{\dfrac{\partial}{\partial x_1}_{|_p},\cdots,\dfrac{\partial}{\partial x_n}_{|_p},\dfrac{\partial}{\partial y_1}_{|_p},\cdots,\dfrac{\partial}{\partial y_n}_{|_p}\bigg\}$$
is the symplectic basis on $T_p\mathbf{M}$.
\end{ex}
\begin{ex}
Let $\mathbf{M}=\C^n$ with coordinates $z_1,\cdots,z_n$, the form $\displaystyle\omega_0=\dfrac{i}{2}\sum_1^ndz_k\wedge d\overline{z}_k$ is symplectic. It equals the previous one under the identification $\C^n\simeq\R^{2n},\, z_k=x_k+iy_k$.
\end{ex}
\begin{ex}
The unit sphere $\mathbf{M}=\mathbf{S}^2$ in $\R^3$ has a symplectic form
$$\displaystyle\omega_p(u,v)=\langle p,u\wedge v\rangle,\quad u,v\in T_p\mathbf{M}.$$
$T_p\mathbf{M}$ may be identified with the plane orthogonal to $p$. This form is closed because it is of top degree and it is non-degenerate because $\langle p,u\wedge v\rangle\not =­0$ when $u­\not=0$ and we take for instance $v=u\wedge p$.\end{ex}
\begin{defi}
Let $(\mathbf{M}_1,\omega_1)$ and $(\mathbf{M}_2,\omega_2)$ be two $2n$-dimensional symplectic manifolds and let $g:\mathbf{M}_1\to \mathbf{M}_2$ be a diffeomorphism. We say that $g$ is a symplectomorphism when $g^\star\omega_2=\omega_1$ i.e.
$$(g^\star\omega_2)_q(u,v)=(\omega_1)_{g(q) }(dg_q(u),dg_q(v))=(\omega_1)_q(u,v).$$
\end{defi}
The theorem of Darboux says that locally the dimension is the only invariant of a symplectic manifold up to diffeomorphism.
\begin{theo}[Darboux]
Let $(\mathbf{M},\omega)$ be a $2n$-dimensional symplectic manifold and let $p\in\mathbf{M}$. Then there exists an open set $U$ and coordinates $x_1,\cdots,x_n,y_1,\cdots,y_n$ on $U$, centered at $p$ such that
$$\displaystyle\omega|_U=\sum_1^ndx_i\wedge dy_i.$$
\end{theo}
Hence locally every symplectic manifold is diffeomorphic to $\R^{2n}$ with linear coordinates $x_1,\cdots,x_n,y_1,\cdots,y_n$ and symplectic form 
$$\displaystyle\omega_0=\sum_1^ndx_i\wedge dy_i.$$
\section{Symplectic form on the Cotangent Bundle of a Manifold}
Let $\mathbf{X}$ be a manifold and $\mathbf{M}=T^\star\mathbf{X}$ its tangent bundle. If the manifold structure on $\mathbf{X}$ is described by coordinate charts $(U,x_1,\cdots,x_n)$ with $x_i:U\to\R$, then at any $x\in U$, the differentials $(dx_1)_x,\cdots,(dx_n)_x$ form a basis of $T^\star_x\mathbf{X}$. If $\xi\in T^\star_x\mathbf{X}$, then $\displaystyle \xi=\sum_1^n\xi_i(dx_i)_x$ for some real coordinates $\xi_1,\cdots,\xi_n$. This induces a map
\begin{equation}
\begin{split}
T^\star U&\longrightarrow\R^{2n}\\
(x,\xi)&\longrightarrow(x_1,\cdots,x_n,\xi_1,\cdots,\xi_n).
\end{split}
\end{equation}
The coordinates $x_1,\cdots,x_n,\xi_1,\cdots,\xi_n$ are the cotangent coordinates associated to the coordinates $x_1,\cdots,x_n$ of $U$. The transition functions on the overlaps are smooth: given two charts $(U,x_1,\cdots,x_n)$ and $(U^\prime,x_1^\prime,\cdots,x_n^\prime)$ let $x\in U\cap U^\prime$; if $\xi\in T^\star_x\mathbf{X}$ then
$$\displaystyle\xi=\sum_1^n\xi_i(dx_i)_x=\sum_{i,j}^n\xi_i\dfrac{\partial x_i}{\partial x^\prime_j}(dx_j^\prime)_x=\sum_1^n\xi^\prime_j(dx_j^\prime)_x$$
where
$$\displaystyle\xi^\prime_j=\sum_i\xi_i\dfrac{\partial x_i}{\partial x_j^\prime}$$
depends smoothly on the variables because $\mathbf{X}$ is a manifold. Thus $\mathbf{M}=T^\star\mathbf{X}$ is a $2n$-dimensional manifold.  

\noindent Let $(U,x_1,\cdots,x_n)$ be a coordinate chart of $\mathbf{X}$, with associated cotangent coordinates $(T^\star U,x_1,\cdots,x_n,\xi_1,\cdots,\xi_n)$. Define a $2$-form $\omega$ on $T^\star U$ by
$$\omega=\sum_1^ndx_i\wedge d\xi_i.$$
If $\displaystyle\alpha=\sum_1^n\xi_idx_i$, then $\omega=-d\alpha$. We are going to show that $\alpha$ is intrinsically defined (hence $\omega$ is also intrinsically defined). Indeed if we consider two overlapping coordinate charts of $\mathbf{M}$, $(T^\star U,x_1,\cdots,x_n,\xi_1,\cdots,\xi_n)$ and $(T^\star U^\prime,x_1^\prime,\cdots,x_n^\prime,\xi_1^\prime,\cdots,\xi_n^\prime)$, then on the intersection $T^\star U\cap T^\star U^\prime$ one has:
$$\displaystyle\xi_j^\prime=\sum_i\xi_i\dfrac{\partial x_i}{\partial x_j^\prime}.$$
Since
$$\displaystyle dx_j^\prime=\sum_i\dfrac{\partial x_j^\prime}{\partial x_i}dx_i$$
we have
$$\displaystyle\alpha=\sum_i\xi_idx_i=\sum_j\xi_j^\prime dx_j^\prime=\alpha^\prime.$$
This follows from the following identity
$$\displaystyle\sum_j\xi_j^\prime dx_j^\prime=\sum_j(\sum_i\xi_i\dfrac{\partial x_i}{\partial x_j^\prime})\sum_k\dfrac{\partial x_j^\prime}{\partial x_k}dx_k=\sum_i\xi_i(\sum_{j,k}\dfrac{\partial x_i}{\partial x_j^\prime}\dfrac{\partial x_j^\prime}{\partial x_k})dx_k=\alpha.$$
The $1$-form $\alpha$ is the tautological or Liouville $1$-form on $\mathbf{M}$ and the $2$-form $\omega$ is the canonical symplectic form on $\mathbf{M}$. We can give a coordinate free definition of $\alpha$: let 
\begin{equation}
\displaystyle
\begin{split}
\pi:\mathbf{M}=T^\star\mathbf{X}&\to \mathbf{X}\\
p=(x,\xi)&\to x,\quad \xi\in T^\star_x \mathbf{X}
\end{split}
\end{equation}
be the natural projection. The tautological $1$-form $\alpha$ may be defined pointwise by
$$\alpha_p=(d\pi_p)^t\xi \in T_p\mathbf{M}$$
with $d\pi_p^t$ the transpose of $d\pi_p$, that is $(d\pi_p)^t\xi=\xi\circ d\pi_p$. Equivalently
$$\alpha_p(v)=\xi((d\pi_p)v),\,v\in T_p\mathbf{M}.$$
The canonical symplectic $2$-form $\omega$ on $\mathbf{M}=T^\star\mathbf{X}$ is defined as $\omega=-d\alpha$ and locally:
$$\displaystyle\omega=\sum_i^ndx_i\wedge d\xi_i.$$
\chapter{Finslerian Geometry}
\label{chap{2}}
\section{Extremals}
We consider an open set $U\subsetñ\R^n,\,x=(x^1,\cdots,x^n)$ and a function 
$$f:U\times\R^n\to\R,\,\mathscr{C}^2.$$
Let $C:\left[a,b\right]\to U,\quad x^i=x^i(t),\,a²t²b$ be a $\mathscr{C}^1$ curve. Curves having fixed end points are called admissibles. We consider the integral
$$\displaystyle J(C)=\int_a^bf(x(t),\dot{x}(t))dt,\quad\dot{x}(t):=\dfrac{dx}{dt}.$$
Let
$$\displaystyle E_i(C)  =\dfrac{d}{dt}\left(\dfrac{\partial f}{\partial y^i}\right)-\dfrac{\partial f}{\partial x^i},\quad y^i:=\dot{x}^i$$
be the Euler-Lagrange equations. We are going to require that the admissible curves are regarded as oriented curves with increasing parameter $t$ and that the integral $J(C)$ is independent of the choice of such a parameter: $t=\varphi(\tau),\,\tau\in\left[c,d\right]\subset \R$ is a change of parameter with $\dfrac{d\varphi}{d\tau}>0$. Suppose that
$$\displaystyle \int_a^bf(x,\dfrac{dx}{dt})dt=\int_c^df(x,\dfrac{dx}{d\tau})d\tau,\quad a=\varphi \left(c\right),b=\varphi(d)$$
this gives 
$$\displaystyle f(x,\dfrac{dx}{dt})\dfrac{dt}{d\tau}=f(x,\dfrac{dx}{d\tau}).$$
In particular for $t=p\tau,\,p>0$ we have
$$\displaystyle pf(x,\dfrac{dx}{dt})=f(x,p\dfrac{dx}{dt}).$$
\begin{defi}
A function $g(u)$, $u=(u^1,\cdots,u^n)$ is positively homogeneous of degree $r$ if for every $p>0$ one has
$$g(pu)=p^rg(u).$$
\end{defi} 
\begin{theo}[Caratheodory]
\label{carath}
The integral $J(C)$ is independent of the choice of the parameter if and only if $f(x,y)$ is positively homogeneous of degree $1$ with respect to the variable $y$.
\end{theo}
We list some properties of homogeneous functions.
\begin{enumerate}
\item Let $f(x,y)$ be a positively homogeneous function with respect to the variable $y$. The derivatives in $x$, $f_{x_i}(x,y)$ are also homogeneous in $y$.
\item Consider a homogenous function in $u$ of degree $r$: $g(pu)=p^rg(u)$; if we differentiate with respect to $p$ and set $p=1$, then 
$$\displaystyle \sum_i\dfrac{\partial g}{\partial u^i}u^i=rg(u).$$
Conversely $\displaystyle \sum_i\dfrac{\partial g}{\partial u^i}u^i=rg(u)$ implies that
$$\displaystyle \sum_i\dfrac{\partial g}{\partial u^i}(pu^i)=rg(pu)=\dfrac{\partial g(pu)}{\partial p}p.$$
Regarding $g(pu)$ as a function $h\left(p\right)$ of $p$, the previous relation shows that
$$\displaystyle\dfrac{d}{dp}\left(\dfrac{h\left(p\right)}{p^r}\right)=0\iff h\left(p\right)=p^rh(1).$$
This is the same as $g(pu)=p^rg(u)$.
\end{enumerate}
\begin{theo}[Euler]
A function of class $\mathscr{C}^1$ is homogeneous of degree $r$ if and only if
\begin{equation}
\displaystyle\displaystyle \sum_i\dfrac{\partial g}{\partial u^i}u^i=rg(u).\end{equation}
\end{theo}
Now if we differentiate the Euler relation with respect to $u^j$ (for a fixed arbitrary j) we obtain
$$\displaystyle \sum_i\dfrac{\partial^2g}{\partial u^i\partial u^j}u^i+\dfrac{\partial g}{\partial u^j}=r\dfrac{\partial g}{\partial u^j}$$
or
$$\displaystyle \sum_i\dfrac{\partial^2g}{\partial u^i\partial u^j}u^i=(r-1)\dfrac{\partial g}{\partial u^j}.$$
Hence we have the
\begin{prop}
If a function $g$ of class $\mathscr{C}^2$ is positively homogeneous of degree $r$ in $u$, then $\dfrac{\partial g}{\partial u^j}$ is positively homogeneous of degree $r-1$.
\end{prop} 
\begin{prop}
If a positively homogeneous function $g(u)$ of degree $r$ is continuous at $u=0$, then
\begin{itemize}
\item $r>0\Longrightarrow g(0)=0$
\item $r=0\Longrightarrow g(u)=constant$.
\end{itemize} 
\end{prop}
\begin{proof}
We have $g(pu)=p^rg(u)$ and $\displaystyle\lim_{p\to 0}p^rg(u)=g(0)$.
\end{proof}
\section{Regularity}
We consider again the Euler-Lagrange equations with the hypothesis that $f(x,y)$ is of class $\mathscr{C}^2$ and homogeneous of degree $1$ in $y$. We set 
$$f_i:=\dfrac{\partial f}{\partial y^i},\quad f_{ij}=\dfrac{\partial f_i}{\partial y^j}.$$
We have $(f=f(x,\dot{x}))$
$$\displaystyle E_i\left(C\right)=f_{ij}\dfrac{dy^j}{dt}+\dfrac{\partial f_i}{\partial x^j}y^j-\dfrac{\partial f}{\partial x^i}=0.$$
Since $f_i$ is $0$-homogeneous in $y$, we have
$$\displaystyle\sum_jf_{ij}y^j=0.$$
Hence the rank of the matrix $(f_{ij})$ must be less than $n$ and we can not write the differential equations $E_i\left(C\right)$ in the so-called canonical form. We have
$$\displaystyle\sum_i\dfrac{\partial}{\partial y^i}\left(\dfrac{\partial f}{\partial x^j}\right)y^i=\dfrac{\partial f}{\partial x^j}$$
or
$$\displaystyle\sum_i\dfrac{\partial f_i}{\partial x^j}y^i=\dfrac{\partial f}{\partial x^j}.$$
Therefore using the just given relations we have
$$\displaystyle\sum_i E_i\left(C\right)y^i=0.$$
This identity means that the $n$-equations $E_i\left(C\right)$ are not independent. 

\noindent We consider the function $F(x,y)=\dfrac{1}{2}f(x,y)^2$, which is $2$-homogeneous in $y$ so that
$$E_i\left(C\right)=\dfrac{d\left(\dfrac{F_i}{f}\right)}{dt}-\dfrac{\dfrac{\partial F}{\partial x^i}}{f}$$
where $F_i:=\dfrac{\partial F}{\partial y^i}$. We take the normalized parameter $\tau$ of $C$ defined by
$$\displaystyle\tau=\int_a^tf(x(s),\dot{x}(s))ds$$
which gives 
$$\dfrac{d\tau}{dt}=f(x(t),\dot{x}(t));\,\,f(x,\dfrac{dx}{dt}\dfrac{dt}{d\tau})=1.$$ 
Hence the Euler equations may be rewritten
$$\displaystyle E_i\left(C\right)=\dfrac{d}{d\tau}\left(\dfrac{F_i}{f}\right)\dfrac{d\tau}{dt}-\dfrac{\dfrac{\partial F}{\partial x^i}}{f}=\dfrac{dF_i}{d\tau}-\dfrac{\partial F}{\partial x^i}=0.$$
We put
$$F_{ij}(x,y)=\dfrac{\partial^2F}{\partial y^i\partial y^j}$$
and we obtain with $y^j=\dfrac{dx^j}{d\tau}$
$$E_i\left(C\right)=F_{ij}(x,y)\left(\dfrac{dy^j}{d\tau}\right)+\bigg\{\left(\dfrac{\partial F_i}{\partial x^j}\right)y^j-\dfrac{\partial F}{\partial x^i}\bigg\}.$$
Consequently if $\det(F_{ij})\not=­0$, the $E_i\left(C\right)$ are $n$ independent equations and can be rewritten in normal form. If $\det(F_{ij})­\not=0$ then the variational problem is called regular.

\begin{ex}We consider now the Weirstrass invariant in dimension $2$. We set
$$f(x,y,\dot{x},\dot{y})=f(x,y,p,q),\quad p=\dot{x},\\ q=\dot{y}.$$
The Euler-Lagrange equations are
$$\displaystyle\begin{cases}
     f_x-\dfrac{d}{dt}f_p &=0 \\\\
     f_y-\dfrac{d}{dt}f_q & =0.
\end{cases}$$
These equations are not in general independent. Now $f$ is $1$-homogeneous in $p$, $q$ so the derivatives $f_p$ and $f_q$ are $0$-homogeneous in $p$, $q$. By Euler relations
$$\displaystyle\begin{cases}
     pf_{pp}+qf_{pq} &=0 \\\\ 

     pf_{qp}+qf_{qq} & =0
     \end{cases}\qquad\Longrightarrow\dfrac{f_{pp}}{q^2}=-\dfrac{f_{qq}}{p^2}=-\dfrac{f_{pq}}{pq}=:W.$$
     $W$ is called the Weirstrass invariant. Since $f_x$, $f_y$ are still $1$-homogeneous in $p$, $q$. We have (Euler relations)
$$\displaystyle\begin{cases}
     pf_{xp}+qf_{xq} &=f_x \\\\ 

     pf_{yp}+qf_{yq} & =f_y.\end{cases}$$\end{ex}
\begin{prop}
The extremal equations reduce to a single equation
\begin{equation}f_{xq}-f_{yq}+(p\dot{q}-\dot{p}q)W\end{equation}
\end{prop}
This is Weirstrass form of the Euler-Lagrange equations. We can observe that $W(x,y,p,q)$ is (-$3)$-homogeneous in $p$, $q$. Because $W=\dfrac{f_{pp}}{q^2}$ and $f_{pp}$ is (-$1)$-homogeneous.     
\section{Finsler Metric}
Let $\mathbf{M}^n$ be a smooth $n$-dimensional manifold and $T_x\mathbf{M}^n$ its tangent space at $x\in \mathbf{M}^n$. Let $\pi:T\mathbf{M}\to \mathbf{M}^n$ be its tangent bundle: $\pi(y)=x$ for $y\in T_x\mathbf{M}$. Consider a local coordinate system $(U,(x^i))$ on $U\subset \mathbf{M}^n$. A tangent vector $y$ at $x=(x^i)$ is written as $\displaystyle\sum_{i=1}^ny^i\dfrac{\partial}{\partial x^i}$. Thus a coordinate system on $T\,\mathbf{M}$ is $\{\pi^{-1}(U),(x^i,y^i)\}$ and $T\,\mathbf{M}$ is a $2n$-dimensional manifold. Consider two non-emptily intersecting coordinate systems on $\mathbf{M}^n$ $(U,(x^i))$ and $(\overline{U},(\overline{x}^r))$: $U\cap\overline{U}­\not=\emptyset$, $\overline{x}^r=\overline{x}^r(x^i)$.

Define $\overline{X}^r_i:=\dfrac{\partial\overline{x}^r}{\partial x^i}$ and $\underline{X}_r^i=\dfrac{\partial x^i}{\partial \overline{x}^r}$. 
So between $(U,(x^i))$ and $(\overline{U},(\overline{x}^r))$ we have
$$\displaystyle\sum_r\overline{y}_r\dfrac{\partial}{\partial\overline{x}^r}=\sum_r\overline{y}_r\left(\sum_i\dfrac{\partial x^i}{\partial \overline{x}^r}\cdot\dfrac{\partial }{\partial x^i}\right)=\sum_iy^i\dfrac{\partial}{\partial x^i}$$
with $\displaystyle y^i=\sum_r\overline{y}^r\dfrac{\partial x^i}{\partial \overline{x}^r}$.
\begin{defi}
A tensor field of type $(1,2)$ on $\mathbf{M}^n$ is a collection of $n^3$ smooth functions $T_{jk}^i(x)$ (the components of T) in each chart $U$ such that if $U\cap\overline{U}\not=­\emptyset$
$$\overline{T}^r_{st}=T^i_{jk}\overline{X}^r_i\underline{X}_s^j\underline{X}_t^k$$
in Einstein's notation.
\end{defi}
We extend this definition to the case of the tangent bundle. We consider smooth functions $T_{jk}^i(x,y)$ in each coordinate neighborhood $\pi^{-1}(U)$. We then obtain a vectorial tensor field (spray tensor field). Differentiating the relation of definition of the tensor field with respect to $\overline{y}^q$ we obtain
$$\dfrac{\partial}{\partial\overline{y}^q}\overline{T}_{st}^r=\dfrac{\partial}{\partial y^h}T^i_{jk}\overline{X}_i^r\underline{X}_s^j\underline{X}_t^k\left(\dfrac{\partial y^h}{\partial\overline{y}^q}\right)=\left(\dfrac{\partial}{\partial y^h}T^i_{jk}\right)\overline{X}_i^r\underline{X}_s^j\underline{X}_t^k\underline{X}_q^h.$$
This shows that the partial derivative with respect to $y^i$ of the components of a vectorial tensor field of type $(1,2)$ gives a vectorial tensor field of type $(1,3)$.
\begin{rem}
One may define analogously vectorial tensors of type $(r,s)$. The partial derivative with respect to $y^i$ of a vectorial tensor of type $(r,s)$ is a vectorial tensor of type $(r,s+1)$.
\end{rem}
Assume that a vectorial scalar field $L(x,y)$ is given in $\mathbf{M}^n$. We can then define the length of a curve
$$\displaystyle x^i:t\in I=\left[a,b\right]\mapsto x^i(t).$$
It is given by
$$\displaystyle s=\int_a^bL(x(t),\dot{x}(t))dt.$$
\begin{defi}
\label{g}
A smooth $n$-dimensional manifold $\mathbf{M}^n$ equipped with such a notion of length is called a Finsler space with metric $L$ if
\begin{enumerate}
  \item The length does not depend on the parametrization $L(x,y)$ and $L$ is $1$-homogenous in $y$.
  \item The length integral gives rise to a regular variational problem, that is the Finslerian spray field $(g_{ij})$ defined by
  $$\displaystyle g_{ij}:=\dot{\partial}_i\dot{\partial}_jF=\dfrac{\partial^2F}{\partial y^i\partial y^j}$$ 
  has a non-vanishing determinant.
  Here $F=\dfrac{1}{2}L^2$, $\partial_i=\dfrac{\partial}{\partial x^i}$, $\dot{\partial}_i=\dfrac{\partial}{\partial y^i}$.
We denote by $(g^{ij}):=(g_{ij})^{-1}$. We observe that $(g_{ij})$ and $(g^{ij})$ are $0$-homogeneous as $F$ is $2$-homogeneous. We define
$$\displaystyle y_i:=g_{ij}y^j=\sum_{j}g_{ij}(x,y)y^j=\dot{\partial}_iF=L(\dot{\partial}_iL)$$
because $\dot{\partial}_iF$ is $1$-homogeneous and of the Euler's theorem. We moreover have
$$\displaystyle\sum_{i,j}g_{ij}(x,y)y^iy^j=2F(x,y)=L^2(x,y)$$
due to the $2$-homogeneity of $F$.
If $g_{ij}(x,y)$ is continuous at $y=0$, then $g_{ij}(x,y)$ is constant in $y$, i.e. it depends on $x$ alone and the previous equation gives
$$\sum_{i,j}g_{ij}(x)y^iy^j=2F(x,y)=L^2(x,y)$$   
that is a Riemannian metric.
\item In each tangent space $\mathbf{M}_x:=T_x\mathbf{M}$ we need a region $\mathbf{M}_x^\star$ such that $L(x,y)$ is differentiable in $y^i$ on $\mathbf{M}_x^\star$, where $\mathbf{M}_x^\star$ does not contain $y=0$ and is a positively conical region ($y\in \mathbf{M}_x^\star$ and $\lambda>0\Longrightarrow \lambda y\in \mathbf{M}_x^\star$).
In this situation we consider $\displaystyle \mathbf{M}^\star=\bigcup_{x\in \mathbf{M}}\mathbf{M}_x^\star$ as the union of the domains of definition of $L$. 

We normalize by setting $l^i=\dfrac{y^i}{L(x,y)}$, then $\displaystyle\sum g_{ij}l^il^j=1$ ($l=(l^i)$ is a unit vector with respect to the fundamental tensor $g=(g_{ij})$) and
$$L(x,l)=L\left(x,\dfrac{y}{L(x,y)}\right)=\dfrac{L(x,y)}{L(x,y)}=1.$$
\item The fundamental function $L(x,y)$ is positive for each $x$ and for each $y\in \mathbf{M}_x^\star$.
\end{enumerate}
\end{defi}
\section{Indicatrix}
The extremals of the integral $\displaystyle s=\int_a^bL(x(t),\dot{x}(t))dt$ are called the geodesics of the Finsler space. They are given by
\begin{equation}\label{30003}\displaystyle\dfrac{d^2x^i}{ds^2}+2G^i\left(x,\dfrac{dx}{ds}\right)=0\end{equation}
where $s$ is the normalized parameter (arc length) with
$$\displaystyle G^i:=\sum g^{ij}G_j$$
and 
$$\displaystyle 2G_j=\sum_i(\dot{\partial}_j\partial_iF)y^i-\partial_jF.$$
We can define the Christoffel symbols
\begin{equation}\label{y}\gamma_{ijk}(x,y)=\sum g_{ih}\gamma_{jk}^h=\dfrac{1}{2}\left(\partial_kg_{ij}+\partial_jg_{ik}-\partial_ig_{jk}\right).\end{equation}
\begin{theo}
\label{z}
$$2G^i(x,y)=\gamma_{jk}^iy^ky^j.$$
\end{theo}
\subsection{Formulas in Riemannian Geometry}
We recall that the Christoffel symbols are given by:
$$\Gamma_{kij}=\dfrac{1}{2}\left(\dfrac{\partial}{\partial x_j}g_{ki}+\dfrac{\partial}{\partial x_i}g_{kj}-\dfrac{\partial}{\partial x_k}g_{ij}\right)=\dfrac{1}{2}\left(g_{ki,j}+g_{kj,i}-g_{ij,k}\right).$$
\begin{equation}\begin{split}\Gamma^m_{ij}&=g^{mk}\Gamma_{kij}=\dfrac{1}{2}g^{mk}\left(\dfrac{\partial}{\partial x_j}g_{ki}+\dfrac{\partial}{\partial x_i}g_{kj}-\dfrac{\partial}{\partial x_k}g_{ij}\right)\\
&=\dfrac{1}{2}g^{mk}\left(g_{ki,j}+g_{kj,i}-g_{ij,k}\right).
\end{split}
\end{equation}
Here $g^{ij}$ is the inverse matrix of the metric tensor $g_{ij}$:
$$\delta_j^i=g^{ik}g_{kj}$$
and thus if $n=\dim \mathbf{M}$
$$n=\delta^i_i=g_i^i=g^{ij}g_{ji}.$$
Everything comes from the relation
\begin{equation}
\label{3000}
Z(x,ky)=k^rZ(x,y)\Longrightarrow y^i\dfrac{\partial Z(x,y)}{\partial y^i}=rZ(x,y)
\end{equation}
The application of \eqref{3000} to $L^2$ yields
\begin{equation}
\label{3001}
L^2(x,y)=g_{ij}y^iy^j
\end{equation}
where
\begin{equation}
\label{3002}
g_{ij}(x,y)=\dfrac{1}{2}\dfrac{\partial^2L^2(x,y)}{\partial y^i\partial y^j}
\end{equation}
is called the Finsler metric tensor.
Introducing the covariant tangent vector
\begin{equation}
\label{3003}
y_i=g_{ij}(x,y)y^j
\end{equation}
we obtain from \eqref{3003} and the Euler  \theoref{3000} with $r=1$ that 
\begin{equation}
\label{3004}
y_i=L(x,y)\dfrac{\partial L(x,y)}{\partial y^i}.
\end{equation} 
To show this we remark since $\dfrac{\partial}{\partial y^i}L^2$ is $1$-homogeneous that
\begin{equation}
\label{3005}
\begin{split}
\displaystyle&\dfrac{1}{2}\sum_jy^j\dfrac{\partial}{\partial y^j}\left(\dfrac{\partial}{\partial y^i}L^2\right)=\dfrac{1}{2}\dfrac{\partial}{\partial y^i}L^2=L\dot{L}\\
&=\sum g_{ij}y^j=y_i.
\end{split}
\end{equation}
Consider the tangent vector
\begin{equation}
\label{3006}
\begin{split}
l_i&=\dfrac{y_i}{L(x,y)}\\
&=\dfrac{\partial L(x,y)}{\partial y^i}.
\end{split}
\end{equation}
The definitions \eqref{3002} and \eqref{3003} imply that
\begin{equation}
\label{3007}
g_{ij}=\dfrac{\partial y_i}{\partial y^j}
\end{equation}
whereas
\begin{equation}
\label{3008}
\dfrac{\partial l_i}{\partial y^j}=\dfrac{h_{ij}(x,y)}{L(x,y)}
\end{equation}
where
\begin{equation}
\label{3009}
h_{ij}(x,y)=g_{ij}(x,y)-l_i(x,y)l_j(x,y).
\end{equation}
In fact:
$$\dfrac{\partial}{\partial y^j}l_i=\dfrac{\partial}{\partial y^j}\dfrac{y_i}{L}=\dfrac{g_{ij}}{L}-\dfrac{l_il_j}{L}.$$ 
\begin{defi}
$(h_{ij})$ is called the angular metric tensor.
\end{defi}
We have
\begin{equation}
\label{3010}
y^ih_{ij}=0\quad g^{ij}h_{ij}=n-1.
\end{equation}
Indeed
\begin{equation}
\begin{split}
y^ih_{ij}&=y^ig_{ij}-y^il_il_j=y_j-l_jy^il_i=y_j-l_jy^i\dfrac{\partial L}{\partial y^i}\\
&=y_j-l_jL\\
&=y_j-y_j\\
&=0
\end{split}
\end{equation}
and
$$g^{ij}h_{ij}=g^{ij}g_{ij}-g^{ij}l_il_j=n-1.$$
As we observed a Finsler geometry reduces to a Riemannian geometry in the case when the metric tensor $g_{ij}(x,y)$ is assumed to be independent of $y^i$. This reads 
\begin{equation}
\label{3011}
C_{ijk}:=\dfrac{1}{2}\dfrac{\partial g_{ij(x,y)}}{\partial y^k}=0.
\end{equation}
This entity is a tensor because differentiation of a tensor with respect to a vector argument gives again a tensor, in contrast to differentiation with respect to the coordinates $x_i$. We have from \eqref{3002} and \eqref{3011}
\begin{equation}
\label{3012}
C_{ijk}=\dfrac{1}{4}\dfrac{\partial^3 L^2}{\partial y^i\partial y^j\partial y^k}
\end{equation}
which says that this tensor is symmetric in all its subscripts.
\begin{defi}
$C_{ijk}(x,y)$ is called the Cartan torsion tensor.
\end{defi}
The Euler theorem indicates that ($r=0$)
\begin{equation}
\label{3013}
y^iC_{ijk}(x,y)=0.
\end{equation}
Similarly the tensor 
\begin{equation}
\label{3014}
C_{ijkm}(x,y)=\dfrac{\partial C_{ijk(x,y)}}{\partial y^m}
\end{equation}
is also symmetric in all its subscripts and satisfies
\begin{equation}
\label{3015}
y^iC_{ijkm}(x,y)=-C_{jkm}.
\end{equation}
The contravariant reciprocal tensor $g^{ij}(x,y)$ of the covariant metric tensor \eqref{3002} is such that
\begin{equation}
\label{3016}
g^{ki}(x,y)g_{ij}(x,y)=\delta_j^k.
\end{equation}
The differentiation with respect to $y^n$ yields:
\begin{equation}
\label{3017}
g^{ki}\dfrac{\partial g_{ij}}{\partial y^n}+g_{ij}\dfrac{\partial g^{ki}}{\partial y^n}=0.
\end{equation}
But \eqref{3011} shows that
\begin{equation}
\label{3018}
\dfrac{1}{2}\dfrac{\partial g^{ki}(x,y)}{\partial y^n}=-C_n^{ki}
\end{equation}
with 
$$C^{ni}_m=C_{mkj}g^{nk}g^{ij},\quad C^i_{jkm}=C_{njkm}g^{in},\quad l^i=g^{ij}l_j$$
To see \eqref{3018}
\begin{equation}
\begin{split}
C_n^{ki}=C_{nrs}g^{kr}g^{is}&=\dfrac{1}{2}g^{kr}g^{is}\dfrac{\partial g_{rs}}{\partial y^n}=-\dfrac{1}{2}g^{is}g_{rs}\dfrac{\partial g^{kr}}{\partial y^n}\\
&=-\dfrac{1}{2}\delta_r^i\dfrac{\partial g^{kr}}{\partial y^n}=-\dfrac{1}{2}\dfrac{\partial g^{ki}}{\partial y^n}.
\end{split}
\end{equation}
The geodesics of a Finsler space are the stationary curves of the functional
\begin{equation}
\label{3019}
\displaystyle I\left(C\right)=\int_{p_1}^{p_2}F(x,dx)=\int_a^bF(x,\dot{y})dt.
\end{equation}
The natural parameter $\tau$ of the (Finslerian) arc-length is defined by
\begin{equation}
\label{3020}
d\tau=L(x,dx)
\end{equation}
and the homogeneity of $F(x,y)$ shows that (from \theoref{carath})
$$\displaystyle\int L(x,dx)=\int L(x,\dfrac{dx}{d\tau})d\tau=\int L(x,\dfrac{dx}{dt})dt$$
for any parameter $t=t(\tau)$ with $\dfrac{dt}{d\tau}­\not=0$.

\noindent The Euler-Lagrange equations take the form
\begin{equation}
\label{3021}
\dfrac{d}{dt}\left(\dfrac{\partial F(x,\dot{x})}{\partial \dot{x}^i}\right)-\dfrac{\partial F(x,\dot{x})}{\partial x^i}=0
\end{equation}
and Finslerian geodesics satisfy the Euler-Lagrange equations associated with the Lagrangian $L(x,\dot{x})$.

\noindent The equations of the geodesics can be written as
\begin{equation}
\label{3022}
\dfrac{d^2x^i}{dt^2}+\gamma_{mn}^i\dot{x}^m\dot{x}^n-\dot{x}^i\dfrac{d(\ln L(x,\dot{x}))}{dt}=0
\end{equation}
where
$$\gamma_{mn}^i(x,y)=\dfrac{1}{2}g^{ik}(x,y)\left(\dfrac{\partial g_{mk}(x,y)}{\partial x^n}+\dfrac{\partial g_{nk}(x,y)}{\partial x^m}-\dfrac{\partial g_{mn}(x,y)}{\partial x^k}\right)$$
are the Finslerian Christoffel symbols.  We are going to confirm this: we use \eqref{3001}: $L^2=g_{ij}(x,y)y^iy^j,\,(y^i=\dot{x}^i)$ in the Euler-Lagrange equation \eqref{3021} to get
$$2L\dfrac{\partial L}{\partial \dot{x}^i}=2g_{ij}y^j\iff\dfrac{\partial L}{\partial \dot{x}^i}=\dfrac{g_{ij}(x,y)y^j}{L}.$$
In the same way
$$2L\dfrac{\partial L}{\partial x^i}=y^my^n\dfrac{\partial g_{mn}}{\partial x^i}$$
so that the Euler-Lagrange equations become:
$$\dfrac{d}{dt}\left(\dfrac{g_{ij}(x,\dot{x})\dot{x}^j}{L(x,\dot{x})}\right)-\dfrac{1}{2}L^{-1}(x,\dot{x})\dot{x}^m\dot{x}^n\dfrac{\partial g_{mn}(x,\dot{x})}{\partial x^i}=0.$$
Multiplying by $L$ we arrive at
$$\displaystyle g_{ij}\{\ddot{x}^j-\dot{x}^j\dfrac{d}{dt}\log L\}+\dot{x}^j\dfrac{d}{dt}g_{ij}(x,\dot{x})-\dfrac{1}{2}\dot{x}^m\dot{x}^n\dfrac{\partial g_{mn}}{\partial x^i}(x,\dot{x})=0.$$
This gives using \eqref{3011}
\begin{equation}
\label{3023}
\displaystyle g_{ij}\{\ddot{x}^j-\dot{x}^j\dfrac{d}{dt}\log L\}+\dot{x}^j\dot{x}^n\dfrac{\partial g_{ij}(x,\dot{x})}{\partial x^n}g_{ij}-\dfrac{1}{2}\dot{x}^m\dot{x}^n\dfrac{\partial g_{mn}}{\partial x^i}(x,\dot{x})=0.\end{equation}
Raising the index $i$ in \eqref{3023} and noting that
$$\dot{x}^m\dot{x}^n\gamma_{min}(x,\dot{x})=\dot{x}^m\dot{x}^n\left(\dfrac{\partial g_{im}(x,\dot{x})}{\partial x^n}-\dfrac{1}{2}\dfrac{\partial g_{mn}(x,\dot{x})}{\partial x^i}\right)$$
we obtain the desired relation
$$\dfrac{d^2x^i}{dt^2}+\gamma_{mn}^i\dot{x}^m\dot{x}^n-\dot{x}^i\dfrac{d(\log L(x,\dot{x}))}{dt}=0.$$
In the case where the parameter $t$ is chosen to be the Finslerian arc-length $\tau$ defined by \eqref{3020} we find that if
\begin{equation}
\label{3024}
x^{\prime i}=\dfrac{dx^i}{d\tau}
\end{equation}
then $L(x,x^\prime)=1$ and the equation of the geodesics reduces to
\begin{equation}
\label{3025}
\dfrac{d^2x^i}{d\tau^2}+\gamma_{mn}^i(x,x^\prime)x^{\prime m}x^{\prime n}=0.
\end{equation}
\section{Review of Tensor calculus}
We shall consider some facts on tensors. Let $\mathbf{E}$ be a finite dimensional real vector of dimension $n$. For us $\mathbf{E}$ will typically be the tangent space at $x\in \mathbf{M}$ to a differentiable manifold $\mathbf{M}$: $\mathbf{E}=T_x\mathbf{M}$. We denote a basis of $E$ by $\partial=(\partial_1,\cdots,\partial_n)$, with its dual basis given by $\sigma=(dx^1,\cdots,dx^n)$.
\begin{defi}
A covariant tensor of rank $r$ is an $r$-multilinear real-valued function
$$Q:\mathbf{E}\times \cdots \times \mathbf{E}\to \R.$$
\end{defi}  
We emphasize that the values of this function must be independent of the basis in which the components are expressed.

\noindent A covariant vector is covariant tensor of rank $1$. The most well-known second rank tensor is the metric tensor $g=(g_{ij})$
$$\displaystyle g(v,w)=\langle u,w\rangle=\sum_{i,j}g_{ij}v^iw^j.$$
In components we have
\begin{equation}
\label{3026}
\begin{split}
\displaystyle Q(v_1,\cdots,v_r)&=Q(\sum_{i_1}v_1^{i_1}\partial_{i_1},\cdots,\sum_{i_r}v_r^{i_r}\partial_{i_r})\\
&=\sum_{i_1\cdots i_r}v_1^{i_1}\cdots v_{r}^{i^r}Q(\partial_{i_1},\cdots,\partial{i_r})
\end{split}
\end{equation}
that is
$$\displaystyle Q(v_1,\cdots,v_r)=\sum_{i_1,\cdots,i_r}Q_{i_1\cdots i_r}v_1^{i_1}\cdots v_r^{i_r},\, Q_{i_1\cdots i_r}=Q(\partial_{i_1},\cdots,\partial_{i_r}).$$ 
The collection of all covariant tensors of rank $r$ forms a vector space of dimension $n^r$. It is denoted by 
$$\displaystyle \mathbf{E}^\star\otimes \mathbf{E}^\star\otimes\cdots\otimes \mathbf{E}^\star=\bigotimes^r\mathbf{E}^\star.$$
If $\alpha,\,\beta$ are two covectors, that is elements of $\mathbf{E}^\star$ we can form a second-rank covariant tensor, the tensor product of $\alpha$ and $\beta$, by
\begin{equation}
\begin{split}
\alpha\otimes\beta&:\mathbf{E}\otimes\mathbf{E}\to \R\\
&\alpha\otimes\beta(v,w):=\alpha(v)\beta(w)
\end{split}
\end{equation}
A contravariant vector that is an element of $\mathbf E$, can be considered as linear functional on covectors by defining
$$v(\alpha):=\alpha(v).$$
\begin{defi}
A contravariant tensor of rank $s$ is a multilinear real-valued function on $s$-tuples of covectors
$$T:\mathbf{E}^\star\times\mathbf{E}^\star\times\cdots\times\mathbf{E}^\star\to\R$$ 
so that 
$$T(\alpha_1,\cdots,\alpha_s)=a_{1i_1}\cdots a_{si_s}T^{i_1,\cdots,i_r}$$
where
$$T^{i_1\cdots i_s}=T(dx^{i_1},\cdots,dx^{i_s}).$$
\end{defi} 
The space of contravariant tensors is
$$\displaystyle \mathbf{E}\otimes \mathbf{E}\otimes\cdots\otimes \mathbf{E}=\bigotimes^s\mathbf{E}.$$
Contravariant vectors are contravariant tensors of rank $1$. 

\noindent We can in the same way define mixed tensors $W:\displaystyle \otimes^s\mathbf{E}^\star\otimes\otimes^r\mathbf{E}\to\R$ of $s$-tuples of covectors and $r$-tuples of vectors
$$W(\alpha_1,\cdots,\alpha_s,v_1,\cdots,v_r)=a_{1i_1}\cdots a_{si_s}W^{i_1\cdots i_s}_{j_1\cdots j_r}v_1^{j_1}\cdots v_r^{j_r};$$
with $W^{i_1\cdots i_s}_{j_1\cdots j_r}=W(dx^{i_1},\cdots,\partial_{j_r}).$

\noindent We have seen that 
\begin{equation}
\label{3028}
C_{ijk}(x)=g_{ih}C^h_{jk}=\dfrac{1}{2}\left(\dot{\partial_k}g_{ij}+\dot{\partial_j}g_{ik}-\dot{\partial_i}g_{jk}\right)=\dfrac{1}{2}\dot{\partial_k}g_{ij}.
\end{equation}
$C_{jik}$ is symmetric in $i,\,j,\,k$ and $C_{ijk}(x,y)y^k=0$.

\noindent Each tangent space $\mathbf{M}_x$ may be regarded as an $n$-dimensional Riemannian space with the Riemann metric
$$g_{ij}(x,y)dy^idy^j$$ 
where $x=(x^i)$ is fixed. Then the $C^i_{jk}$ are the Christoffel symbols of $\mathbf{M}_x$. This a very special type of Riemannian space since the $\gamma_{ijk}$ of a general Riemannian space are not symmetric in all indices.

\noindent For $y$ with $L(x,y)>0$ we have $L(x,l)=1$, $l^i=\dfrac{y^i}{L(x,y)}$.

\noindent The set
$$I_x=\{y\in \mathbf{M}_x\,/\,L(x,y)=1\}$$
is called the indicatrix at the point $x$. It is a hypersurface of $M_x$ written in parametric form as
$$y^i=y^i(u^\alpha),\,\alpha=1,\cdots,n-1.$$
$I_x$ has the induced Riemannian metric 
$$g_{\alpha\beta}(u)=g_{ij}(x,y(u))B_\alpha^iB_\beta^j,\,\,B_\alpha^i=\dfrac{\partial y^i}{\partial u^\alpha}.$$
Differentiating $L(x,y(u))=1$ with respect to $u^\alpha$ we get 
$$l_iB_\alpha^i=0$$
which implies that $l_i$ are covariant components of the normal vector of $I_x$. In terms of the angular metric
$$h_{ij}(x,y)=g_{ij}-l_il_j=L(\dot{\partial}_jl_i)$$
the induced metric may be written 

\begin{equation}\label{30001}g_{\alpha\beta}(u)=h_{ij}(x,y(u))B_\alpha^iB_\beta^j.\end{equation}
Thus the Riemannian metric $g_{\alpha\beta}$ of $I_x$ may be regarded as the induced metric from the angular metric tensor $h$.

\noindent We have seen that
$$h_{ij}(x,y)y^i=0$$
so the rank $r$ of the matrix $(h_{ij})$ must be less than $n$. As 
$$\rank((g_{ij})-(l_il_j))³|\rank(g_{ij}-\rank(l_il_j))|=n-1,$$
we have $\rank((g_{ij})-(l_il_j))=n-1$. Hence
\begin{prop}
The components $h_{ij}$ of the angular metric tensor $h$ constitute a matrix of rank $n-1$. A non-trivial solution $(v^i)$ to the system of linear equations 
$$h_{ij}v^j=0,\,i=1,\cdots,n$$
must be proportional to $y^i$.
\end{prop}
Consider the Taylor expansion of the function $L^2(x,y)$ at a point $y_0\in I_x$:
$$L^2(x,y)=L^2(x,y_0)+\left[\dot{\partial}_iL^2\right](y^i-y_0^i)+\left[\dfrac{\dot{\partial}_i\dot{\partial}_jL^2}{2}\right](y^i-y_0^i)(y^j-y_0^j)+\cdots$$
where $\left[\cdots\right]$ means that $y=y_0$ is assumed. From the homogeneity of $L$ we have
$$\left[\dot{\partial_i}L^2\right]y^i_0=2L^2(x,y_0)$$
and
$$\left[\dot{\partial_i}\dot{\partial_j}L^2\right]l^i_0=\dot{\partial_j}L^2(x,y_0)$$
so that:
$$L^2(x,y)=g_{ij}(x,y_0)y^iy^j+\cdots.$$
For $y_0\in I_x$, the set
$$I_x(y_0)=\big\{y\,/\,g_{ij}(x,y_0)y^iy^j=1\big\}$$
is called the osculating indicatrix at $y_0$. Therefore the indicatrix $I_x$ is osculating with the osculating indicatrix $I_x(y_0)$ at each point $y_0$.
\begin{rem}
In Riemannian geometry, the geodesic $X(t)$ starting at the origin with initial speed $(v^i)$ has Taylor expansion in the chart
$$X^i(t)=tv^i-\dfrac{1}{2}\Gamma_{jk}^iv^jv^k+O(t^3).$$
\end{rem}
We would like to give some important Finsler spaces but before we proceed, we first recall the classical definition of the covariant derivative.

\noindent The covariant derivative is a generalization of the classical notion of directional derivative: $\nabla_uv$, where $u$ is a vector defined at a point $p$ of a vector space and $v$ a vector defined in a neighborhood of $p$.

\noindent Suppose first that a Riemann manifold $\mathbf{M}$ is embedded into the Euclidean space $(\R^n,\langle\cdot\,;\cdot\rangle)$ by $\psi:U\subset \mathbf{M}\to \R^n$ such that the tangent space at $\psi\left(p\right)\in M\subset\R^n$ is spanned by the vectors
$$\dfrac{\partial\psi}{\partial x_i}_{|_p},\,1²i²d.$$
The metric is
$$g_{ij}=\langle\dfrac{d\psi}{dx^i},\dfrac{d\psi}{dx^j}\rangle.$$
Consider a vector field $v=v^j\dfrac{\partial\psi}{\partial x^j}$. One has
$$\dfrac{\partial v}{\partial x^i}=\dfrac{\partial v^j}{\partial x^i}\dfrac{\partial\psi}{\partial x^j}+v^j\dfrac{\partial^2\psi}{\partial x^i\partial x^j}.$$
The last term is not tangent to $\mathbf{M}$ but we can write 
$$\dfrac{\partial^2\psi}{\partial x^i\partial x^j}=\Gamma_{ij}^k\dfrac{\partial\psi}{\partial x^k}+n$$
with $n$ is a vector normal to the tangent space.

\noindent The covariant derivative $\nabla_{e_i}v$ is just the tangential part of $\dfrac{\partial v}{\partial x^i}$:
$$\nabla_{e_i}v=\dfrac{\partial v}{\partial x^i}-n=\left(\dfrac{\partial v^k}{\partial x^i}+v^j\Gamma^k_{ij}\right)\dfrac{\partial\psi}{\partial x^k}.$$
Now
$$\langle\dfrac{\partial^2\psi}{\partial x^i\partial x^j},\dfrac{\partial \psi}{\partial x^l}\rangle=\Gamma_{ij}^k\langle\dfrac{\partial\psi}{\partial x^k},\dfrac{\partial\psi}{\partial x^l}\rangle=\Gamma_{ij}^kg_{kl}.$$ 
On the other hand 
$$\dfrac{\partial g_{ab}}{\partial x^c}=\langle\dfrac{\partial^2\psi}{\partial x^c\partial x^a},\dfrac{\partial\psi}{\partial x^b}\rangle+\langle\dfrac{\partial\psi}{\partial x^a},\dfrac{\partial^2\psi}{\partial x^c\partial x^b}\rangle$$
which implies, using the symmetry of the scalar product and swapping the order of partial differentiation
$$\dfrac{\partial g_{jk}}{\partial x^i}+\dfrac{\partial g_{ki}}{\partial x^j}-\dfrac{\partial g_{ij}}{\partial x^k}=2\langle\dfrac{\partial^2\psi}{\partial x^i\partial x^j},\dfrac{\partial\psi}{\partial x^k}\rangle$$
or 
$$g_{kl}\Gamma_{ij}^k=\dfrac{1}{2}\left(\dfrac{\partial g_{jl}}{\partial x^i}+\dfrac{\partial g_{li}}{\partial x^j}-\dfrac{\partial g_{ij}}{\partial x^l}\right).$$

\noindent For the general case: Given coordinates $x^i,i=0,1,2,\cdots$, any tangent vector can be described by its components in the basis: $e_i=\dfrac{\partial}{\partial x^i}$. The covariant derivative of a basis vector is again a vector and we can write
$$\nabla_{e_i}e_j=\Gamma^k_{ij}e_k.$$
We require that: a covariant derivative $\nabla$ at a point $p$ of a smooth manifold $\mathbf{M}$ assigns a tangent vector $(\nabla_vu)_p$ to each pair consisting of a tangent vector $v$ at $p$ and a vector field $u$ defined in a neighborhood of $p$ such that
\begin{itemize}
\item $(\nabla_{gx+hy}u)_p=g\left(p\right)(\nabla_{x}u)_p+h\left(p\right)(\nabla_{y}u)_p$
\item $(\nabla_{v}(u+w))_p=(\nabla_{v}u)_p+(\nabla_{v}w)_p$
\item $(\nabla_{v}(fu))_p=f\left(p\right)(\nabla_{v}u)_p+(\nabla_{v}f)_pu_p$ with $\displaystyle (\nabla_{v}f)=\lim_{t\to 0}t^{-1}(f(\varphi(t))-f\left(p\right))$ where $\varphi:\left[-1,1\right]\to \mathbf{M},\,\varphi(0)=p,\,\varphi^\prime(0)=v$.
\end{itemize}
In coordinates if $v=v^ie_i$ and $u=u^je_j$ we get
$$\nabla_vu=\nabla_{v^ie_i}(u^je_j)=\left( v^iu^j\Gamma_{ij}^k+v^i\dfrac{\partial u^k}{\partial x^i}\right)e_k$$
(take the curves $\varphi^i:\left[-1,1\right]\to M,\,t\mapsto p+te_i$). The covariant derivative is the usual derivative along the coordinates with correction terms which tell how the coordinates change. We introduce the notation
$$\nabla_{e_j}v=v_{;j}^se_s,\, v_{;j}^s=\dfrac{\partial v^i}{\partial x^j}+v^k\Gamma_{kj}^i.$$
The covariant derivative of a type $(r,s)$ tensor field along $e_c$ is given by
\begin{equation}
\begin{split}
(\nabla_{e_c}T)^{a_1\cdots a_r}_{b_1\cdots b_s}&=\dfrac{\partial}{\partial x_c}T^{a_1\cdots a_r}_{b_1\cdots b_s}+\Gamma_{dc}^{a_1}T^{da_2\cdots a_r}_{b_1\cdots b_s}+\cdots+\Gamma_{dc}^{a_r}T^{a_1\cdots a_{r-1}d}_{b_1\cdots b_s}\\
&-\Gamma_{b_1c}^{d}T^{a_1\cdots a_r}_{db_2\cdots b_s}-\cdots-\Gamma_{b_sc}^{d}T^{a_1\cdots a_r}_{b_1b_2\cdots b_{s-1}d}.
\end{split}
\end{equation}
\begin{rem}
The covariant derivative of the (covariant) metric tensor is zero:
$$g_{mn;k}=\dfrac{\partial g_{mn}}{\partial x^k}-\Gamma_{mk}^sg_{sn}-\Gamma^s_{nk}g_{ms}=0.$$ 
\end{rem}
Indeed we have
$$g^{\alpha\nu}g_{\nu\beta}=\delta_\beta^\alpha.$$
Let us take any derivative $\delta$ of this identity; we get:
$$\delta g^{ij}g_{jl}+g^{ij}\delta g_{jl}=0.$$
Multiplying this relation by $g^{rl}$ we have:
$$\delta g^{ij}g_{jl}g^{rl}=\delta g^{ij}g_j^r=-g^{rl}g^{ij}\delta g_{jl}.$$
We now make the substitution $r\to j,\,l\to n,\,j\to m$ and obtain
\begin{equation}
\label{3029}
\delta g^{ij}=-g^{jn}g^{im}\delta g_{mn}.
\end{equation}
Besides it can be shown that
\begin{equation}
\label{3030}
\dfrac{\partial g^{ij}}{\partial x^k}+g^{lj}\Gamma_{lk}^i+g^{il}\Gamma_{lk}^j=0
\end{equation}
with
\begin{equation}
\label{3031}
g^{lj}\Gamma_{lk}^i=g^{jn}(\delta^i_s)\Gamma_{nk}^s=g^{jn}g^{im}g_{ms}\Gamma_{nk}^s
\end{equation}
\begin{equation}
\label{3032}
g^{il}\Gamma_{lk}^j=g^{im}(\delta^j_s)\Gamma_{mk}^s=g^{im}g^{jn}g_{ns}\Gamma_{mk}^s
\end{equation}
Substituting \eqref{3029}, \eqref{3031} and \eqref{3032} into \eqref{3030} we arrive at
$$\dfrac{\partial g_{mn}}{\partial x^k}=\Gamma_{mk}^sg_{sn}+\Gamma^s_{nk}g_{ms}\iff g_{mn;k}=0.$$
\subsection{Locally Minkowski space}
Let $g_{ij}$ be the fundamental tensor of an $n$-dimensional Finsler space $\mathbf F^n=(\mathbf M^n,L(x,y))$. We have two extreme cases: $g_{ij}$ depends on $x$ alone and on $y$ alone. In the first case we have a Riemannian space  and from equation \eqref{3011} we have $C_{ijk}=0$. Since $C_{ijk}$ are the components of the Cartan tensor, the dependence on $x$ alone is independent of the choice of coordinate system.

On the other hand, the dependence on $y$ alone is characterized by $\partial_k g_{ij}=0$. Then \defiref{g} shows that $g_{ij}$ depends on $y$ alone, if and only if $L$ depends on $y$ alone. But $\partial_i L$ and $\partial_k g_{ij}$ do not constitute tensor fields and hence this motivates the following definition
\begin{defi}
A Finsler space $\mathbf F^n= (\mathbf M^n, L(x,y))$ is called a locally Minkowski space, if $M^n$ is covered by coordinate neighborhoods in each of which $L$ does not depend on the coordinates $(x^{i})$. Such coordinates $(x^{i})$ are called adapted.
\end{defi}
\subsection{Link with symplectic manifolds}
\label{u}
We refer in this part to \citep{shen1, shen2}. We first introduce the following
\begin{defi}
A Minkowski norm on a vector space $\mathbf V$ is function $L:\mathbf V\to \R_{+}^\times$ such that
\begin{enumerate}
\item $L$ is smooth on $\mathbf V\setminus\{0\}$,\\
\item$L$ is homogeneous of degree $1$,\\
\item for all $y\in\mathbf V^\times$, the symmetric bilinear form
\begin{equation}
\begin{split}
g_y: &\mathbf{V}\times \mathbf{V}\to \R\\
&(u,v)\mapsto\dfrac{1}{2}\dfrac{\partial^2 L^2(y+su+tv)}{\partial s\partial t}|_{t=s=0}=\dfrac{1}{2}\dfrac{\partial^2 L^2}{\partial y^{i}\partial y^j}(y)
\end{split}
\end{equation}
is positive definite.
\end{enumerate}
\end{defi}
Here a Finsler manifold $\mathbf M$ is a manifold with a Minkowski norm at the tangent space of each of its points. This Minkowski norm is required to be smooth on $T\mathbf M\setminus\{0\}$. We recall
\begin{defi}[Legendre transformation]
The Legendre transformation $l:\mathbf V\to \mathbf V^\star$ is defined as $l(y)=g_y(y,.)$ for $y\in \mathbf V\setminus\{0\}$, and $l(0)=0$.
\end{defi}
\begin{defi}[Global Legendre transformation]
Let $T\mathbf M$ be the tangent bundle and $\pi$ the canonical projection on $\mathbf M$. The Legendre transformation is defined as
\begin{equation}
\begin{split}
\mathcal{L}:T\mathbf{ M}&\to T^*\mathbf {M} \\
&\xi\mapsto l_{\pi(\xi)}(\xi).
\end{split}
\end{equation}
It is a bijection and a diffeomorphism $T\mathbf M\setminus\{0\}\to T^\star\mathbf M\setminus\{0\}$.
\end{defi}
From \chapref{chap3} it follows that one has a canonical symplectic structure on $T^\star\mathbf M\setminus\{0\}$. No such canonical symplectic structure is known for the tangent bundle however. But if $\mathbf M $ is a Finsler manifold, then $T\mathbf M\setminus\{0\}$ has canonical symplectic structure induced by the so-called Hilbert $1$-form on $T\mathbf M\setminus\{0\}$.
\begin{defi}[Hilbert $1$-form]
Let $L$ be a Finsler norm on $\mathbf M$. Then the Hilbert $1$-form $\eta\in \Omega^1(T\mathbf M\setminus\{0\})$ is defined as
$$\eta|_y=-g_{ij}(y)y^{i}dx^j|_y,y\in T\mathbf M\setminus\{0\}.$$
\end{defi} 
The next proposition shows that $\eta$ is globally defined.
\begin{prop}
If $\mathcal{L}:T\mathbf M\to T^*\mathbf M$ is the Legendre transformation induced by a Finsler norm, then
$$\mathcal{L}^\star\alpha=-\eta,$$
$d\eta$ is a symplectic form for $T\mathbf M\setminus\{0\}$, and $\mathcal{L}:T\mathbf M\setminus\{0\}\to T^\star\mathbf M\setminus\{0\}$ is a symplectic mapping.
\end{prop}
\begin{proof}
The first claim follows directly from the definition by expanding the left hand side. Since $d\alpha$ is non-degenerate, it follows that $d\eta$ is non-degenerate.
\end{proof}
We recall that the geodesic spray is given by
\begin{equation}\displaystyle \mathbb{G}|_y=y^{i}\dfrac{\partial}{\partial x^{i}}|_y-2G^{i}(y)\dfrac{\partial}{\partial y^{i}}|_y, y\in T\mathbf M\setminus\{0\}.\end{equation}
\begin{prop}
In a Finsler space $\mathbf F= (\mathbf M, L)$, the Hilbert $1$-form $\eta$ and the geodesic spray $\mathbb{G}$ satisfy 
$$d\eta(\mathbb{G},.)=d(\dfrac{1}{2}L^2),$$
so $X_{\frac{1}{2}L^2}=\mathbb{G}$. What is more $X_L= \mathbb{G}/L$. 
\end{prop}
\begin{proof}
Using equation \eqref{y} and \theoref{z} we have
\begin{equation}
\begin{split}
d\eta(\mathbb G, .)&=\left(\left(\dfrac{\partial g_{ij}}{\partial x^s}-\dfrac{\partial g_{is}}{\partial x^j}\right)y^{i}y^j+2g_{is}G^{i}\right)dx^s+\dfrac{1}{2}\dfrac{\partial L^2}{\partial y^{i}}dy^{i}\\
&=\dfrac{1}{2}\dfrac{\partial L^2}{\partial x^{i}}dx^{i}+\dfrac{1}{2}\dfrac{\partial L^2}{\partial y^{i}}dy^{i}.
\end{split}
\end{equation} 
The second assertion is immediate.
\end{proof}
\subsection{Conformal change of a Finsler metric}
The absolute length $|v|$ of a tangent vector $v$ at a point $x$ of a Finsler space $\mathbf F^n=(\mathbf M, L)$ is defined as the value $L(x,v)$ provided that $L(x,y)>0$ for any $y$. On the other hand, if $g_{ij}(x,y)\xi^{i}\xi^j$ is positive-definite, then we get the length $|v|_y$ of $v$ relative to $y$ which is equal to $\displaystyle \sqrt{g_{ij}(x,y)v^{i}v^j}$.

As to the angle $\theta$ between two vectors $u$ and $v$ at $x$ it is defined by,

$$\displaystyle\cos(\theta)=\dfrac{\displaystyle g_{ij}(x,y)u^{i}v^j}{ \sqrt{g_{ij}(x,y)v^{i}v^j} \sqrt{g_{ij}(x,y)u^{i}u^j}}.$$
Consequently, this notion of angle is not symmetric. Thus, the above $\theta$ may be called the angle of $v$ with respect to $u$. If we take $y$ as $u$, then we get the angle of $v$ with respect to $y$ as 
$$\displaystyle\cos(\theta)=\dfrac{\displaystyle y_iv^{i}}{L(x,y) \sqrt{g_{ij}(x,y)v^{i}v^j}},$$
where $y_i=g_i(x,y)y^j$.
\begin{defi}
Let $\mathbf F^n=(\mathbf M, L)$ and $^\star\mathbf F^n=(\mathbf M, ^\star L)$ be two Finsler spaces on the same underlying manifold $\mathbf{M}$. If the angles of any tangent vector with respect to any $y$ in $\mathbf F^n$ and $^\star\mathbf F^n$, are equal to each other then $^\star \mathbf F$ is called conformal to $\mathbf F^n$ and the change $L\to ^\star L$ of the metrics is called a conformal change of metrics.
\end{defi}
Hence $^\star \mathbf F $ is conformal to $\mathbf F^n$ if and only if 
\begin{equation}
\label{30002}
^\star L^{2} \,^\star g_{pq}(x,y)v^pv^q(y_rv^r)(y_sv^s)-L^2g_{pq}(x,y)v^pv^q(^\star y_rv^r)(^\star y_sv^s)=0
\end{equation}
for any $x$, $y$ and $v$.
\begin{theo}[\citep{matsu}, p. 784]
The change $L\to ^\star L$ of metrics is conformal, if and only if we have a function $c(x)$ of position alone, satisfying $^\star L= e^{c(x)}L.$
\end{theo}
\section{Example of Finsler spaces: Finsler spaces with m-th Root metric}
We consider an $n$-dimensional Finsler manifold $\mathbf{F}^n= (\mathbf{M},L)$ with such a metric $L(x,y)$ that
\begin{equation}
\label{3033}
L^m=a_{hi\ldots k}(x)y^{h}y^{i}\ldots y^k,
\end{equation}
where $a_{hi\ldots k}(x)$ are components of an $m$-th covariant symmetric tensor, i.e., $L^m$ is a symmetric homogeneous polynomial in $y^{i}$ of the $m$-th degree.
\begin{defi}
The Finsler metric $L$ given by \eqref{3033} is called an $m$-th root metric.
\end{defi}
In the case $m=2$ the metric $L$ is Riemannian, and in the cases $m=3$ and $4$ these metrics are called cubic and quartic, respectively.
\begin{ex}
\label{f}
We have two well-known cubic metrics:
\begin{itemize}
\item $n=2$, \,\,\,$(L_1)^3=(y^1)^3+(y^2)^3$.\\
\item $n=3$, \,\,\,$(L_1)^3=(y^1)^3+(y^2)^3+(y^3)^3-3y^1y^2y^3$.
\end{itemize}
\end{ex}
Now we are concerned with $F^n=(M,L(x,y))$ with $m$-th root metric $L$, given by \eqref{3033}. Differentiating \eqref{3033} by $y^{i}$ and then by $y^j$, we obtain
$$mL^{m-1}l_i=ma_{ij\ldots k}y^j\ldots y^k,$$
$$m(m-1)L^{m-2}l_il_j+mL^{m-2}h_{ij}=m(m-1)a_{ijk\ldots h}(x)y^k\ldots y^h$$
with $l_i$ given by equation \eqref{3006} and $h_{ij}$ given by equation \eqref{3008}. Consequently, we get

\begin{equation}\label{3034} l_i=a_i,\quad h_{ij}=(m-1)(a_{ij}-a_ia_j),\end{equation}
where we put 
\begin{equation}\label{3035}a_i=\displaystyle \dfrac{a_{ij\ldots k}y^j\ldots y^k}{L^{m-1}},\quad a_{ij}=\dfrac{a_{ijk\ldots h}y^k\ldots y^h}{L^{m-2}}.\end{equation}
From \eqref{3034} it follows that the fundamental tensor $g_{ij}=h_{ij}+l_il_j$ is given by 
\begin{equation}
g_{ij}=(m-1)a_{ij}-(m-2)a_ia_j.
\end{equation}
\begin{defi}
A domain $D$ of a Finsler space with $m$-th root metric is called regular, if the matrix $(a_{ij})$ is non-singular at every point of $D$.
\end{defi}
\begin{ex}
We deal with the cubic metrics $L_1$ and $L_2$ of Example \ref{f}.
For $L_1$ we have
$$a_{11}=\dfrac{a_{11j}y^j}{L_1}=\dfrac{y^1}{L_1},\quad a_{22}=\dfrac{y^2}{L_1},\quad a_{12}=\dfrac{a_{12j}y^j}{L_1}=0,$$
which imply $\det(a_{ij})=\dfrac{y^1y^2}{L_1^2}$. Hence $L_1$ is regular, if $y^1y^2­0$.
For $L_2$ we have
$$a_{11}=\dfrac{y^1}{L_2},\quad a_{22}=\dfrac{y^2}{L_2},\quad a_{12}=\dfrac{-y^3}{L_2},$$
which imply $$-\det(a_{ij})=\dfrac{L^3_2}{4}=\dfrac{1}{8}(y^1+y^2+y^3)\left((y^2-y^3)^2+(y^3-y^1)^2+(y^1-y^2)^2\right)$$
Hence $L_2$ is regular in the domain where $y^1+y^2+y^3\not=0$. We shall restrict our consideration of Finsler spaces with $m$-th root metric to a regular domain in the above sense.
\end{ex}
Let $G^{i}(x,y)$ the fundamental function given in \eqref{30003}. We set $\displaystyle G^{i}_{jk}=\dot{\partial}_j\dot{\partial}_kG^i$. We recall that by very definition $G^{i}$ is a homogeneous function of degree $2$, therefore it follows from Euler-relation that
$$G_{jk}^{i}(x,y)y^jy^k=2 G^{i}.$$
\begin{defi}
A Finsler space $\mathbf F=(\mathbf M, L)$ is called a Berwald space, if $G_{jk}^{i}$ are functions of position alone, meaning that they do not depend on $y$. 
\end{defi}
The purpose of the present paragraph is to consider three-dimensional Finsler spaces with cubic metric which are Berwald spaces and locally Minkowski spaces. All cubic metrics of dimension three are divided into the following six classes according to the normal forms: in the notation $(x^{i})=(x,y,z)$ and $(y^{i})=(p,q,r)$, we have \citep[p. 862]{matsu}
\begin{equation}
\label{3036}
\begin{split}
 L^3&=c_1p^3+c_2q^3+c_3r^3+6bpqr, \quad c_1c_2c_3b\not=­0,\\
L^3&=c_1p^3+c_2q^3+c_3r^3, \quad c_1c_2c_3­\not=0,\\
 L^3&=c_1p^3+c_2q^3+6bpqr, \quad c_1c_2b\not=­0,\\
 L^3&=c_1p^3+6bpqr, \quad c_1b­\not=0,\\
 L^3&=6bpqr, \quad b­\not=0,\\
 L^3&=3apr^2+bq^3, \quad ab\not=­0.
\end{split}
\end{equation}
\begin{theo}[\citep{matsu}, p. 886]
The three-dimensional Finsler spaces with cubic metric of the normal form given in equation \eqref{3036} are Berwald spaces, if and only if   
\begin{enumerate}
\item $c_1=c_1(x)$, $c_2=c_2(y)$, $c_3=c_3(z)$, $b^3=kc_1c_2c_3$, with $k\not=­0$ a constant. The space is locally Minkowski.\\
\item $c_1=e^df_1(x)$, $c_2=e^df_2(y)$, $c_3=e^df_3(z)$, $d=d(x,y,z)$, $8b^3=-e^{3d}f_1f_2f_3$. The space is conformal to a Minkowski space with the metric\\
\begin{equation}
\label{3037}
L^3=p^3+q^3+r^3-3pqr.
\end{equation}
\item $c_1=c_1(x)$, $c_2=c_2(y)$, $c_3=c_3(z)$. The space is locally Minkowski. \\
\item $c_1=c_1(x)$, $c_2=c_2(y)$, $b^3=c_1c_2w(z)$. The space is locally Minkowski.\\
\item $c_1=c_1(x)$, $b^3=c_1g(y,z)$.\\
\item The space is conformal to a Minkowski space with metric $(pqr)^{1/3}$.\\
\item $a=a(x,z)$, $b=b(y)$.
\end{enumerate}
\end{theo}
\begin{cor}[\citep{matsu}, p. 886]
We have three-dimensional Berwald spaces with cubic metric of two kinds which are conformal to one of the Minkowski metrics
$$L_1=(p^3+q^3+r^3-3pqr)^{1/3},\quad L_2=(pqr)^{1/3}.$$
The conformal factors are arbitrary functions of $x,y,z$.
\end{cor}
\begin{ex}
Another example of Finsler metric $(\mathbf{M},L)$ with $m-th$ root metric is given by the following $4$-th root metric
\begin{equation*}
\begin{split}
L^4(x,y)&=(y_1+y_2+y_3+y_4)\cdot(y_1+y_2-y_3-y_4)\cdot(y_1-y_2+y_3-y_4)\cdot\\
&(y_1-y_2-y_3+y_4),
\end{split}
\end{equation*}
where $x=(x_1,x_2,x_3,x_4)$ and $y=(y_1,y_2,y_3,y_4)$. Indeed one easily verifies that it is invariant under all transpositions of the symmetric group $S_4$. 
\end{ex}
\section{Connections on Principal Bundles and Finslerian connections}
\begin{par}
According to the "standard model" of elementary particle physics, every fundamental force (boson) is associated with a kind of curvature. But the curvatures involved are not only the curvatures of space-time, but curvatures associated with the notion of a connection on a principal bundle. We assume that we are dealing with a right principal $\mathbf{G}$-bundle. We remind that it is a quadruple $(\mathbf{P},\mathbf{M},\mathbf{G},\pi)$. $\mathbf{P}$ is the total space, $\mathbf{M}$ the base manifold, $\mathbf{G}$ is a Lie group acting on the right on $\mathbf{P}$ and $\pi$ the projection is a submersion. We will denote it (the principal bundle) $\pi:\mathbf{P}\to \mathbf{M}$ or $\mathbf{G}\to \mathbf{P}\to \mathbf{M}$. See below.
\end{par}
\begin{defi}
\label{prbun}
Let $\mathbf P$ and $\mathbf M$ be two differentiable manifolds and $\pi:\mathbf P\to \mathbf M$ a surjective submersion. Let $\mathbf G$ be a Lie group with a right action on $\mathbf P$. Then $(\mathbf P,\mathbf M,\pi,\mathbf G)$ is a principal fiber bundle over $\mathbf M$ with structural group $\mathbf{G}$ if
\begin{itemize}
  \item $\mathbf G$ acts freely on $\mathbf P$: $\exists \quad a: p.a=p\iff a=e$.
  \item $\pi(p_1)=\pi(p_2)\iff$ $p_1=p_2.a$ ($p_1$ et $p_2$ are in the same orbit if and only if they have the same image under $\pi$ on $\mathbf M$); one identifies therefore $\mathbf M$ with the quotient under that equivalence relation. 
  \item $\mathbf P$ is locally trivial i.e. every $x\in \mathbf M$ has a neighborhood $U$, such that there exists a diffeomorphism $\psi:\pi^{-1}(U)\to U\times \mathbf G$: $\psi(p)=(\pi(p),\eta(p))$ with $\eta(p)\in\mathbf G$ and $\psi(p.a)=(\pi\left(p\right),\eta\left(p\right)).a$.  
\end{itemize}
\end{defi}
\begin{par}
Let $q\in \mathbf P$, $x=\pi(q)\in \mathbf M$, $p\in\pi^{-1}(x)$. Let the group right action be denoted by $\Phi$:
\begin{equation}
\begin{split}
\Phi&:\mathbf G\times \mathbf P\to \mathbf P\\
\Phi(g,p)&=\Phi_g(p):=p.g
\end{split}
\end{equation}
\end{par} 
\begin{par}
We can think of the group as an action which pushes points in the bundle along the fibers. 
\end{par}
\subsection{Horizontal spaces, vertical spaces, and connections}
Taking $p$ as previously. We define the vertical subspace $ V_p\subset T_p\mathbf P$ as
$$ V_p=\ker (\pi_{\star})_p, \text{with $\pi_{\star}$ the differential of $\pi$}.$$
A vector field $X$ on $\mathbf P$ is vertical if $X_p\in  V_p$ for any $p$. The Lie bracket of two vertical vector fields is again vertical. Indeed consider two vector fields $X$ and $Y$ belonging to $ V$ and $f\in \mathscr{C}^\infty(\mathbf M)$. We have
\begin{equation}
\begin{split}
\pi_{\star}\left[X,Y\right]_p(f)&:=\left[X,Y\right]_p(f\circ \pi)\\
&=X_p(Y(f\circ \pi))-Y_p(X(f\circ \pi))\\
&=X_p((\pi_{\star}\circ Y)(f))-Y_p((\pi_{\star}\circ X)(f))\\
&=0;
\end{split}
\end{equation}
where we have used the fact that
$$(\pi_{\star}\circ X)(f)=X(f\circ\pi)\quad \text{and}\quad (\pi_{\star}\circ Y)(f)=Y(f\circ\pi).$$
Therefore the vertical subspaces define a $\mathbf G$-invariant distribution $ V\subset T\mathbf P$; because $\pi\circ\Phi_g(p)=\pi$ implies that we have
$$(\Phi_g)_{\star}V_p= V_{p.g}.$$
In the absence of any extra structure, there is no natural way to choose a complement to $ V_p$ in $T_p\mathbf P$. This is in a sense what a connection provides.
\begin{defi}
A connection on $\mathbf P$ is a smooth choice of horizontal subspaces $ H_p\subset T_p\mathbf P$ complementary to $ V_p$:
$$T_p\mathbf P=V_p\oplus H_p$$
and such that $(\Phi_g)_{\star}H_p=H_{p.g}$. This also means that a connection is a $G$-invariant distribution $H\subset T\mathbf P$ complementary to $V$.
\end{defi} 
We visualize below the local representation of a connection and horizontal subspaces
\begin{figure}
\begin{center}
\tikzstyle{grisEncadre}=[dashed]
\tikzstyle{nfEncadre}=[thick]
\tikzstyle{nffEncadre}=[very thick]
\tikzstyle{ddEncadre}=[densely dotted]
\begin{tikzpicture}
\draw [grisEncadre] (0,0,0) -- (2.1,0,2.2) -- (2,2.2,2) -- (0,2.2,0) -- cycle;
\draw [grisEncadre] (-1.2,0,2.2) -- (0,-0.8,2.2) -- (0,1.6,2.2) -- (-1.2,2.4,2.2) -- cycle;
\draw [grisEncadre] (0,0,0) -- (-1.2,0,2.2);
\draw [grisEncadre] (2.1,0,2.2) -- (0,-0.8,2.2);
\draw [grisEncadre] (2,2.2,2) -- (0,1.6,2.2);
\draw [grisEncadre] (0,2.2,0) -- (-1.2,2.4,2.2);
\draw (-5,-4.3) ..controls +(3,1) and +(-4,0).. (5,-5);
\draw (-3,-6) ..controls +(3,1) and +(-4,0).. (7,-7);
\draw (-5,-4.3) -- (-3,-6);
\draw (5,-5) -- (7,-7);
\draw [grisEncadre] (0.3,-4.7) -- (-0.5,-4.8) -- (0.06,-5.3) -- (1,-5.3) -- cycle;
\node at (-0.7,0.27) {\tiny{$\bullet$}};
\node at (-1,0.16) {$q$};
\node at (0.2,-5) {\tiny{$\bullet$}};
\node at (2,1.97) {\tiny{$\bullet$}};
\node at (2,1.6) {$p$};
\node at (3.2,2.25) {\tiny{$\bullet$}};
\node at (4,1.8) {$\Phi_{g}(p) = p \cdot g$};
\draw (0.2,-5) ..controls +(-3,1.3) and +(0.7,-0.5).. (0.75,-0.3);
\draw (0.75,-0.3) ..controls +(0,0.1) and +(0,0).. (-0.3,0.1);
\draw (-0.3,0.1) ..controls +(-3,1) and +(0,0).. (4.4,2.5);
\draw [nfEncadre] (-2.5,1.3) -- (-0.7,0.25) -- (2,-1.22);
\node at (2.5,-1.4) {$T_{q}\mathbf G$};
\draw [->,>=latex] (2.1,2.2) to[bend left = 35] (3.1,2.4);
\node at (2.6,2.8) {$\Phi_{g}$};
\draw [ddEncadre] [->,>=latex] (-0.65,0.17) to[bend left = 20] (0.3,-4.6);
\draw [->,>=latex] (2.6,0.7) to[bend right = 25] (1.5,0.65);
\node at (3.1,0.6) {$T_{q}\mathbf P$};
\node at (-1.5,-3) {$\textup{fiber} \hspace{1mm} \mathbf G$};
\node at (0.6,-2.5) {$\pi$};
\node at (1.57,-4.89) {$x = \pi (p) = \pi (q)$};
\draw [->,>=latex] (1.5,-7.1) to[bend right = 20] (0.8,-5.4);
\node at (1.5,-7.5) {$T_{\pi (q)}\mathbf M$};
\node at (5.5,-6.5) {$\textup{base} \hspace{1mm} \mathbf M$};
\draw [nffEncadre] (-6,-1.5) ..controls +(0.3,0.3) and +(-0.2,0).. (-5.8,3.5);
\draw [nffEncadre] (-6,-1.5) ..controls +(0.3,0.3) and +(-0.2,0).. (-5.8,-6.4);
\node at (-6.35,-1.5) {$\mathbf P$};
\end{tikzpicture}
\end{center}
\caption{Visualizing a local representation of a principal bundle.}
\end{figure}
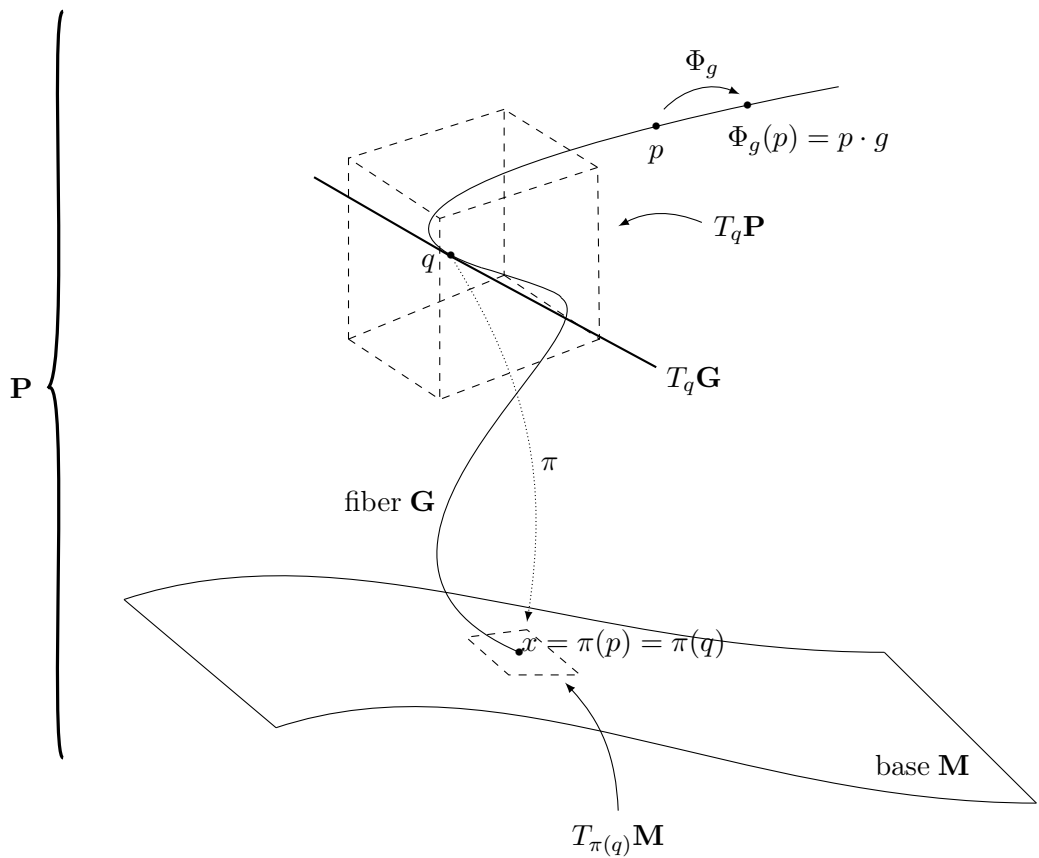
\clearpage
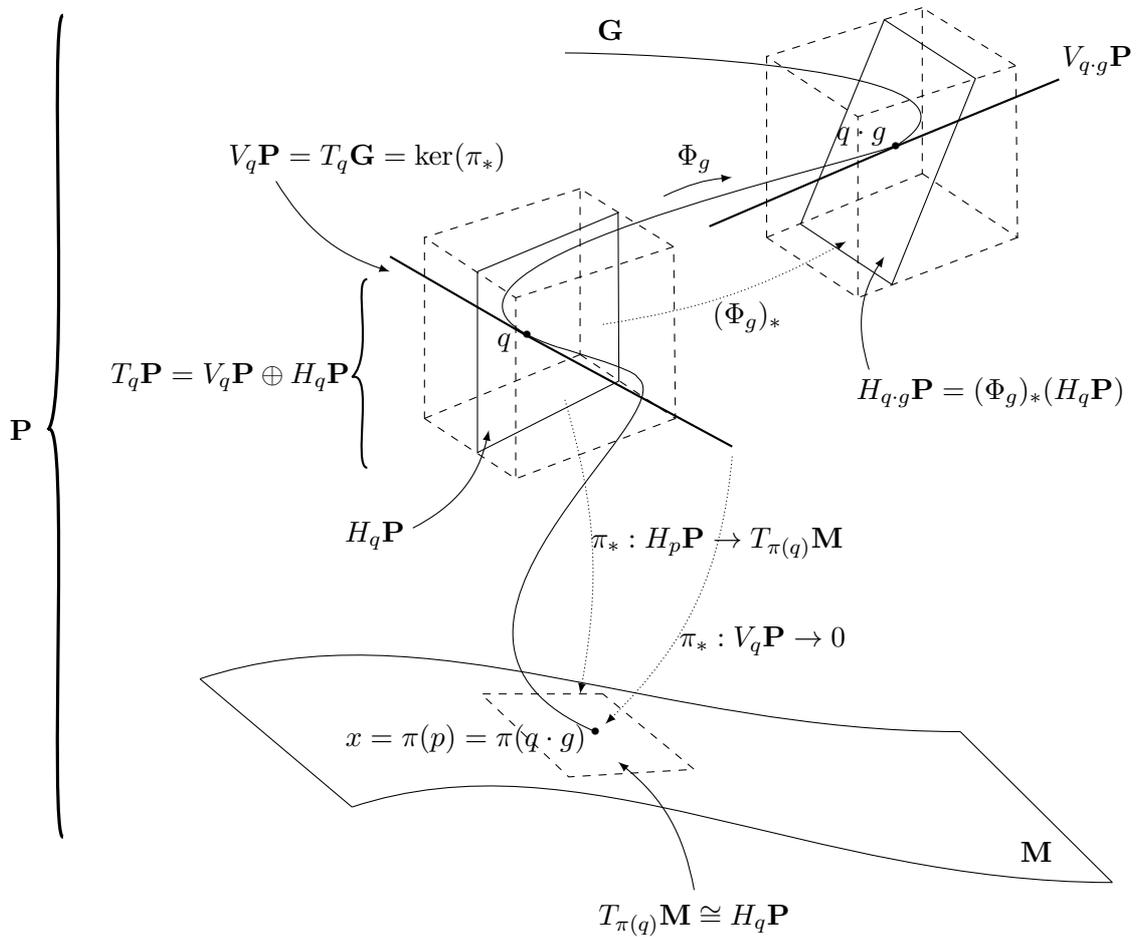
\begin{figure}
\begin{center}
\tikzstyle{grisEncadre}=[dashed]
\tikzstyle{nfEncadre}=[thick]
\tikzstyle{nffEncadre}=[very thick]
\tikzstyle{ddEncadre}=[densely dotted]
\begin{tikzpicture}
\draw [grisEncadre] (4.5,2.4,0) -- (6.6,2.4,2.2) -- (6.5,4.6,2) -- (4.5,4.6,0) -- cycle;
\draw [grisEncadre] (3.3,2.4,2.2) -- (4.5,1.6,2.2) -- (4.5,4,2.2) -- (3.3,4.8,2.2) -- cycle;
\draw [grisEncadre] (4.5,2.4,0) -- (3.3,2.4,2.2);
\draw [grisEncadre] (6.6,2.4,2.2) -- (4.5,1.6,2.2);
\draw [grisEncadre] (6.5,4.6,2) -- (4.5,4,2.2);
\draw [grisEncadre] (4.5,4.6,0) -- (3.3,4.8,2.2);
\draw [grisEncadre] (0,0,0) -- (2.1,0,2.2) -- (2,2.2,2) -- (0,2.2,0) -- cycle;
\draw [grisEncadre] (-1.2,0,2.2) -- (0,-0.8,2.2) -- (0,1.6,2.2) -- (-1.2,2.4,2.2) -- cycle;
\draw [grisEncadre] (0,0,0) -- (-1.2,0,2.2);
\draw [grisEncadre] (2.1,0,2.2) -- (0,-0.8,2.2);
\draw [grisEncadre] (2,2.2,2) -- (0,1.6,2.2);
\draw [grisEncadre] (0,2.2,0) -- (-1.2,2.4,2.2);
\draw (-5,-4.3) ..controls +(3,1) and +(-4,0).. (5,-5);
\draw (-3,-6) ..controls +(3,1) and +(-4,0).. (7,-7);
\draw (-5,-4.3) -- (-3,-6);
\draw (5,-5) -- (7,-7);
\draw [grisEncadre]  (0.3,-4.5) -- (-1.3,-4.5) -- (-0.15,-5.6) -- (1.5,-5.5) -- cycle;
\node at (-0.7,0.27) {\tiny{$\bullet$}};
\node at (0.2,-5) {\tiny{$\bullet$}};
\node at (4.15,2.77) {\tiny{$\bullet$}};
\draw (0.2,-5) ..controls +(-3,1.3) and +(0.7,-0.5).. (0.75,-0.3);
\draw (0.75,-0.3) ..controls +(0,0.1) and +(0,0).. (-0.3,0.1);
\draw (-0.3,0.1) ..controls +(-3,1) and +(-0.2,-0.2).. (4.2,2.8);
\draw (4.2,2.8) ..controls +(1.5,1) and +(1,0).. (-0.2,4);
\draw [nfEncadre] (-2.5,1.3) -- (-0.7,0.25) -- (2,-1.22);
\draw [nfEncadre] (1.7,1.7) -- (4.15,2.76) -- (6.3,3.65);
\draw (0.5,1.88) -- (0.51,-0.35) -- (-1.35,-1.3) -- (-1.35,1.1) -- cycle;
\draw (5.2,3.66) -- (4.1,0.93) -- (2.9,1.73) -- (4,4.44) -- cycle;
\draw [->,>=latex] (1.5,-7.1) to[bend right = 20] (0.5,-5.4);
\node at (1.5,-7.5) {$T_{\pi (q)}\mathbf M \cong  H_{q}\mathbf P$};
\node at (6,-6.6) {$\mathbf M$};
\node at (-1.5,-5.1) {$x = \pi (p) = \pi (q \cdot g)$};
\draw [ddEncadre] [->,>=latex] (2,-1.35) to[bend left = 20] (0.33,-4.9);
\draw [ddEncadre] [->,>=latex] (-0.2,-0.6) to[bend left = 13] (0,-4.5);
\node at (2.4,-3.8) {$\pi_{\ast} : V_{q}\mathbf P \rightarrow 0$};
\node at (1.8,-2.5) {$\pi_{\ast} : H_{p}\mathbf P \rightarrow T_{\pi (q)}\mathbf M$};
\draw [->,>=latex] (-2.2,-2.3) to[bend right = 25] (-1.2,-1);
\node at (-2.7,-2.4) {$H_{q}\mathbf P$};
\node at (-1,0.16) {$q$};
\node at (3.7,2.9) {$q \cdot g$};
\draw [->,>=latex] (3.8,-0.2) to[bend left = 25] (4,1.3);
\node at (5.4,-0.5) {$H_{q \cdot g}\mathbf P = (\Phi_{g})_{\ast}(H_{q}\mathbf P)$};
\draw [ddEncadre] [->,>=latex] (0.3,0.4) to[bend right = 10] (3.5,1.5);
\node at (2.2,0.5) {$(\Phi_{g})_{\ast}$};
\draw [->,>=latex] (1.1,2.1) to[bend left = 10] (2,2.35);
\node at (1.5,2.6) {$\Phi_{g}$};
\node at (6.8,3.9) {$V_{q \cdot g}\mathbf P$};
\node at (0.4,4.3) {$\mathbf G$};
\draw [->,>=latex] (-4,2.3) to[bend right = 20] (-2.5,1.1);
\node at (-2.8,2.6) {$V_{q}\mathbf P = T_{q}\mathbf G = \ker (\pi_{\ast})$};
\draw [nfEncadre] (-3,-0.3) ..controls +(0.3,0.3) and +(-0.2,0).. (-2.8,1);
\draw [nfEncadre] (-3,-0.3) ..controls +(0.3,0.1) and +(-0.2,0).. (-2.8,-1.5);
\draw [nffEncadre] (-7,-1) ..controls +(0.3,0.3) and +(-0.2,0).. (-6.8,4.5);
\draw [nffEncadre] (-7,-1) ..controls +(0.3,0.3) and +(-0.2,0).. (-6.8,-6.4);
\node at (-7.35,-1) {$\mathbf P$};
\node at (-4.6,-0.3) {$T_{q}\mathbf P = V_{q}\mathbf P \oplus H_{q}\mathbf P$};
\end{tikzpicture}
\end{center}
\caption{Horizontal spaces.}
\end{figure}
The action of $\mathbf G$ on $\mathbf P$ defines a map $\varrho:\mathfrak{g}\to \mathcal{H}(\mathbf P)$, with $\mathcal{X}(\mathbf P)$ the set of vector fields on $\mathbf P$. It assigns to every $v\in \mathfrak{g}$, the vector field $\varrho(v)$ whose value at $p$ is given by
$$\varrho_p(v)=\dfrac{d}{dt}(p.e^{tv})\vert_{t=0}.$$
Notice that 
$$\pi_{\star}\varrho_p(v)=\dfrac{d}{dt}\pi(p.e^{tv})\vert_{t=0}=\dfrac{d}{dt}\pi(p)\vert_{t=0}=0,$$
whence $\varrho(v)$ is a vertical vector field. In fact, since $\mathbf G$ acts freely on $\mathbf P$, the map $v\to \varrho_p(v)$ is an isomorphism between $\mathfrak{g}$ and $V_p$ for every $p$.
\begin{lem}
$$(\Phi_g)_{\star}\varrho(v)=\varrho(ad_{g^{-1}}v).$$
\end{lem}
\begin{proof}
By definition, at $p\in \mathbf P$ we have
\begin{equation}
\begin{split}
(\Phi_g)_{\star}\varrho_p(v)&=\dfrac{d}{dt}\Phi_g(p.e^{tv})\vert_{t=0}\\
&=\dfrac{d}{dt}(p.e^{tv}g)\vert_{t=0}\\
&=\dfrac{d}{dt}(p.gg^{-1}e^{tv}g)\vert_{t=0}\\
&=\dfrac{d}{dt}(p.ge^{tag_{g^{-1}}v})\vert_{t=0}\\
&=\varrho_{p.g}(ad_{g^{-1}}v).
\end{split}
\end{equation}
\end{proof}
 The horizontal subspace $H_p\subset T_p\mathbf P$ being a linear subspace, is annihilated by $k=\dim \mathbf G$ linear equations $T_p\mathbf P\to \R$ (its orthogonal). In other words, $H_p$ is the kernel of $k$ one-forms at $p$, the components of a one form $\omega$ at $p$ with values in a $k$-dimensional vector space. There is a natural such vector space, namely the Lie algebra of $\mathbf G$, and since $\omega$ annihilates horizontal vectors (vectors belonging to $H_p$ for any $p$), it is defined by what it does on vertical vectors, and we have a natural map $V_p\to \mathfrak{g}$ given by the inverse of $\varrho_p$. This prompts the following definition
 \begin{defi}
 The connection one-form of a connection $H\subset T\mathbf P$ is the $\mathfrak{g}$-valued one form $\omega$ defined by
  \begin{equation}
  \omega(X)=
\begin{cases}
      v& \text{if $X=\varrho(v)$ }, \\
      0& \text{if $X$ is horizontal};
\end{cases}
\end{equation}
and obeys to the identity
$$(\Phi_{g}){\star}\omega=ad_{g^{-1}}\omega.$$
 \end{defi}
 Now a form on $\mathbf P$ is said to be horizontal if it it annihilates the vertical vectors. The following picture shows how the previous definition works
 \clearpage
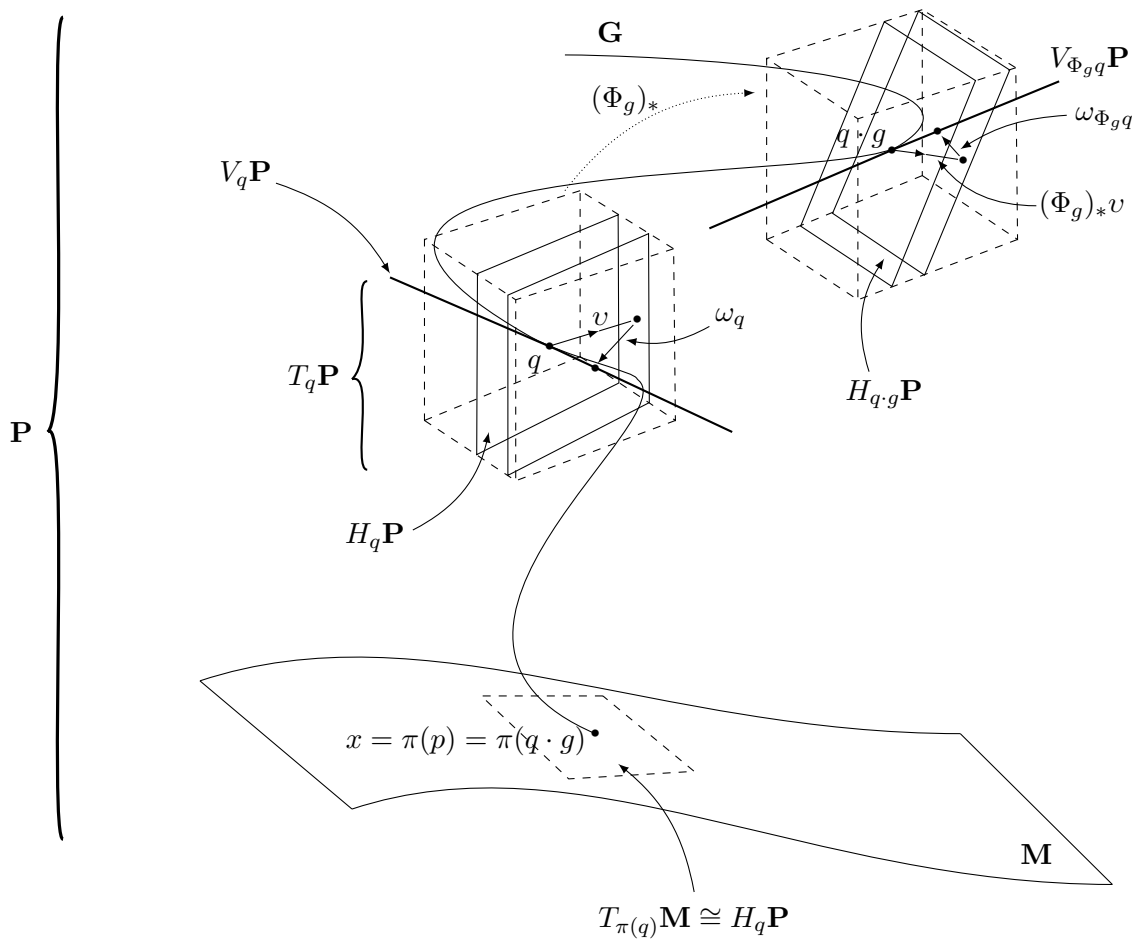
\begin{figure}
\begin{center}
\tikzstyle{grisEncadre}=[dashed]
\tikzstyle{nfEncadre}=[thick]
\tikzstyle{nffEncadre}=[very thick]
\tikzstyle{ddEncadre}=[densely dotted]
\begin{tikzpicture}
\draw [grisEncadre] (4.5,2.4,0) -- (6.6,2.4,2.2) -- (6.5,4.6,2) -- (4.5,4.6,0) -- cycle;
\draw [grisEncadre] (3.3,2.4,2.2) -- (4.5,1.6,2.2) -- (4.5,4,2.2) -- (3.3,4.8,2.2) -- cycle;
\draw [grisEncadre] (4.5,2.4,0) -- (3.3,2.4,2.2);
\draw [grisEncadre] (6.6,2.4,2.2) -- (4.5,1.6,2.2);
\draw [grisEncadre] (6.5,4.6,2) -- (4.5,4,2.2);
\draw [grisEncadre] (4.5,4.6,0) -- (3.3,4.8,2.2);
\draw [grisEncadre] (0,0,0) -- (2.1,0,2.2) -- (2,2.2,2) -- (0,2.2,0) -- cycle;
\draw [grisEncadre] (-1.2,0,2.2) -- (0,-0.8,2.2) -- (0,1.6,2.2) -- (-1.2,2.4,2.2) -- cycle;
\draw [grisEncadre] (0,0,0) -- (-1.2,0,2.2);
\draw [grisEncadre] (2.1,0,2.2) -- (0,-0.8,2.2);
\draw [grisEncadre] (2,2.2,2) -- (0,1.6,2.2);
\draw [grisEncadre] (0,2.2,0) -- (-1.2,2.4,2.2);
\draw (-5,-4.3) ..controls +(3,1) and +(-4,0).. (5,-5);
\draw (-3,-6) ..controls +(3,1) and +(-4,0).. (7,-7);
\draw (-5,-4.3) -- (-3,-6);
\draw (5,-5) -- (7,-7);
\draw [grisEncadre]  (0.3,-4.5) -- (-1.3,-4.5) -- (-0.15,-5.6) -- (1.5,-5.5) -- cycle;
\node at (-0.4,0.14) {\tiny{$\bullet$}};
\node at (0.2,-5) {\tiny{$\bullet$}};
\node at (4.10,2.74) {\tiny{$\bullet$}};
\draw (0.2,-5) ..controls +(-3,1.3) and +(0.7,-0.5).. (0.75,-0.3);
\draw (0.75,-0.3) ..controls +(0,0.1) and +(0.5,-0.3).. (-0.7,0.29);
\draw (-0.7,0.29) ..controls +(-4.3,2.5) and +(-0.8,-0.55).. (4.2,2.8);
\draw (4.2,2.8) ..controls +(1.6,1) and +(1,0).. (-0.2,4);
\draw [nfEncadre] (-2.5,1.05) -- (-0.3,0.09) -- (2,-1);
\draw [nfEncadre] (1.7,1.7) -- (4.15,2.76) -- (6.3,3.65);
\draw (0.5,1.88) -- (0.51,-0.35) -- (-1.36,-1.3) -- (-1.35,1.1) -- cycle;
\draw (0.9,1.63) -- (0.91,-0.61) -- (-0.95,-1.575) -- (-0.95,0.815) -- cycle;
\draw (5.2,3.66) -- (4.1,0.93) -- (2.9,1.73) -- (4,4.44) -- cycle;
\draw (5.65,3.8) -- (4.53,1.08) -- (3.32,1.9) -- (4.45,4.58) -- cycle;
\draw [->,>=latex] (1.5,-7.1) to[bend right = 20] (0.5,-5.4);
\node at (1.5,-7.5) {$T_{\pi (q)}\mathbf M \cong H_{q}\mathbf P$};
\node at (6,-6.6) {$\mathbf M$};
\node at (-1.5,-5.1) {$x = \pi (p) = \pi (q \cdot g)$};
\draw [->,>=latex] (-2.2,-2.3) to[bend right = 25] (-1.2,-1);
\node at (-2.7,-2.4) {$H_{q}\mathbf P$};
\node at (-0.6,-0.1) {$q$};
\node at (3.7,2.9) {$q \cdot g$};
\draw [->,>=latex] (3.8,-0.2) to[bend left = 25] (4,1.3);
\node at (4,-0.5) {$H_{q \cdot g}\mathbf P$};
\draw [ddEncadre] [->,>=latex] (-0.2,2.2) to[bend left = 25] (2.3,3.5);
\node at (0.6,3.4) {$(\Phi_{g})_{\ast}$};
\node at (6.7,3.9) {$V_{\Phi_{g}q}\mathbf P$};
\node at (0.4,4.3) {$\mathbf G$};
\draw [->,>=latex] (-4,2.3) to[bend left = 20] (-2.5,1.1);
\node at (-4.4,2.4) {$V_{q}\mathbf P$};
\draw [nfEncadre] (-3,-0.3) ..controls +(0.3,0.3) and +(-0.2,0).. (-2.8,1);
\draw [nfEncadre] (-3,-0.3) ..controls +(0.3,0.1) and +(-0.2,0).. (-2.8,-1.5);
\draw [nffEncadre] (-7,-1) ..controls +(0.3,0.3) and +(-0.2,0).. (-6.8,4.5);
\draw [nffEncadre] (-7,-1) ..controls +(0.3,0.3) and +(-0.2,0).. (-6.8,-6.4);
\node at (-7.35,-1) {$\mathbf P$};
\node at (-3.5,-0.3) {$T_{q}\mathbf P$};
\node at (0.2,-0.149) {\tiny{$\bullet$}};
\node at (0.75,0.5) {\tiny{$\bullet$}};
\draw [->,>=latex] (0.7,0.42) -- (0.23,-0.09);
\draw (0.25,0.34) -- (0.67,0.47);
\draw [->,>=latex] (-0.4,0.14) -- (0.25,0.34);
\node at (4.7,2.985) {\tiny{$\bullet$}};
\node at (5.04,2.6) {\tiny{$\bullet$}};
\draw [->,>=latex] (4.99,2.66) -- (4.75,2.93);
\draw [->,>=latex] (4.10,2.74) -- (4.55,2.68);
\draw (4.55,2.68) -- (4.97,2.615);
\draw [->,>=latex] (6.4,3.2) to[bend right = 14] (5,2.75);
\node at (6.9,3.15) {$\omega_{\Phi_{g}q}$};
\draw [->,>=latex] (6,2) to[bend left = 25] (4.7,2.61);
\node at (6.6,2) {$(\Phi_{g})_{\ast}\upsilon$};
\draw [->,>=latex] (1.7,0.4) to[bend left = 20] (0.6,0.22);
\node at (1.98,0.5) {$\omega_{q}$};
\node at (0.27,0.5) {$\upsilon$};
\end{tikzpicture}
\end{center}
\caption{$\mathbf G$-equivariance of the connection form.}
\end{figure}
 \begin{rem}
 A choice of connection (form) is in fact a choice of projection of $T_p\mathbf P$ on $V_p$ with the complement $H_p$ varying in an equivariant manner \rm{(}$\Phi_g{_{\star}}H_p=H_{p.g}$\rm{)}.  
 \end{rem}
 Let us see how one can project $T_p\mathbf P$ on the vertical space $V_p$. 
 \clearpage
 \begin{figure}
\begin{center}
\tikzstyle{grisEncadre}=[dashed]
\tikzstyle{nfEncadre}=[thick]
\tikzstyle{nffEncadre}=[very thick]
\tikzstyle{ddEncadre}=[densely dotted]
\begin{tikzpicture}[scale=1.5]
\draw [grisEncadre] (0,0) -- (2.7,-0.2) -- (2.7,2.3) -- (0,2.5) -- cycle;
\draw [grisEncadre] (-1.25,-1.25) -- (1.45,-1.45) -- (1.45,1.05) -- (-1.25,1.25) -- cycle;
\draw [grisEncadre] (0,0) -- (-1.25,-1.25);
\draw [grisEncadre] (2.7,-0.2) -- (1.45,-1.45);
\draw [grisEncadre] (2.7,2.3) -- (1.45,1.05);
\draw [grisEncadre] (0,2.5) -- (-1.25,1.25);
\draw (-0.1,-1.33) -- (1.62,-0.12) -- (1.45,2.39) -- (-0.3,1.18) -- cycle;
\draw (1,-1.42) -- (2.62,-0.19) -- (2.5,2.31) -- (0.75,1.11) -- cycle;
\draw [nfEncadre] (-2,0.8) -- (4.5,0.1);
\node at (0.6,0.53) {\tiny{$\bullet$}};
\node at (1.9,0.38) {\tiny{$\bullet$}};
\node at (2.45,1.85) {\tiny{$\bullet$}};
\draw (-2.8,0.3) to[bend left = 16] (3.7,-0.3);
\draw [->,>=latex] (2.43,1.80) -- (1.92,0.42);
\draw (1.5,1.175) -- (2.41,1.82);
\draw [->,>=latex] (0.6,0.53) -- (1.5,1.175);
\draw [->,>=latex] (0.68,0.56) -- (1.85,0.44);
\draw [nfEncadre] (0.7,-2.2) ..controls +(-0.2,0.1) and +(0,-0.1).. (-1.4,-2.1);
\draw [nfEncadre] (0.7,-2.2) ..controls +(-0.2,0.08) and +(0,-0.1).. (2.9,-2.1);
\node at (0.5,-2.6) {$T_{q}\mathbf P = V_{q}\mathbf P \oplus H_{q}\mathbf P$};
\node at (0.55,0.35) {$q$};
\node at (3.8,-0.5) {$\mathbf G$};
\node at (4.6,0.35) {$V_{q}\mathbf P$};
\node at (1.4,1.3) {$\mathcal{V}$};
\draw [->,>=latex] (1.8,3) to[bend left = 15] (1.35,2.15);
\node at (1.85,3.2) {$H_{q}\mathbf P$};
\end{tikzpicture}
\end{center}
\caption{Using $H_{q}\mathbf P$ to project $\mathcal{V} \in T_{q}\mathbf P$ on to $V_{q}\mathbf P$.}
\end{figure}
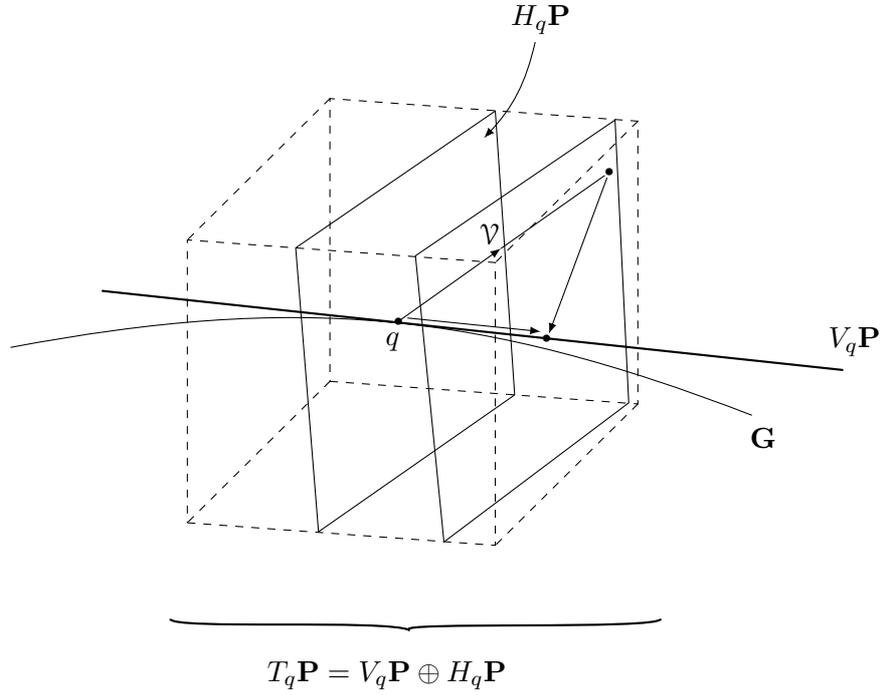
 \begin{par}
 Let $\omega_1$ a $1$-form with values in $\mathfrak{g}$ with and $\omega_2$ another $1$-form with values in $\mathfrak{g}$. We define the $2$-form
 $$\left[\omega_1,\omega_2\right]:(X,Y)\to \left[\omega_1(X),\omega_2(Y)\right]-\left[\omega_1(Y),\omega_2(X)\right]$$
 \end{par}
 \begin{par}
Consider now a connection with a given connection one form $\omega$; its curvature is defined by the formula
 \begin{equation}
\Omega=d\omega+\dfrac{1}{2}\left[\omega,\omega\right].
\end{equation}
\end{par}
We can thus rewrite the previous equation in the form
$$\Omega(X,Y)=d\omega(X,Y)+\left[\omega(X),\omega(Y)\right].$$
When $\omega$ is given in matrix form, i,e, there exists a representation of $\mathfrak{g}$ in some $M_n(\R)$, the matrix $\left[\omega(X),\omega(Y)\right]$ (the commutator of $\omega(X)$ and $\omega(Y)$) is just $\omega(X)\wedge\omega(Y)$:
$$\left[\omega(X),\omega(Y)\right]=\omega(X)\wedge\omega(Y);$$
because if $\omega=\omega_{ij}$ then
$$\displaystyle\omega\wedge\omega(X,Y)=\sum_k\omega_{ik}\wedge\omega_{kj}(X,Y)=\sum_k(\omega_{ik}(X))(\omega_{kj}(Y))-(\omega_{ik}(Y))(\omega_{kj}(X)),$$
with $(\omega_{ij}(X))$ denoting the matrix with components $\omega_{ij}(X)$.
\begin{par}
The curvature hence takes the form
$$\Omega=d\omega+\omega\wedge\omega$$
\end{par}
 Next we define the notions of Cartan geometry and connection
\begin{defi}
A Cartan geometry $\xi=(\mathbf P,\omega,\mathbf G,\mathbf H)$ on a manifold $\mathbf M$ modeled on $(\mathfrak{g},\mathfrak{h})$ with group $\mathbf H$ consists of the following data:
\begin{itemize}
\item a smooth manifold $\mathbf M$
\item a principal right $\mathbf H$ bundle $\mathbf P$ over $\mathbf M$.
\item\textit{ a $\mathfrak{g}$-valued $1$-form $\omega$ on $\mathbf P$ \rm{(}called Cartan connection\rm{)} which satisfies the following conditions:}
\begin{enumerate}
\item \textit{for each point $p\in \mathbf P$, the linear map $\omega_p:T_p\mathbf P\to\mathfrak{g}$ is a linear isomorphism;}
\item \textit{$(\Phi_h)^{\star}\omega=\rm{ad}_{h^{-1}}\omega$ for all $h\in \mathbf H$; }
\item \textit{$\omega(\varrho(v))=v$ for all $v\in\mathfrak{h}$ with $\varrho(v)$ the associated vector field.}
\end{enumerate}
\end{itemize}
\end{defi}
The curvature of a Cartan connection $\omega$ is defined analogously and verifies the relation
$$\Omega=d\omega+\omega\wedge\omega,$$
when $\omega$ is a matrix of one-forms. If it vanishes the model geometry is called flat.
\begin{rem}
A Cartan connection in $\mathbf P$ is not a connection in the usual sense since it is not $\mathfrak{h}$-valued. It can however be considered as a connection in a bundle associated to the bundle $\mathbf P$.
\end{rem}
\subsection{Connections in some Locally Minkowski spaces}
In this part we restrict attention to Finsler spaces as defined in subsection \ref{u}.

\noindent In many ways, Finsler geometry is analogous to Riemannian geometry. However, a typical difference is that Finsler geometry's objects exist on $T\mathbf M$ whereas in Riemann geometry they exist on $\mathbf M$. For example, in Finsler geometry, curvature is a tensor on $T\mathbf M\setminus \{0\}$, wheresas in Riemannian geometry, it is a tensor on $\mathbf M$. For this reason we need to study vectors and co-vectors on $T\mathbf M$, that is, elements in $T(T\mathbf M\setminus \{0\})$ and $T(T^\star\mathbf M\setminus \{0\})$. In this section we define the horizontal-vertical decomposition in $T(T\mathbf M\setminus \{0\})$ and $T(T^\star\mathbf M\setminus \{0\})$. This decomposition will greatly simplify calculations in local coordinates. It will also give a certain structure compatible with the Finsler metric. For example, the tangent vector to a geodesic will be a horizontal vector. Also, the derivative of $F$ will be a vertical co-vector. In order to introduce the horizontal-vertical decomposition, one needs a non-linear connection, that is one needs some structure on $T\mathbf M\setminus \{0\}$.
\begin{defi}[Non-linear connection]
A non-linear connection on a manifold $\mathbf M$ is a collection of locally defined $1$-homogeneous functions $N_j^i$ on $T\mathbf M\setminus \{0\}$ satisfying the transformation rules
$$\displaystyle \dfrac{\partial \tilde{x}^j}{\partial x^i}\tilde{N}^h_j=\dfrac{\partial \tilde{x}^h}{\partial x^i}N^j_i-\dfrac{\partial^2 \tilde{x}^h}{\partial x^i\partial x^j}y^j.$$
\end{defi}
We recall that the transformation rules of the coefficient $G^i$ of the geodesic spray is given by 
\begin{equation}
\label{4000}
\tilde{G}^r=\dfrac{\partial\tilde{x}^r}{\partial x^i}G^i-\dfrac{1}{2}\dfrac{\partial^2\tilde{x}^r}{\partial x^i\partial x^s}y^iy^s.
 \end{equation}
 Therefore if we take the derivative of $G^i$ with respect to $y^i$:
 \begin{equation}\label{4001}
N^i_j=\dfrac{\partial G^i}{\partial y^j}
\end{equation}
 we see that the $N_j^i$ are the coefficients for a non-linear connection.
 
\subsubsection{Decomposition of $T(T\mathbf M\setminus \{0\})$}
We recall that under a change of coordinates $\dfrac{\partial}{\partial x^i}$ and $\dfrac{\partial}{\partial y^i}$ vary in the following way
\begin{equation}
\label{4002}
\displaystyle \dfrac{\partial}{\partial x^i}|_y=\dfrac{\partial\tilde{x}^r}{\partial x^i}\dfrac{\partial}{\partial \tilde{x}^r}|_y+\dfrac{\partial^2\tilde{x}^r}{\partial x^i\partial x^s}y^s\dfrac{\partial}{\partial \tilde{y}^r}|_y
\end{equation} 
and
\begin{equation}
\label{4003}
\displaystyle \dfrac{\partial}{\partial y^i}|_y=\dfrac{\partial\tilde{x}^r}{\partial x^i}\dfrac{\partial}{\partial \tilde{y}^r}|_y.
\end{equation}
Therefore the vector space $\spa\left\{\dfrac{\partial}{\partial x^i}|_y,\, 1,\ldots,n\right\}$ depends on local coordinates. However, if $\mathbf M$ is equipped with a non-linear connection $N_j^i$, and we let
$$\dfrac{\delta}{\delta x^i}|_y=\dfrac{\partial}{\partial x^i}-N_i^k(y)\dfrac{\partial}{\partial y^k}|_y\in T(T\mathbf M\setminus \{0\})$$
then
$$\dfrac{\delta}{\delta x^i}|_y=\dfrac{\partial\tilde{x}^r}{\partial x^i}\dfrac{\delta}{\delta \tilde{x}^r}|_y.$$
Thus the $2n$-dimensional vector space $T_y(T\mathbf M\setminus\{0\})$ has two $n$-dimensional subspaces,
$$\mathcal{V}_yT\mathbf M=\spa\left\{\dfrac{\partial}{\partial y^i}|_y\right\},\quad \mathcal{H}_yT\mathbf M=\spa\left\{\dfrac{\delta}{\delta x^i}|_y\right\},$$
and these are independent of local coordinate coordinates . Let us also define
$$\displaystyle\mathcal{V}T\mathbf M=\bigcup_{y\in T\mathbf M\setminus \{0\}} \mathcal{V}_yT\mathbf M ,\quad \mathcal{H}T\mathbf M=\bigcup_{y\in T\mathbf M\setminus \{0\}} \mathcal{H}_yT\mathbf M .$$
Thence pointwise
$$T(T\mathbf M\setminus \{0\})=\mathcal{V}T\mathbf M\oplus\mathcal{H}T\mathbf M.$$
Vectors in $\mathcal{V}T\mathbf M$ are called vertical vectors and vectors in $\mathcal{H}T\mathbf M$ are called horizontal vectors.
\subsubsection{Decomposition of $T^\star(T\mathbf M\setminus \{0\})$}
In the same way one can show that dually the co-vectors
$$\delta y^i|_y=dy^i|_y+N_j^i(y)dx^j|_y$$
and
$$dx^i|_y$$
span canonical subspaces of $T_y^\star(T\mathbf M\setminus \{0\})$ 
$$\mathcal{V}_y^\star T\mathbf M=\spa\left\{\delta y^i|_y\right\},\quad \mathcal{H}_y^\star T\mathbf M=\spa\left\{dx^i|_y\right\}.$$
Moreover we have a canonical decomposition 
$$T^\star(T\mathbf M\setminus \{0\})=\mathcal{V}^\star T\mathbf M\oplus\mathcal{H}^\star T\mathbf M.$$
Co-vectors in $\mathcal{V}^\star T\mathbf M$ are called vertical co-vectors and those in $\mathcal{H}^\star T\mathbf M$ are called horizontal co-vectors.
\subsection{Finsler connections}
\begin{defi}[Finsler connection] A Finsler connection is determined by a triple $(N,H,C)$ where $N$ is a non-linear connection on $\mathbf M$ and $H=(H^i_{jk})$, $C=(C^i_{jk})$ are collections of locally defined $0$-homogeneous functions  $H^i_{jk}, C^i_{jk}:T\mathbf M\setminus \{0\}\to\R$ satisfying the transformation rules
\begin{equation}
\label{4004}
\dfrac{\partial \tilde{x}^l}{\partial x^i}H^i_{jk}=\dfrac{\partial^2\tilde{x}^l}{\partial x^j\partial x^k}+\dfrac{\partial \tilde{x}^r}{\partial x^j}\dfrac{\partial \tilde{x}^s}{\partial x^k}\tilde{H}^l_{rs},
\end{equation}
\begin{equation}
\label{4005}
C_{jk}^i=\dfrac{\partial x^i}{\partial \tilde{x}^p}\dfrac{\partial \tilde{x}^q}{\partial x^j}\dfrac{\partial \tilde{x}^r}{\partial x^k}\tilde{C}^p_{qr}.
\end{equation}
Let $\pi:T\mathbf M\ \mathbf M$ be the canonical projection. The Finsler connection induced by $(N,H,C)$ is the mapping 
$$\nabla:T_y(T\mathbf M\setminus\{0\})\times\mathcal{X}(T\mathbf M)\to T_{\pi(y)},\quad(Y,X)\mapsto\nabla_y(Y)$$
satisfying the following properties
\begin{enumerate}
\item $\nabla$ is linear over $\R$ in $X$ and $Y$.\\
\item If $f\in \mathscr{C}^\infty(M)$ and $y\in T\mathbf M\setminus\{0\}$, then in local coordinates
\begin{equation*}
\begin{split}
&\nabla_{\frac{\delta}{\delta x^i}|_y}\left(f\dfrac{\partial}{\partial x^j}|_x\right)=df\left(\dfrac{\partial}{\partial x^i}|_x\right)f\dfrac{\partial}{\partial x^j}|_x+fH^m_{ij}\dfrac{\partial}{\partial x^m}|_x\\
&\nabla_{\frac{\partial}{\partial y^i}|_y}\left(f\dfrac{\partial}{\partial x^j}|_x\right)=fC^m_{ij}\dfrac{\partial}{\partial x^m}|_x\end{split}
\end{equation*}
\end{enumerate}
\end{defi}
\begin{ex}
Let $(\mathbf M, L)$ be a Finsler space, and let 
\begin{equation*}\begin{split}\Gamma_{ijk}&=\dfrac{1}{2}\left(\dfrac{\delta g_{ik}}{\delta x^j}+\dfrac{\delta g_{ij}}{\delta x^k}-\dfrac{\delta g_{jk}}{\delta x^i}\right)\\
\Gamma^i_{jk}&=g^{ir}\Gamma_{ijk},
\end{split}
\end{equation*}
be locally defined functions. Then $(N_j^i,\Gamma_{jk}^i , 0)$ is the Chern-Rund connection.
\end{ex}
\begin{ex}
We recall that on a Finsler manifold the Cartan tensor $C_{ijk}$ is defined as 
\begin{equation*}
\begin{split}
C_{ijk}&=\dfrac{1}{2}\dfrac{\partial g_{ij}}{\partial y^k}=\dfrac{1}{4}\dfrac{\partial^3 L^2}{\partial y^i\partial y^j\partial y^k},\\
C_{jk}^l&=g^{li}C_{ijk}.
\end{split}
\end{equation*}
The Cartan connection is the Finsler connection $(N_j^i,\Gamma_{jk}^i , C_{jk}^i)$.
\end{ex}
\subsection{Covariant derivative and Curvature}
\begin{defi}[Covariant derivative]
Suppose $\mathbf  M$ is a manifold with a non-linear connection $N_j^i$. The the covariant derivative (induced by $N_j^i$) is a mapping
$$D:T_x\mathbf M\setminus\{0\}\times \mathcal{X}(\mathbf M)\to T_x\mathbf M\setminus\{0\},\quad (y,X)\mapsto D_y(X)$$
determined by the following properties
\begin{enumerate}
\item $D_y(X+Y)=D_y(X)+D_y(Y)$ for all $X,\,Y\in\mathcal{X}(\mathbf M)$ and $y\in T_\mathbf M\setminus\{0\}$,\\
\item $D_y(fX)=df(y)X+fD_y(X)$ for all $X\in \mathcal{X}(\mathbf M)$ and $f\in \mathcal{C}^\infty(\mathbf M)$,\\
\item in local coordinates , $D_y\left(\dfrac{\partial}{\partial x^i}|_x\right)=N^j_i(y)\dfrac{\partial}{\partial x^j}|_x$ for all $y\in T_x\mathbf M\setminus\{0\} $.
\end{enumerate}
\end{defi}
In Riemannian geometry, the Riemann curvature is a tensor on $M$. We next derive an analogous curvature tensor in a Finslerian setting, which will be a tensor on $T\mathbf M\setminus\{0\}$.
\begin{defi}[Curvature]
Let us set 
$$R^m_{jk}=\dfrac{\delta N_k^m}{\delta x^j}-\dfrac{\delta N_j^m}{\delta x^k}$$
and 
$$R^m_{ijk}=\dfrac{\partial R^m_{jk}}{\partial y^i}.$$
Then the curvature tensor $R$ is the $(1,3)$-tensor defined as 
$$R=R^m_{ijk}dx^i\otimes dx^j\otimes dx^k\otimes \dfrac{\delta}{\delta x^m}.$$
\end{defi}

\end{document}